\documentclass[11pt]{article}
\usepackage{amsmath,amsthm,amssymb}
\usepackage[usenames,dvipsnames]{xcolor}
\usepackage{enumerate}
\usepackage{graphicx}
\usepackage{cite}
\usepackage{comment}
\usepackage{oands}
\usepackage{tikz}
\usepackage{changepage}
\usepackage{bbm}
\usepackage{mathtools}
\usepackage[margin=1in]{geometry}

\usepackage[pagewise,mathlines]{lineno}
\usepackage{appendix}
\usepackage{stmaryrd} 
\usepackage{multicol}
\usepackage{microtype}
\usepackage[colorinlistoftodos,textsize=footnotesize]{todonotes}

\definecolor{dgreen}{RGB}{0,180,0}

\usepackage[pdftitle={The Tutte embedding of the mated-CRT map converges to Liouville quantum gravity},
  pdfauthor={Ewain Gwynne, Jason Miller, and Scott Sheffield},
colorlinks=true,linkcolor=NavyBlue,urlcolor=RoyalBlue,citecolor=PineGreen,bookmarks=true,bookmarksopen=true,bookmarksopenlevel=2,unicode=true,linktocpage]{hyperref}

\theoremstyle{plain}
\newtheorem{thm}{Theorem}[section]

\newtheorem{lem}[thm]{Lemma}
\newtheorem{prop}[thm]{Proposition}

\def\@rst #1 #2other{#1}
\newcommand\MR[1]{\relax\ifhmode\unskip\spacefactor3000 \space\fi
  \MRhref{\expandafter\@rst #1 other}{#1}}
\newcommand{\MRhref}[2]{\href{http://www.ams.org/mathscinet-getitem?mr=#1}{MR#2}}

\theoremstyle{definition}
\newtheorem{defn}[thm]{Definition}
\newtheorem{remark}[thm]{Remark}

\numberwithin{equation}{section}

\newcommand{\dsb}{\begin{adjustwidth}{2.5em}{0pt}
\begin{footnotesize}}
\newcommand{\dse}{\end{footnotesize}
\end{adjustwidth}}

\newcommand{\ssb}{\begin{adjustwidth}{2.5em}{0pt}}
\newcommand{\sse}{\end{adjustwidth}}

\newcommand{\aryb}{\begin{eqnarray*}}
\newcommand{\arye}{\end{eqnarray*}}
\def\alb#1\ale{\begin{align*}#1\end{align*}}
\def\allb#1\alle{\begin{align}#1\end{align}}
\newcommand{\eqb}{\begin{equation}}
\newcommand{\eqe}{\end{equation}}
\newcommand{\eqbn}{\begin{equation*}}
\newcommand{\eqen}{\end{equation*}}

\newcommand{\BB}{\mathbbm}
\newcommand{\ol}{\overline}
\newcommand{\ul}{\underline}
\newcommand{\op}{\operatorname}

\newcommand{\frk}{\mathfrak}
\newcommand{\eqD}{\overset{d}{=}}
\newcommand{\ep}{\epsilon}
\newcommand{\rta}{\rightarrow}

\newcommand{\wt}{\widetilde}
\newcommand{\wh}{\widehat} 
\newcommand{\mcl}{\mathcal}

\newcommand{\bdy}{\partial}
\newcommand{\rng}{\mathring}

\newcommand{\indshift}{\theta}

\newcommand{\SLE}{\mathrm{SLE}}

\let\originalleft\left
\let\originalright\right
\renewcommand{\left}{\mathopen{}\mathclose\bgroup\originalleft}
\renewcommand{\right}{\aftergroup\egroup\originalright}

\title{The Tutte embedding of the mated-CRT map converges to Liouville quantum gravity}

\date{  }
\author{
\begin{tabular}{c} Ewain Gwynne\\[-5pt]\small Cambridge \end{tabular}
\begin{tabular}{c} Jason Miller\\[-5pt]\small Cambridge \end{tabular}
\begin{tabular}{c} Scott Sheffield\\[-5pt]\small MIT \end{tabular}
}

\begin{document}

\maketitle

\begin{abstract}
We prove that the Tutte embeddings (a.k.a.\ harmonic/barycentric embeddings) of certain random planar maps converge to $\gamma$-Liouville quantum gravity ($\gamma$-LQG).  
Specifically, we treat mated-CRT maps, which are discretized matings of correlated continuum random trees, and $\gamma$ ranges from $0$ to $2$ as one varies the correlation parameter. We also show that the associated space-filling path on the embedded map converges to space-filling SLE$_{\kappa}$ for $\kappa =16/\gamma^2$ (in the annealed sense) and that simple random walk on the embedded map converges to Brownian motion (in the quenched sense).

This work constitutes the first proof that a discrete conformal embedding of a random planar map converges to LQG. Many more such statements have been conjectured. Since the mated-CRT map can be viewed as a coarse-grained approximation to other random planar maps (the UIPT, tree-weighted maps, bipolar-oriented maps, etc.), our results indicate a potential approach for proving that embeddings of these maps converge to LQG as well.

To prove the main result, we establish several (independently interesting) theorems about LQG surfaces decorated by space-filling SLE. There is a natural way to use the SLE curve to divide the plane into ``cells'' corresponding to vertices of the mated-CRT map.  We study the law of the {\em shape} of the origin-containing cell, in particular proving moments for the ratio of its squared diameter to its area. We also give bounds on the degree of the origin-containing cell and establish a form of ergodicity for the entire configuration. Ultimately, we use these properties to show (with the help of a general theorem proved in a separate paper) that random walk on these cells converges to a time change of Brownian motion, which in turn leads to the Tutte embedding result.
\end{abstract}


\tableofcontents

\section{Introduction}
\label{sec-intro}

\subsection{Overview and main results}
\label{sec-overview}

Since at least the 1980's, physicists have interpreted random planar maps as approximations of continuum random surfaces. We will not survey the (vast) physics literature on this topic here; see the reference overview in \cite{shef-kpz} or \cite{nakayama-lqg}.
 The continuum random surfaces we have in mind are known as {\bf Liouville quantum gravity} (LQG) surfaces, and their local behavior is characterized by a real parameter $\gamma \in (0,2]$ that in some sense determines the ``roughness'' of the surface.  Several precise scaling limit conjectures involving random planar maps and LQG surfaces were formulated in~\cite{shef-kpz,shef-zipper,wedges,hrv-disk,curien-glimpse}.
 
Moreover, there are now several senses in which random planar maps have been {\em proved} to converge to LQG surfaces: as metric spaces~\cite{legall-uniqueness,miermont-brownian-map,lqg-tbm1,lqg-tbm2,lqg-tbm3}, as path-decorated metric spaces \cite{gwynne-miller-saw,gwynne-miller-perc}, as mated pairs of trees~\cite{shef-burger,kmsw-bipolar,gkmw-burger,bhs-site-perc,lsw-schnyder-wood}, as collections of loops~\cite{bhs-site-perc}, etc. However, conspicuously absent from the results obtained thus far is a proof of any of the scaling limit conjectures involving {\em discrete conformal embeddings}.

One way to formulate such a conjecture is as follows. Imagine that one samples a random planar map with $n$ vertices (or $n$ faces) and applies a ``discrete conformal embedding'' that sends each vertex to a point in $\BB C$. The embedding induces a measure on $\BB C$ in which each embedded vertex is assigned mass $1/n$.  The conjecture is that as $n \to \infty$, the measure should converge in law to an LQG measure on the disk as constructed in~\cite{shef-kpz} (or~\cite{rhodes-vargas-log-kpz}). Furthermore, the simple random walk on the embedded map should converge in law to a two dimensional Brownian motion independent from the measure, modulo time parameterization. 

Attempts to formulate such a conjecture in a canonical way are complicated by the fact that there are actually several ``discrete conformal embeddings'' that can be claimed to be canonical in some sense.
\begin{enumerate}
\item {\bf Circle packings} (see the overview in \cite{stephenson-circle-packing}).
\item {\bf Square tilings} via the {\bf Smith diagram}~\cite{brooks-dissection} or Cannon's {\bf combinatorial Riemann mapping theorem}~\cite{cannon-riemann-mapping} or Schramm's {\bf combinatorial square tiling}~\cite{schramm-square-tiling}.
\item {\bf Riemann uniformizations} of the continuum surface obtained by viewing faces as unit-side-length polygons, glued together along edges.
\item {\bf Tutte embeddings} (a.k.a.\ {\bf harmonic embeddings} or {\bf barycentric embeddings}).
\end{enumerate}

There are also several kinds of random planar maps that are known (in at least {\em some} sense) to scale to $\gamma$-LQG  for $\gamma \in (0,2)$.

\begin{enumerate}
\item {\bf Uniform random triangulations} (or quadrangulations, or planar maps with other face sizes).
\item {\bf Randomly {\em decorated} planar maps} of various types, where the decoration may consist of a spanning tree, an Ising or Potts model instance, a self-avoiding walk, an FK-cluster decomposition, a bipolar orientation, a Schynder wood, a Gaussian free field, or an instance of some other statistical physics model.
\item {\bf Mated-CRT maps}, which are discretized matings of continuum random trees~\cite{aldous-crt1,aldous-crt2,aldous-crt3} and which we will define precisely below.
\end{enumerate}

Given $M$ distinct embedding procedures and $N$ distinct random planar map models, one can formulate $MN$ conjectures, each stating that the discrete conformal embedding of the map converges in law to some form of LQG. 
The aim of this paper is to completely solve {\em one} of these conjectures. 
Specifically, we treat the \emph{Tutte embedding} of the mated-CRT map. Mated-CRT maps are defined just below and illustrated in Figure~\ref{fig::map}. These maps are a natural place to start when proving embedding convergence results, for the following reasons.
\begin{enumerate}
\item They are among the easiest random planar map models to define and simulate.
\item They come with a parameter (the correlation of the trees) which one can vary to make the limiting surface $\gamma$-LQG for any $\gamma \in (0,2)$.
\item They each come \emph{a priori} with an embedding into $\BB C$ arising from the theory developed in~\cite{wedges,sphere-constructions,ag-disk} which allows us to reduce the convergence of the Tutte embedding to a quenched invariance principle for random walk in a certain highly inhomogeneous random environment (see Section~\ref{sec-approach} just below for further discussion of this). 
\item They can be understood as coarse-grained approximations of {\em any one} of the other models listed above.\footnote{Rougly speaking, this is because a sample from one of these models can be represented as a planar map obtained by ``gluing together'' two {\em discrete} random trees. The contour functions of the trees together form a two dimensional random walk --- the scaling limit of which is a two dimensional Brownian motion with some diffusion matrix, which in turn encodes a mated pair of continuum random trees. (This is what the authors have called {\em peanosphere convergence}~\cite{mullin-maps,shef-burger,kmsw-bipolar,gkmw-burger,lsw-schnyder-wood}). 
On both the discrete and continuum sides, the ``interface'' between the two trees is a sort of ``space-filling path'' and one can define a ``cell'' to be a region traversed by this path during a fixed interval of time. If one couples the discrete two dimensional walk with the continuum two dimensional Brownian motion, so that they are close with high probability, then one can show that discrete cells closely approximate their continuum analogs (in the sense that the adjacency relation on the set of discrete cells agrees with the continuum analog with high probability), which in turn correspond to a mated-CRT map. Such coarse-graining ideas are used in~\cite{ghs-map-dist,gm-spec-dim,gh-displacement,dg-lqg-dim} to transfer estimates for graph distances and random walk on mated-CRT maps to some of these other random planar map models.}  As we discuss in Section~\ref{sec-extensions}, it is possible that the results of this paper will lead to scaling limit theorems for the Tutte embeddings of these other planar map models as well.
\end{enumerate}

\subsubsection{Mated-CRT maps}

We now explain the construction of the mated-CRT map. Fix $\gamma \in (0,2)$ and consider a two dimensional Brownian motion $(L,R)$ with covariance matrix chosen so that 
\eqb \label{eqn-bm-cov}
\op{Var}(L_t) = \op{Var}(R_t) = |t| \quad \op{and} \quad  \op{Cov}(L_t,R_t) = -\cos\left(\frac{\pi\gamma^2}{4} \right) |t| , 
\eqe
 or some minor variant thereof (this correlation ranges from $-1$ to $1$ as $\gamma$ ranges from $0$ to $2$). For concreteness, we consider the setting shown in Figure~\ref{fig::map}, which is a Brownian motion as in~\eqref{eqn-bm-cov} started from $(0,0)$, run for one unit of time, and conditioned to stay in the positive quadrant and end up at $(1,0)$. Such a conditioned Brownian motion is defined via a limiting procedure in~\cite{sphere-constructions}, building on results of~\cite{shimura-cone}.
This Brownian excursion corresponds to a disk version of the mated-CRT map (mated-CRT maps with other topologies are defined in Section~\ref{sec-mated-crt-map}). 

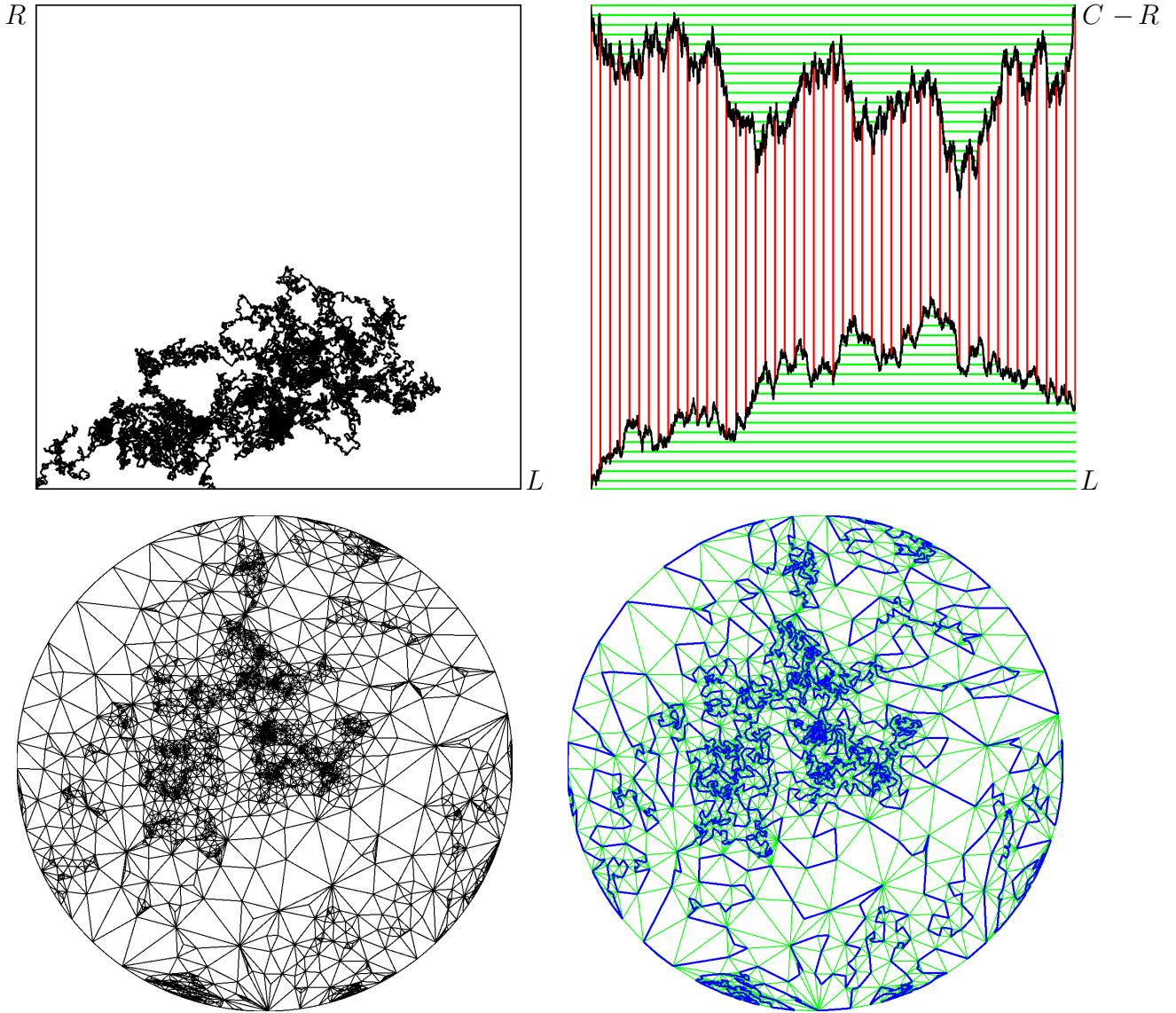
\begin{figure}
\begin{center}

\begin{tikzpicture}
\node (img) {\reflectbox{\includegraphics[width=0.44\textwidth,angle =90]{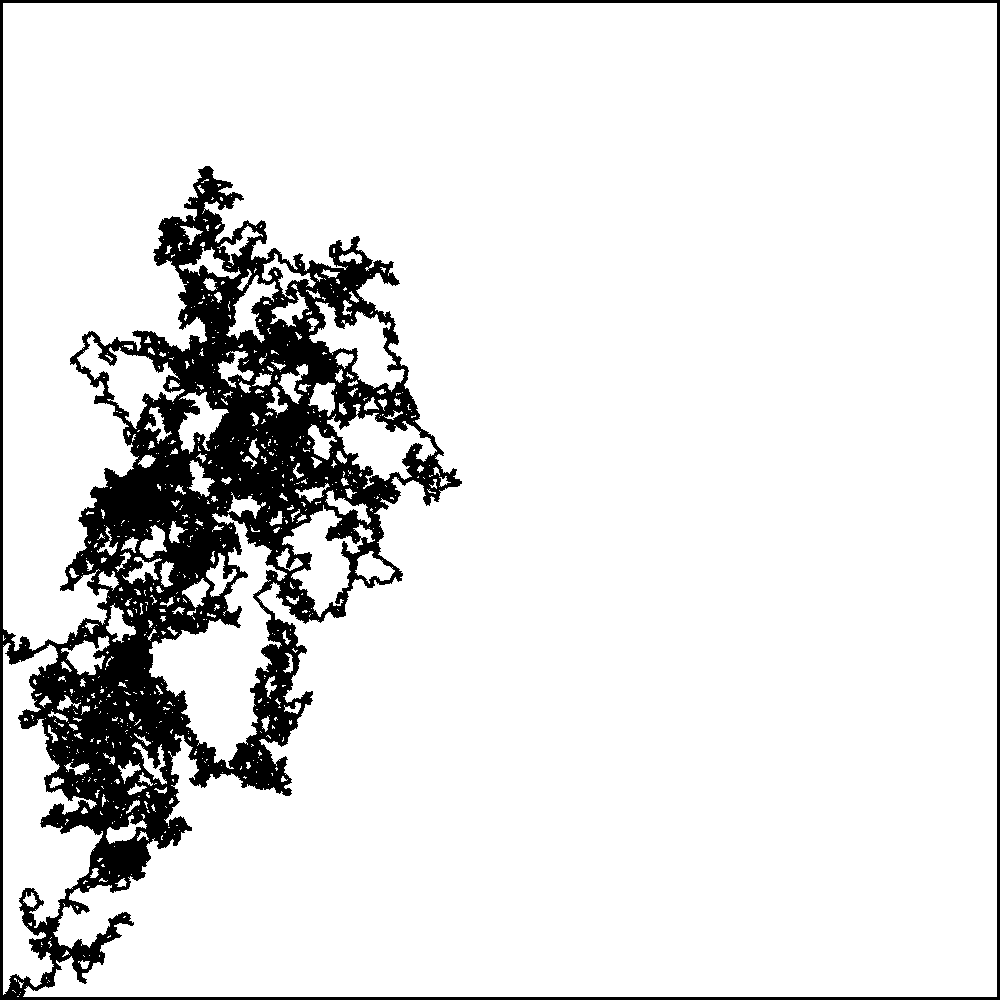}}};
\node[below left,text width=0.2cm,align=center] at (img.north west){{\large $R$}};
\node[above right,text width=0.3cm,align=left,xshift=-.2cm] at (img.south east){{\large $L$}};
\end{tikzpicture} \hspace{0.01\textwidth} 
\begin{tikzpicture}
\node (img) {\includegraphics[width=0.44\textwidth]{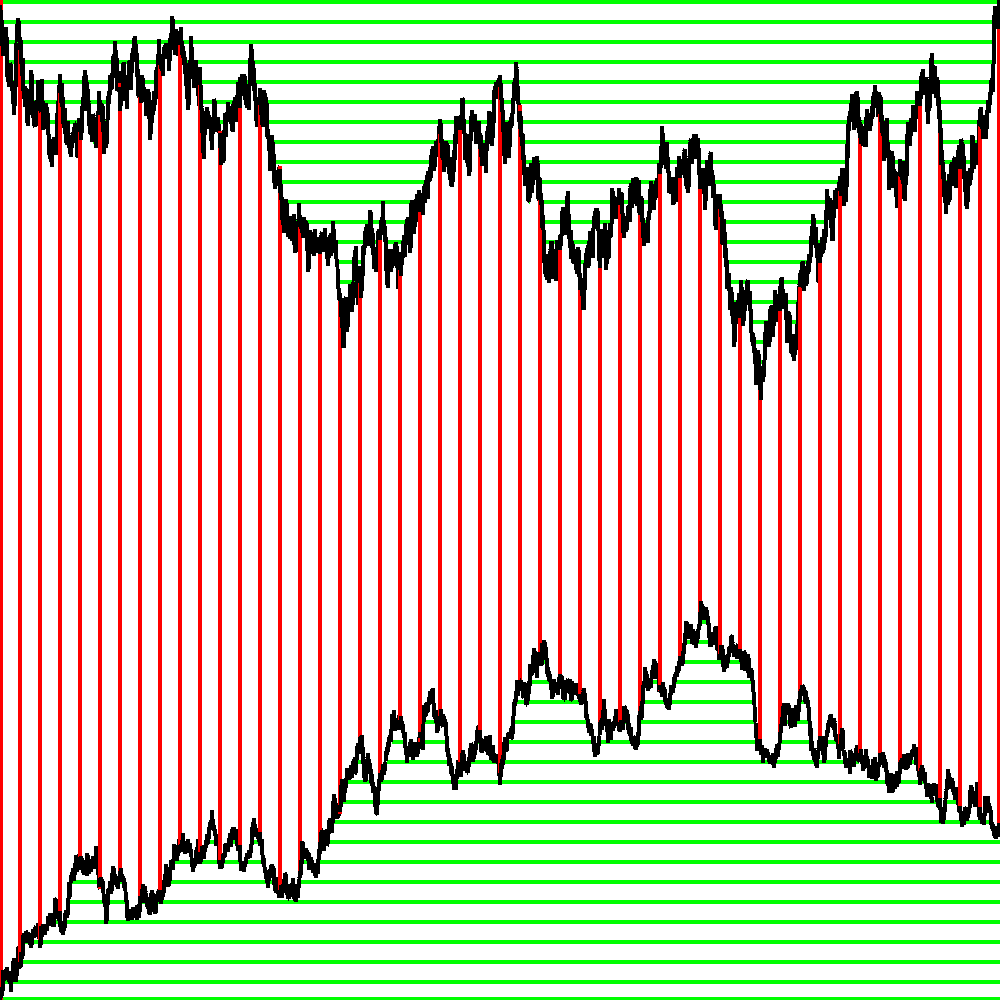}};
\node[below left,text width=-0.2cm,align=right,xshift=-.1cm] at (img.north east){{\large $C~-~\!R$}};
\node[above right,text width=-0.2cm,align=right,xshift=-.2cm] at (img.south east){{\large $L$}};
\end{tikzpicture}

\vspace{0.01\textheight}

\includegraphics[width=0.45\textwidth]{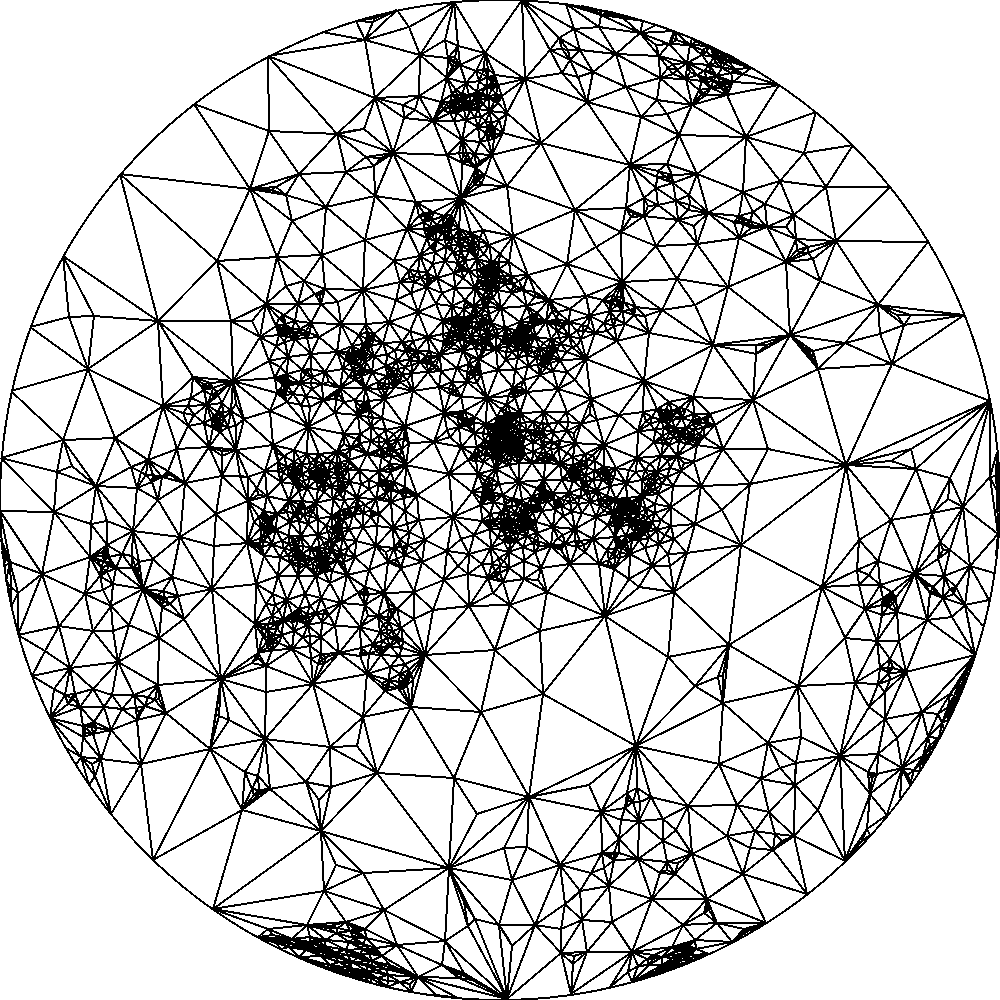}\hspace{0.05\textwidth}\includegraphics[width=0.45\textwidth]{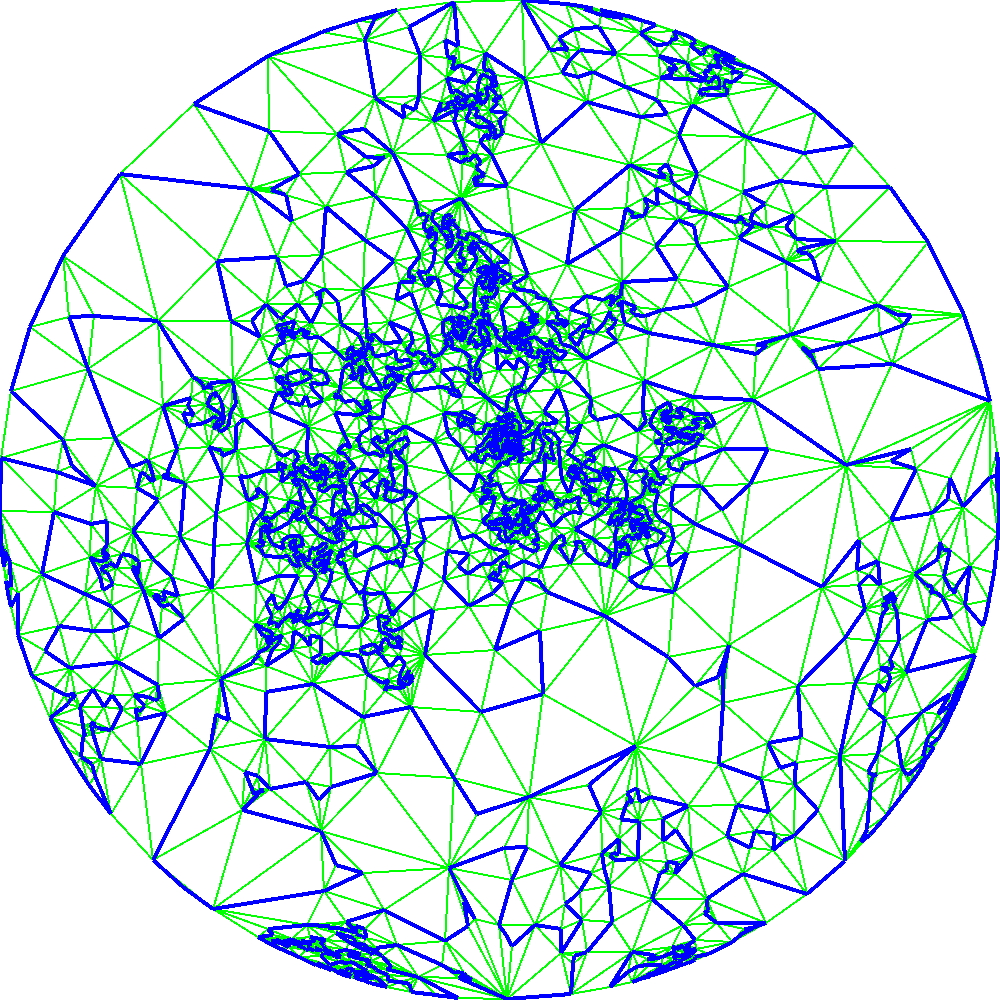}
\end{center}
\vspace{-0.03\textheight}
\caption{\label{fig::map} 
{\bf Top left:} A unit length Brownian excursion $(L,R)$ in ${\BB R}_+^2$ from $(0,0)$ to $(1,0)$. The correlation between $L$ and $R$ is given by $-\cos(\gamma^2\pi/4$); in this simulation we have taken $\gamma=\sqrt{2}$ (hence $\kappa=8$) so that $L,R$ are independent. 
{\bf Top right:} To define the mated-CRT map, we first choose a large constant $C>0$ so that the graphs of $L$ and $C-R$ do not intersect. Then, we divide the space between these two graphs into $n=20,\!000$ vertical strips (red). Each strip corresponds to a vertex of the mated-CRT map (which is identified with the horizontal coordinate of its right boundary), and two strips are connected by an edge if they share a boundary line or they are connected by a horizontal green chord below the graph of $L$ or above the graph of $C-R$, or a double edge if both of these conditions are satisfied; several such chords are shown in green. 
{\bf Bottom left:} The Tutte embedding of the mated-CRT map into the unit disk ${\BB D}$. Vertices (resp.\ edges) are represented by the small disks (resp.\ straight lines).
 {\bf Bottom right:} The space-filling path which corresponds to following the red chords from left to right.  Theorem~\ref{thm-tutte-conv0} says that the measure which assigns mass $1/n$ to each of the embedded vertices converges to the area measure associated with a $\gamma$-LQG surface, this space-filling path converges to a space-filling $\SLE_{\kappa}$ curve for $\kappa=16/\gamma^2$, and the simple random walk on the embedded map converges to Brownian motion (modulo time parameterization).  }
\end{figure}

For $n\in\BB N$, the \emph{mated-CRT map} $\mcl G^{1/n}$ (with the disk topology) with $n$ vertices is the graph with vertex set $ (\frac{1}{n} \BB Z) \cap (0,1]$, with two vertices $x_1,x_2 \in   (\frac{1}{n}\BB Z) \cap (0,1]$ connected by an edge if and only if either
\allb \label{eqn-inf-adjacency0}
&\left( \inf_{t\in [x_1- 1/n , x_1]} L_t \right) \vee \left(\inf_{t\in [x_2-1/n,x_2]} L_t \right) \leq \inf_{t\in [x_1,x_2-1/n]} L_t  \notag \\
&\qquad \op{or} \qquad \left( \inf_{t\in [x_1- 1/n , x_1]} R_t \right) \vee \left(\inf_{t\in [x_2-1/n,x_2]} R_t \right) \leq \inf_{t\in [x_1,x_2-1/n]} R_t  .
\alle
The vertices are connected by two edges if $|x_1-x_2|\not=1/n$ and~\eqref{eqn-inf-adjacency0} holds for both $L$ and $R$. 
The condition for $L$ in~\eqref{eqn-inf-adjacency0} is equivalent to the existence of a horizontal line segment below the graph of $L$ whose endpoints are of the form $(t_1,L_{t_1})$ and $(t_2,L_{t_2})$ for $t_1 \in [x_1-1/n,x_1]$ and $t_2 \in [x_2-1/n,x_2]$, and similarly for $R$. This allows us to give an equivalent, more geometric, version of the definition of $\mcl G^{1/n}$, as illustrated in the top-right panel of Figure~\ref{fig::map}. 
 
The {\em boundary} of $\mcl G^{1/n}$ is defined by 
\eqb \label{eqn-mated-crt-map-bdy}
\bdy\mcl G^{1/n} := \left\{ y \in  (\frac{1}{n} \BB Z) \cap (0,1] : \inf_{t \in [y-1/n,y]} L_t \leq \inf_{t\in [y,1]} L_t \right\} .
\eqe
In the setting of Figure~\ref{fig::map}, these vertices correspond to the vertical strips connected by a horizontal green line segment to a point on the right boundary of the figure.

The above defines the mated-CRT map as a \emph{graph}, but it is easy to see that this graph also comes with a natural \emph{planar map} structure, under which it is a triangulation; see Figure~\ref{fig-arc-graph}.

\begin{figure}[ht!]
\begin{center}
\includegraphics[scale=.8]{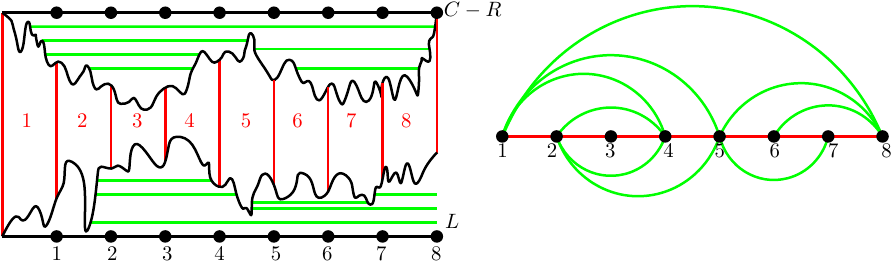} 
\caption{\label{fig-arc-graph} \textbf{Left:} A sketch of the construction of the mated-CRT map with $n=8$ vertices as described in Figure~\ref{fig::map}, with one representative green horizontal segment shown for each edge between non-consecutive vertices and one orange segment for each boundary edge (here time is scaled by $8$ to avoid writing fractions). \textbf{Right:} One can define a planar map structure on the mated-CRT map by identifying each vertex to the point of $(\frac{1}{n} \BB Z)\cap (0,1]$ equal to the right coordinate of the corresponding stripe, then connecting any two nearest-neighbor points of $(\frac{1}{n} \BB Z) \cap (0,1]$ by a line segment (red); and connecting any two non-nearest-neighbor points of $(\frac{1}{n} \BB Z)\cap (0,1]$ whose corresponding stripes are joined by a horizontal line lying below (resp.\ above) the graph of $L$ (resp.\ $C-R$) by an arc lying below (resp.\ above) the real line, shown in green. We note that the mated-CRT map is a triangulation with boundary with no self-loops (but double edges arising from pairs of vertices joined both above and below the real line) when equipped with this planar map structure. See~\cite[Figure 1]{gms-harmonic} for a more detailed explanation.
}
\end{center}
\end{figure}

\subsubsection{Tutte embedding}
 
To construct the Tutte embedding of the mated-CRT map $\mcl G^{1/n}$, we fix an interior vertex $\BB x^n$ and we enumerate the boundary vertices (in numerical order) as $\{y_1,\dots,y_k\}$. For $j = 1,\dots,k$ we let $\frk p(y_j)$ be the conditional probability given $\mcl G^{1/n}$ that a simple random walk started from $\BB x^n$ first hits the boundary at a vertex in the boundary arc $\{y_1,\dots,y_j\}$.
The boundary vertices $y_1, y_2, \ldots, y_k$ are then mapped in counterclockwise order around the complex unit circle via $y_j \mapsto e^{2\pi i \frk p(y_j)}$. This makes it so that the hitting probability of the random walk started from $\BB x^n$ approximates the uniform measure on the unit circle. One then maps the interior vertices of $\mcl G^{1/n}$ into the unit disk via the discrete harmonic extension of the boundary values. This is what we call the {\bf Tutte embedding} of the mated-CRT map centered at $\BB x^n$. Note that $\BB x^n$ is approximately mapped to the center of the disk in this embedding.\footnote{\label{footnote::computingembedding} One may compute $\frk p$ by first finding the function $\psi$ that is zero on $\{ y_1, y_2, \ldots, y_k \}$ and $1$ at $\BB x^n$ and discrete harmonic elsewhere, and then observing that the probability that a random walk from $\BB x^n$ first reaches  $\{ y_1, y_2, \ldots, y_k \}$ via a given edge is proportional to the gradient of $\psi$ along that edge. One can find $\psi$ using a sparse matrix solver, or using more ad hoc harmonic relaxation techniques (which may involve first getting an approximation to $\psi$ by considering a mated-CRT map with fewer vertices built from the same Brownian motion). The same is true for the next step: finding the harmonic extension of the map from $\{ y_1, y_2, \ldots, y_k \}$ to the circle.}

\begin{thm} \label{thm-tutte-conv0}
Fix $\gamma \in (0,2)$ and define the correlated Brownian excursion $(L,R)$ and the associated mated-CRT maps $\mcl G^{1/n}$ for $n\in\BB N$ as above. Also let $\BB t$ be sampled uniformly at random from $[0,1]$ (independently from $(L,R)$) and for $n\in\BB N$, let $\BB x^n \in   (\frac{1}{n} \BB Z) \cap (0,1]$ be the vertex of $\mcl G^{1/n}$ with $\BB t \in [\BB x^n - 1/n , \BB x^n]$.  
For $n\in\BB N$, let $\mu^n$ be the random measure on $\BB D$ which assigns mass $1/n$ to each of the $n$ points in the Tutte embedding of the mated-CRT map centered at $\BB x^n$. We have the following convergence in probability as $n \to \infty$. 
\begin{enumerate}
\item The measures $\mu^n$ converge weakly to a limiting $\gamma$-LQG measure $\mu$ on $\BB D$ that is a.s.\ determined by $(L,R)$. In particular, $\mu$ is the $\gamma$-LQG measure associated with a unit area, unit boundary length quantum disk, as defined in~\cite[Section 4.5]{wedges}. 
\item The space-filling path on the embedded mated-CRT map which comes from the left-right ordering of the vertices converges uniformly to space-filling $\op{SLE}_{\kappa}$ with $\kappa=16/\gamma^2$, parameterized by $\gamma$-LQG mass (see Section~\ref{sec-wpsf-prelim} for a review of the definition of this curve).
\item The conditional law given $\mcl G^{1/n}$ of the simple random walk on the embedded map started from $\BB x^n$ and stopped upon hitting $\bdy\mcl G^{1/n}$ converges to the law of Brownian motion started from 0 and stopped upon hitting $\bdy\BB D$, modulo time parameterization. 
\end{enumerate} 
\end{thm}

\begin{remark} There are only a few ways to draw random paths in a Euclidean domain that {\em provably} converge to SLE {\em in the Euclidean sense} (as parameterized or unparameterized paths). Aside from direct Loewner evolution discretizations or loop soup discretizations \cite{lupu-loop-soup-cle}, previous examples have applied to special values of $\kappa$, namely the $\kappa$ for which $\{\kappa, 16/\kappa \} \cap \{2,3,4,6 \} \not = \emptyset$. By contrast, Theorem~\ref{thm-tutte-conv0} spans the full range of $\kappa > 4$ (with boundaries corresponding to all $\kappa < 4$).

The construction is also computationally efficient: sampling a mated-CRT map with $n$ vertices requires $O(n)$ steps, and the embedding requires solving two sparse systems of linear equations (each with $n$ vertices; recall Footnote~\ref{footnote::computingembedding}) which can be implemented for millions of vertices with a sparse matrix package. 
\end{remark}

\begin{remark} \label{remark-lbm}
This work only shows the convergence of random walk on the embedded mated-CRT map to Brownian motion \emph{modulo time parametrization}. One expects that in fact the embedded walk converges in law w.r.t.\ the uniform topology to \emph{Liouville Brownian motion}~\cite{berestycki-lbm,grv-lbm}, the natural notion of Brownian motion on a $\gamma$-LQG surface in the continuum. It is obtained from ordinary Brownian motion by applying a time change that depends on the LQG surface.  
The convergence of random walk on the embedded mated-CRT map to Liouville Brownian motion was recently established in~\cite{bg-lbm}, building on the results of the present paper.
\end{remark}

\begin{figure}[ht!]
\begin{center}
\includegraphics[width=.75\textwidth]{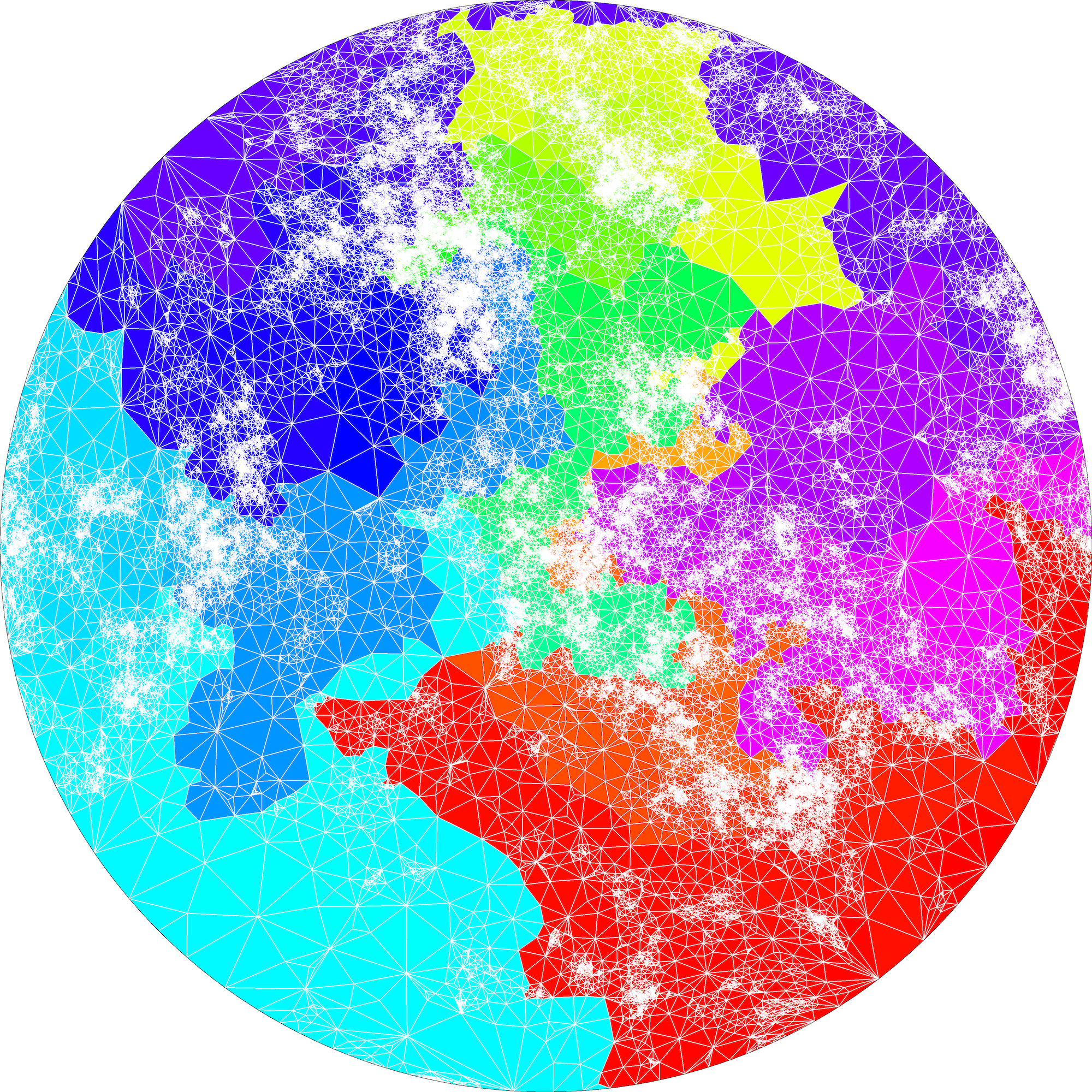} 

\vspace{0.015\textheight}

\includegraphics[width=0.4\textwidth, height=0.02\textheight]{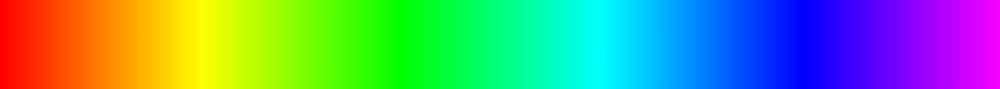}

Time ${\longrightarrow}$
\end{center}
\vspace{-0.03\textheight}
\caption{\label{fig::map2} Another instance of the mated-CRT map as in Figure~\ref{fig::map}, this time with $n=100,\!000$; edges are drawn in white.  Faces are colored according to the order in which they are hit by the associated space-filling path.  Time is parameterized by the fraction of vertices visited so far.  Theorem~\ref{thm-tutte-conv0} implies that this path converges to an $\SLE_{\kappa}$ loop as $n \to \infty$.  In this case, $\kappa=8$ since the coordinates of the Brownian excursion were taken to be independent.  Color changes are abrupt because most of the vertices of the embedded map are clustered together in small regions (which appear white in the figure due to the presence of many edges). The boundary of the region corresponding to a range of colors is a form of $\SLE_\kappa$, with $\kappa=16/\kappa$, which is $2$ in this case. Additional simulations, including ones for other values of $\gamma$ and $\kappa$, can be found at \url{http://statslab.cam.ac.uk/~jpm205/tutte-embeddings.html}. }
\end{figure}

\subsection{From mated-CRT maps to random walk in random environment}
\label{sec-approach}

The main feature that makes the mated-CRT map special is that, thanks to the results in~\cite{wedges,sphere-constructions,ag-disk}, this map comes {\em a priori} with a certain kind of embedding into $\BB C$ described by SLE-decorated LQG. To explain this, we note that the mated-CRT map construction as described in Figure~\ref{fig::map} has a continuous time (``$n=\infty$") variant: the disk version of the so-called \emph{peanosphere}. This is a topological measure space decorated by a space-filling curve $\eta$ obtained as the quotient of $[0,1] \times [0,C]$ when we identify each red line between the graphs of $ L$ and $C - R$ and each green line under the graph of $L$ or above the graph of $C-R$ to a point. 
The measure is the pushforward of Lebesgue measure on a horizontal line segment between the two graphs and the space-filling curve $\eta$ is the one obtained by tracing such a segment from left to right at unit speed. It is easy to see using Moore's theorem~\cite{moore} that this space has the topology of the disk (see, e.g., the argument in \cite[Section~1.3]{wedges} for the case of the sphere --- the disk case is similar).  We use the term {\bf peanosphere} informally to describe the space-filling-path-decorated topological measure space obtained this way. 

The peanosphere does not {\em a priori} come with a conformal structure. However, it is shown in~\cite{wedges,sphere-constructions,ag-disk} that an instance of the disk version of the peanosphere a.s.\ has a canonical (albeit non-explicit) conformal structure, and hence a canonical embedding into $\BB D$.  We will discuss this embedding and its relationship to the mated-CRT map in more detail in Section~\ref{sec-sg-def}, but for now let us give a brief overview.

Under its canonical embedding into $\BB D$, the measure on the peanosphere maps to a variant of the $\gamma$-LQG measure (in particular, the one associated with a quantum disk) and the curve $\eta$ maps to a space-filling form of $\SLE_{\kappa}$ with $\kappa = 16/\gamma^2$, parameterized so that it covers $t$ units of $\gamma$-LQG mass in $t$ units of time. 

Each vertex $x\in  (\frac{1}{n} \BB Z) \cap (0,1]$ of the mated-CRT map $\mcl G^{1/n}$ associated with $(L,R)$ corresponds to the \emph{cell} $\eta([x-1/n,x])$. 
Two vertices of $\mcl G^{1/n}$ are connected by an edge if and only if the intersection of the corresponding cells contains a non-trivial connected set (so that, e.g., intersections along a Cantor-like set do not count). This graph of cells is called the \textbf{$1/n$-structure graph} of $\eta$ (as in~\cite{ghs-dist-exponent}) because it encodes the topological structure of the cells; see Figure~\ref{fig-sg-def-disk} for an illustration.

\begin{remark} \label{remark-explicit}
The proof in~\cite{wedges,sphere-constructions,ag-disk} shows only that the above embedding of the peanosphere into $\BB C$ is a.s.\ equal to a deterministic functional of the correlated Brownian excursion $(L,R)$. It does \emph{not} give an explicit description of this functional. 
However, once Theorem~\ref{thm-tutte-conv0} is established we can make this functional explicit: it is simply the $n\to\infty$ limit of Tutte embeddings of the associated mated-CRT maps described in Figure~\ref{fig::map}.\footnote{To put this a different way, the reader may recall that the construction of $\SLE_4$ as a ``level line'' of the GFF as in~\cite{ss-contour} is {\em discretely approximated}. Precisely, a GFF level line is the limit of its discrete counterparts, which are level sets of projections of the GFF onto spaces of piecewise linear functions~\cite{ss-dgff}. By contrast, the construction of a general imaginary geometry flow line from the GFF in~\cite{ig1,ig2,ig3,ig4} is (conjectured but) not known to be discretely approximated in any similar way. Currently, the only way to construct the path from the field is indirect: one first produces a coupling of the GFF and the path, then shows that in this coupling the field a.s.\ determines the path. Theorem~\ref{thm-tutte-conv0} can be interpreted as saying that construction of SLE-decorated LQG from a correlated CRT pair is discretely approximated in a certain sense.}
\end{remark}

One can define an a priori embedding of the mated-CRT map by sending each vertex $x$ to the corresponding cell $\eta([x-1/n,x])$. Since $\eta$ is parameterized by $\gamma$-LQG mass, it is clear that under the \emph{a priori} embedding, the counting measure on vertices, scaled by $1/n$, converges to the $\gamma$-LQG measure and the discrete space-filling path obtained by following the vertices in left-right order converges to $\eta$. Hence to prove Theorem~\ref{thm-tutte-conv0} (modulo the random walk convergence statement), we only need to show that the a priori embedding is close to the Tutte embedding. Since the Tutte embedding is defined in terms of hitting probabilities for simple random walk on the map, Theorem~\ref{thm-tutte-conv0} will be a straightforward consequence of the following theorem.   

\begin{thm} \label{thm-quenched-clt0}
As $n\rta\infty$, the image of simple random walk on the mated-CRT map under the a priori embedding $x\mapsto \eta(x)$ (i.e., the simple random walk on the graph of cells $\eta([x-1/n,x])$) stopped when it hits $\bdy \mcl G^{1/n}$ converges in law to Brownian motion modulo time parameterization in the quenched sense, uniformly in the choice of starting point. 
\end{thm}

See Theorem~\ref{thm-rw-conv} for a more precise version of Theorem~\ref{thm-quenched-clt0}, which also treats the whole-plane and sphere cases.

\begin{figure}[ht!]
\begin{center}
\includegraphics[scale=.6]{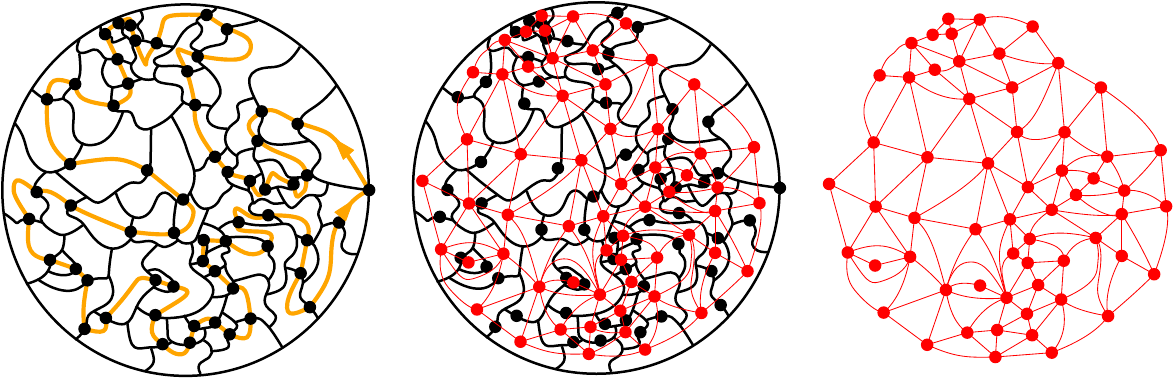} 
\caption{\label{fig-sg-def-disk} \textbf{Left:} A segment of a space-filling curve $\eta : [0,1] \rta \ol{\BB D}$, divided into cells $  \eta([x-1/n ,x])$ for $x\in (\frac{1}{n} \BB Z) \cap (0,1]$. The order in which the cells are hit is shown by the orange path. This figure looks like what we would expect to see for $\kappa \geq 8$ ($\gamma \leq \sqrt 2$), since the cells are simply connected (see Figure~\ref{fig-weird-cell}, left for an illustration of the case $\kappa \in (4,8)$). \textbf{Middle:} We draw a red point in each cell and connect the points whose cells share a corresponding boundary arc. The red graph is then an a priori embedding of the mated-CRT map into $\BB C$. Note that this graph is a triangulation, with faces corresponding to the points where three of the black curves meet. \textbf{Right:} Same as the middle picture but without the original cells, so that only the a priori embedding of the mated-CRT map is visible. 
We will prove that the simple random walk on the graph of cells converges to Brownian motion modulo time parameterization, and hence the a priori embedding is close to the Tutte embedding when the number of vertices is large. 
}
\end{center}
\end{figure}

\begin{figure}[t!]
 \begin{center}
\includegraphics[scale=.85]{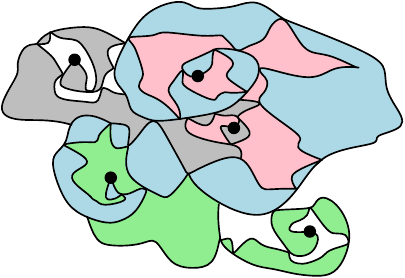}
\vspace{-0.01\textheight}
\caption{Four typical space-filling SLE$_\kappa$ cells for $\kappa\in (4,8)$. The picture is slightly misleading since the set of ``pinch points" where the left and right boundaries of each cell meet is actually uncountable, with no isolated points, but has Hausdorff dimension less than 2. The points where $\eta$ starts and finishes filling in each cell are shown with black dots. Note that the grey and green cells intersect at several points, but do not share a connected boundary arc so are \emph{not} considered to be adjacent. This is natural since one can think of the blue cell as lying in between the grey and green cells. In fact, two cells which intersect, but do not share a connected boundary arc, will always be separated by another cell in this manner. The case when $\kappa \geq 8$ is much simpler: in this setting, the interiors of the cells are simply connected and two cells are connected by an edge if and only if they intersect. 
}\label{fig-weird-cell}
\end{center}
\vspace{-1em}
\end{figure} 

 At first glance, Theorem~\ref{thm-quenched-clt0} would appear something like a conventional random walk in random environment (RWRE) problem: one has a random walk (jumping from cell to cell) and wants to show it approximates Brownian motion. There is a vast literature on the topic of RWRE; see~\cite{bf-rwre-survey,biskup-rwre-survey} for recent surveys. However, our situation does not fit into the conventional RWRE framework. Indeed, due to the fractal nature of the $\gamma$-LQG measure the cell sizes vary dramatically from one location to another, so the environment is highly spatially inhomogeneous. As a consequence of this, one should not expect the random walk on the adjacency graph of cells to converge uniformly to standard Brownian motion---instead, the walk to converge to a Brownian motion with a \emph{random} time parameterization (the so-called \emph{Liouville Brownian motion}; see Remark~\ref{remark-lbm}). Moreover, spatially translating the environment of cells corresponding to the whole-plane mated-CRT map changes the normalization of the field associated with the LQG surface, so the environment is not stationary with respect to spatial translations. 
 
Nevertheless, it follows from basic properties of the LQG measure that our environment is in some sense approximately translation invariant \emph{modulo scaling}, i.e., the collections of cells which intersect small neighborhoods of two different points $z,w \in \BB D$ agree in law (approximately) if we re-scale them both in the same way, e.g., so that the cells containing $z$ and $w$ each have unit Lebesgue measure. In fact, as we will explain below, it is possible to define an infinite-volume graph of cells which locally looks like the disk version considered above and which satisfies a certain exact notion of translation invariance modulo scaling. 

In order to prove Theorem~\ref{thm-quenched-clt0}, we will employ a general quenched invariance principle for random walk in two-dimensional random environments which are only required to be ``translation invariant modulo scaling" in the above sense (Theorem~\ref{thm-general-clt-uniform} just below), which was proven in~\cite{gms-random-walk}. In order to show that the hypotheses of the theorem are satisfied in our setting, we will (among other things) need to establish certain moment estimates for space-filling SLE cells which are of independent interest (for example, these moment bounds are also used in~\cite{gms-harmonic,gm-spec-dim}). 

\subsection{Scaling limit for random walk in random environment}
\label{sec-clt-thm} 

In this subsection we state a version of the main result of~\cite{gms-random-walk} which gives conditions under which random walk on the adjacency graph of a random collection of cells (such as the one discussed in the preceding subsection) converges to Brownian motion. As in~\cite{gms-random-walk}, we will work in the whole-plane rather than the disk. Let us first describe what we mean by an ``adjacency graph of cells". 

\begin{defn} \label{def-cell-config}
A \emph{cell configuration} on $\BB C$ consists of the following objects.
\begin{enumerate}
\item A locally finite collection $\mcl H$ of compact connected subsets of $\BB C$ (``cells") with non-empty interiors whose union is all of $\BB C$ and such that the intersection of any two elements of $\mcl H$ has zero Lebesgue measure.
\item A symmetric relation $\sim$ on $\mcl H\times \mcl H$ (``adjacency") such that if $H\sim H'$, then $H\cap H'\not=\emptyset$ and $H\not=H'$. 
\item A function $\frk c = \frk c_{\mcl H}$ (``conductance") from the set of pairs $(H,H')$ with $H\sim H'$ to $(0,\infty)$ such that $\frk c(H,H') = \frk c(H',H)$. 
\end{enumerate}
\end{defn}
 
We will typically slightly abuse notation by making the relation $\sim$ and the function $\frk c$ implicit, so we write $\mcl H$ instead of $(\mcl H,\sim , \frk c)$.  
We view $\mcl H$ as a weighted graph whose vertices are the cells of $\mcl H$ and whose edge set is
\eqb \label{eqn-cell-edges}
\mcl E\mcl H := \left\{ \{H,H'\} \in \mcl H\times \mcl H : H\sim H' \right\} ,
\eqe
and such that any edge $\{H,H'\}$ is assigned the weight $\frk c(H,H')$. 
Note that any two cells which are joined by an edge intersect but two intersecting cells need not be joined by an edge. 
We emphasize that the cells of $\mcl H$ are \emph{not} required to be simply connected. In particular, space-filling SLE$_{\kappa}$ type cells for $\kappa \in (4,8)$, as illustrated in Figure~\ref{fig-weird-cell}, are allowed.

For a cell configuration $\mcl H$ and $z\in \BB C$, we write $H_z$ for one of the cells in $\mcl H$ containing $z$, chosen in some arbitrary manner if there is more than one such cell (the cell is unique for Lebesgue-a.e.\ $z$). 
Following~\cite{gms-random-walk}, we define
\eqb \label{eqn-cell-restrict}
 \mcl H(A) := \left\{H\in\mcl H : H\cap A \not=\emptyset \right\}   ,\quad \forall A\subset\BB C
\eqe 
and view $\mcl H(A)$ as an edge-weighted graph with edge set consisting of all of the edges in $\mcl E\mcl H$ joining elements of $\mcl H(A)$ and conductances given by the restriction of $\frk c$. We also define a metric on the space of cell configurations by
\allb \label{eqn-cell-metric}
\BB d^{\op{CC}}(\mcl H,\mcl H') &:= \int_0^\infty e^{-r} \wedge \inf_{f_r} \big\{ \sup_{z\in \BB C} |z - f_r(z)|   \notag \\
&\qquad \qquad   + \max_{\{H_1,H_2\} \in \mcl E\mcl H(B_r(0))} |\frk c(H_1,H_2) - \frk c'(f_r(H_1) , f_r(H_2) ) | \big\}  \,dr
\alle
where each of the infima is over all homeomorphisms $f_r : \BB C\rta \BB C$ such that $f_r$ takes each cell in $\mcl H(B_r(0))$ to a cell in $\mcl H'(B_r(0))$ and preserves the adjacency relation, and $f_r^{-1}$ does the same with $\mcl H$ and $\mcl H'$ reversed.  We will typically be working with a random cell configuration, i.e., a random variable taking values in the space of cell configurations equipped with the Borel $\sigma$-algebra generated by the above metric.

We also define the primal and dual stationary measures on $\mcl H$ by
\eqb \label{eqn-stationary-measure}
\pi(H) := \sum_{\substack{H'\in\mcl H : H'\sim H}} \frk c(H,H') \quad \op{and} \quad
\pi^*(H) := \sum_{\substack{H'\in\mcl H : H'\sim H}} \frac{1}{ \frk c(H,H') }. 
\eqe

In~\cite{gms-random-walk}, we proved that the random walk on a random cell configuration $\mcl H$ with conductances $\frk c$ which satisfies the following hypotheses converges to Brownian motion.
Here, for $C>0$ and $z\in\BB C$ we write $C(\mcl H-z)$ for the cell configuration obtained by translating all of the cells by $-z$ then scaling all of the cells by $C$.  
\begin{enumerate}
\item \textbf{Translation invariance modulo scaling.} There is a (possibly random and $\mcl H$-dependent) increasing sequence of open sets $U_j \subset \BB C$, each of which is either a square or a disk, whose union is all of $\BB C$ such that the following is true. Conditional on $\mcl H$ and $U_j$, let $z_j$ for $j\in\BB N$ be sampled uniformly from Lebesgue measure on $U_j$. Then the shifted cell configurations $\mcl H - z_j$ converge in law to $\mcl H$ modulo scaling as $j\rta\infty$, i.e., there are random numbers $C_j >0$ (possibly depending on $\mcl H$ and $z_j$) such that $ C_j(\mcl H-z_j) \rta \mcl H$ in law with respect to the metric~\eqref{eqn-cell-metric}. \label{item-hyp-resampling} 
\item \textbf{Ergodicity modulo scaling.} Every real-valued measurable functions $F = F(\mcl H)$ which is invariant under translation and scaling, i.e., $F(C(\mcl H-z)) = F(\mcl H)$ for each $z\in\BB C$ and $C>0$, is a.s.\ equal to a deterministic constant. \label{item-hyp-ergodic} 
\item \textbf{Finite expectation.} With $H_0$ the cell containing 0 and $\pi$ and $\pi^*$ as in~\eqref{eqn-stationary-measure},  \label{item-hyp-moment}
\eqb\label{eqn-hyp-moment}
\BB E\left[ \frac{\op{diam}(H_0)^2}{\op{area}(H_0)} \pi(H_0)     \right] <\infty  \quad\op{and} \quad 
\BB E\left[ \frac{\op{diam}(H_0)^2}{\op{area}(H_0)} \pi^*(H_0)     \right] <\infty 
\eqe 
where $\op{diam}$ and $\op{area}$ denote Euclidean diameter and Lebesgue measure, respectively. 
\item \textbf{Connectedness along lines.} Almost surely, for each horizontal or vertical  line segment $L \subset \BB C$, the subgraph of $\mcl H$ induced by the set of cells which intersect $L$ is connected. \label{item-hyp-adjacency}
\end{enumerate}

We note that the combination of hypotheses~\ref{item-hyp-resampling} and~\ref{item-hyp-ergodic} is referred to as \emph{ergodicity modulo scaling} in~\cite{gms-random-walk}. 
Several equivalent formulations of hypothesis~\ref{item-hyp-resampling} are given in~\cite[Definition 1.2]{gms-random-walk}, but we state only the formulation which we will use in the present paper. 
The following is~\cite[Theorem 3.10]{gms-random-walk}. 

\begin{thm}[\cite{gms-random-walk}] \label{thm-general-clt-uniform}
Let $\mcl H$ be a random cell configuration satisfying the above four hypotheses. 
For $z\in\BB C$, let $Y^z$ denote the random walk on $\mcl H$ started from $H_z$ (with conductances $\frk c$). For $j\in\BB N_0$, let $\wh Y_j^z$ be an arbitrarily chosen point of the cell $Y_j^z$ and extend $\wh Y^z$ from $\BB N_0$ to $[0,\infty)$ by piecewise linear interpolation. 
There is a deterministic covariance matrix $\Sigma$ with $\det\Sigma\not=0$ such that the following is true. For each fixed compact set $A \subset \BB C$, it is a.s.\ the case that as $\ep \rta 0$, the maximum over all $z\in A$ of the Prokhorov distance between the conditional law of $\ep \wh Y^{z/\ep}$ given $\mcl H$ and the law of Brownian motion started from $z$ with covariance matrix $\Sigma$, with respect to the topology on curves modulo time parameterization (as defined in Section~\ref{sec-cmp-metric}), tends to 0. 
\end{thm}

We emphasize that the random walk in Theorem~\ref{thm-general-clt-uniform} converges to Brownian motion \emph{uniformly} on compact subsets of $\BB C$. This will be important for us since it implies the convergence of certain discrete harmonic functions (e.g., the coordinates of the Tutte embedding functions) to their continuum counterparts.

\begin{figure}[ht!]
\begin{center}
\includegraphics[scale=.8]{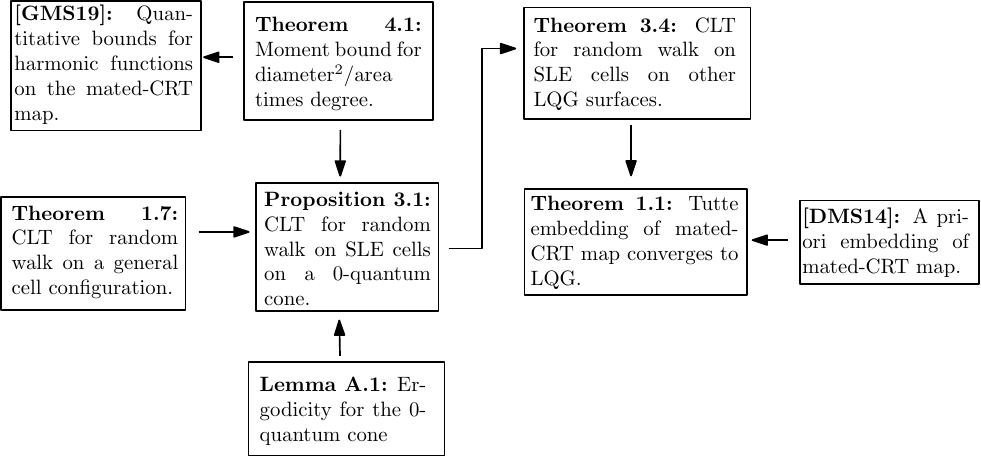} 
\caption{\label{fig-outline}
Schematic outline of the proof of Theorem~\ref{thm-tutte-conv0}. 
}
\end{center}
\end{figure}

Figure~\ref{fig-outline} illustrates how the various statements in this paper fit together to yield Theorem~\ref{thm-tutte-conv0}. 
In order to extract Theorem~\ref{thm-quenched-clt0}, and thereby Theorem~\ref{thm-tutte-conv0}, from Theorem~\ref{thm-general-clt-uniform}, we will first consider a variant of the adjacency graph of cells discussed in Section~\ref{sec-approach} which satisfies the hypotheses of Theorem~\ref{thm-general-clt-uniform}. In particular, we will look at the adjacency graph of cells associated with a whole-plane space-filling SLE$_\kappa$ curve parameterized by $\gamma$-LQG length with respect to an independent 0-quantum cone. The 0-quantum cone is a $\gamma$-LQG surface with the topology of the whole plane which describes the local behavior of a general $\gamma$-LQG surface near a Lebesgue typical point. See Section~\ref{sec-lqg-prelim} for more on the 0-quantum cone. 

We will use Theorem~\ref{thm-general-clt-uniform} to show that random walk on the above adjacency graph of cells converges to Brownian motion modulo time parameterization. To do this, we need to check the four hypotheses of the theorem (Proposition~\ref{prop-0cone}). We note that the limiting covariance matrix $\Sigma$ will be a positive scalar multiple of the identity matrix in our setting due to the rotational invariance of the law of the 0-quantum cone. 
As we will see, hypotheses~\ref{item-hyp-resampling},~\ref{item-hyp-ergodic} and~\ref{item-hyp-adjacency} follow relatively easily from the definitions of the 0-quantum cone and the associated graph of cells. The fact that the origin is ``Lebesgue typical" is what allows us to check hypothesis~\ref{item-hyp-resampling}. To check hypothesis~\ref{item-hyp-ergodic}, we will need a basic ergodicity property for the 0-quantum cone, which we prove in Appendix~\ref{sec-ergodicity}.

The verification of the finite expectation condition given in hypothesis~\ref{item-hyp-moment} will require more work. We first note that the hypothesis is equivalent to $\BB E[\op{deg}(H_0) \op{diam}(H_0)^2/\op{area}(H_0) ]  < \infty$, where $\op{deg}$ denotes the degree, since the conductances on our cell configuration will take values in $\{1,2\}$ (conductance 2 correponds to double edges). The proof of this estimate is based on estimates for SLE and LQG and is given in Section~\ref{sec-moment}. We will in fact prove a slightly stronger statement (giving moments up to order $4/\gamma^2$) which also has other applications. For example, our moment bound is used in~\cite{gms-harmonic} to prove several quantitative estimates for harmonic functions on the mated-CRT map, which are then transferred to other random planar maps --- like spanning-tree weighted maps and the UIPT --- in~\cite{gm-spec-dim}.

Once we have checked the hypotheses of Theorem~\ref{thm-general-clt-uniform} in the case of space-filling SLE$_\kappa$ cells on a 0-quantum cone, the theorem will tell us that random walk on the adjacency graph of such cells converges to Brownian motion. Using local absolute continuity, we will then argue that the same is true with any other $\gamma$-LQG surface in place of a 0-quantum cone (see Theorem~\ref{thm-rw-conv}). This will in particular imply Theorem~\ref{thm-quenched-clt0}.

\subsection{Related works} 
\label{sec-related-works}

Several celebrated papers have recently established qualitative results about random walk and discrete conformal embeddings that apply to general classes of infinite random planar maps, sometimes including the infinite volume version of the mated-CRT map.  See, for example,~\cite{benjamini-schramm-topology,benjamini-curien-uipq-walk,gn-recurrence,lee-conformal-growth,lee-uniformizing,gill-rohde-type,abgn-bdy,georg-poisson-bdy,angel-hyperbolic,ahnr-hyperbolic,chn-causal,cg-liouville,gjn-macroscopic-circles}. 
These papers have explored recurrence versus transience, parabolicity versus hyperbolicity, Tutte embeddings, Martin boundaries, heat kernel bounds, circle packings, discrete harmonic functions and other related topics. 
Some of these works have specific consequences in our setting; for example, the main result of~\cite{gn-recurrence} directly implies recurrence for infinite volume mated-CRT maps (see~\cite[Section 2.2]{gms-harmonic}). However these papers do not treat fine mesh scaling limits as we do here.

The papers~\cite{ghs-dist-exponent,ghs-map-dist,gms-harmonic,gm-spec-dim,gh-displacement,dg-lqg-dim,gp-dla,gjn-macroscopic-circles} also study mated-CRT maps (\cite{ghs-dist-exponent} uses the term ``structure graph" instead of ``mated-CRT map"). The papers~\cite{ghs-dist-exponent,ghs-map-dist,dg-lqg-dim,gp-dla} prove estimates for graph distances in mated-CRT maps, which are then transferred to other random planar maps (e.g., uniform maps and/or spanning tree-decorated maps) using a strong coupling of random walk and Brownian motion~\cite{zaitsev-kmt}. 
The work~\cite{gms-harmonic} proves quantitative bounds for the Dirichlet energy and modulus of continuity of harmonic functions on the mated-CRT map, using the moment bound of Theorem~\ref{thm-moment} in the present paper as a starting point. The results of~\cite{gms-harmonic} are used in~\cite{gm-spec-dim} to prove a lower bound for the return probability of the random walk on the mated-CRT map to its starting point and an upper bound for the graph distance displacement of the walk after $n$ steps. These results are then transferred to other random planar maps (such as the UIPT) using strong coupling. The paper~\cite{gh-displacement} proves an upper bound for the graph distance displacement of the random walk on the mated-CRT map and on other random planar maps.
The paper~\cite{gjn-macroscopic-circles} shows that there are no macroscopic circles in the circle packing of the mated-CRT map, building on the estimates of~\cite{gms-harmonic}. 
 
Biskup, Ding, and Goswami~\cite{bdg-lqg-rw} study random walk in a different environment which is also a natural discretization of LQG, obtained by weighting the edges of $\BB Z^2$ by the exponential of $\gamma$ times a discrete GFF, and prove various estimates for this walk.

The article~\cite{gms-poisson-voronoi} considers simple random walk on the Voronoi tessellation associated with a Poisson point process on a Brownian surface and shows that it converges (modulo time parameterization) to Brownian motion.

Finally, we mention that there is an extensive ongoing project involving papers by the present authors as well as Holden, Sun, and others which aims to prove that uniformly random triangulations under the so-called \emph{Cardy embedding} converge in the scaling limit to $\sqrt{8/3}$-LQG. The Cardy embedding (named for Cardy's formula for percolation~\cite{smirnov-cardy}) is a form of discrete conformal embedding that is defined using crossing probabilities for percolation interfaces (rather than hitting probabilities for random walk like in the Tutte embedding). Papers used in the proof of the Cardy embedding convergence include~\cite{gwynne-miller-char,gwynne-miller-perc,bhs-site-perc,ghss-negative-moments,hlls-cut-pts,hls-sle6,bhs-site-perc} as well as several works in preparation. 

Unlike the present paper, the convergence proof for the Cardy embedding makes heavy use of results about the metric space structure of uniform random planar maps and $\sqrt{8/3}$-LQG so only works for $\gamma=\sqrt{8/3}$.

\bigskip

\noindent\textbf{Acknowledgements.} 
We thank an anonymous referee for helpful comments on an earlier version of this paper.
We thank the Mathematical Research Institute of Oberwolfach for its hospitality during a workshop where part of this work was completed.  E.G.\ was partially funded by NSF grant DMS-1209044.  S.S.\ was partially supported by NSF grants DMS-1712862 and DMS-1209044 and a Simons Fellowship with award number 306120.

\section{Preliminaries}
\label{sec-prelim}

\subsection{Basic definitions}
\label{sec-basic}

\subsubsection{Basic notation} 
\label{sec-basic-notation}

\noindent
We write $\BB N$ for the set of positive integers and $\BB N_0 = \BB N\cup \{0\}$. 
\vspace{6pt}

\noindent
For $a,b \in \BB R$ with $a<b$, we define the discrete interval $[a,b]_{ \BB Z} := [a, b]\cap \BB Z $.
\vspace{6pt}
   
\noindent
If $a$ and $b$ are two ``quantities'' (i.e., functions from any sort of ``configuration space'' to the real numbers) we write $a\preceq b$ (resp.\ $a \succeq b$) if there is a constant $C > 0$ (independent of the values of $a$ or $b$ and certain other parameters of interest) such that $a \leq C b$ (resp.\ $a \geq C b$). We write $a \asymp b$ if $a\preceq b$ and $a \succeq b$. We typically describe dependence of implicit constants in lemma/proposition statements and require constants in the proof to satisfy the same dependencies.
\vspace{6pt}

\noindent
If $a$ and $b$ are two quantities depending on a variable $x$, we write $a = O_x(b)$ (resp.\ $a = o_x(b)$) if $a/b$ remains bounded (resp.\ tends to 0) as $x\rta 0$ or as $x\rta\infty$ (the regime we are considering will be clear from the context). 
\vspace{6pt}
 
\noindent
We write $a = o_x^\infty(b)$ if $a = o_x(b^p)$ for every $p\in\BB R$. 
\vspace{6pt}

\noindent
For a graph $G$, we write $\mcl V(G)$ and $\mcl E(G)$, respectively, for the set of vertices and edges of $G$, respectively. We sometimes omit the parentheses and write $\mcl VG = \mcl V(G)$ and $\mcl EG = \mcl E(G)$. For $v\in\mcl V(G)$, we write $\op{deg}(v )$ for the degree of $v$ (i.e., the number of edges with $v$ as an endpoint). 
\vspace{6pt}

\subsubsection{Metric on curves modulo time parameterization}
\label{sec-cmp-metric}

 If $\beta_1 : [0,T_{\beta_1}] \rta \BB C$ and $\beta_2 : [0,T_{\beta_2}] \rta \BB C$ are continuous curves defined on possibly different time intervals, we set 
\eqb \label{eqn-cmp-metric}
\BB d^{\op{CMP}} \left( \beta_1,\beta_2 \right) := \inf_{\phi } \sup_{t\in [0,T_{\beta_1} ]} \left| \beta_1(t) - \beta_2(\phi(t)) \right| 
\eqe 
where the infimum is over all increasing homeomorphisms $\phi : [0,T_{\beta_1}]  \rta [0,T_{\beta_2}]$ (the CMP stands for ``curves modulo parameterization"). It is shown in~\cite[Lemma~2.1]{ab-random-curves} that $\BB d^{\op{CMP}}$ induces a complete metric on the set of curves viewed modulo time parameterization. 

In the case of curves defined for infinite time, it is convenient to have a local variant of the metric $\BB d^{\op{CMP}}$. Suppose $\beta_1 : [0,\infty) \rta \BB C$ and $\beta_2 : [0,\infty) \rta \BB C$ are two such curves. For $r > 0$, let $T_{1,r}$ (resp.\ $T_{2,r}$) be the first exit time of $\beta_1$ (resp.\ $\beta_2$) from the ball $B_r(0)$ (or 0 if the curve starts outside $B_r(0)$). 
We define 
\eqb \label{eqn-cmp-metric-loc}
\BB d^{\op{CMP}}_{\op{loc}} \left( \beta_1,\beta_2 \right) := \int_1^\infty e^{-r} \left( 1 \wedge \BB d^{\op{CMP}}\left(\beta_1|_{[0,T_{1,r}]} , \beta_2|_{[0,T_{2,r}]} \right) \right) \, dr ,
\eqe 
so that $\BB d^{\op{CMP}}_{\op{loc}} (\beta^n , \beta) \rta 0$ if and only if for Lebesgue a.e.\ $r > 0$, $\beta^n$ stopped at its first exit time from $B_r(0)$ converges to $\beta$ stopped at its first exit time from $B_r(0)$ with respect to the metric~\eqref{eqn-cmp-metric}. 
We note that the definition~\eqref{eqn-cmp-metric} of $\BB d^{\op{CMP}}\left(\beta_1|_{[0,T_{1,r}]} , \beta_2|_{[0,T_{2,r}]} \right)$ makes sense even if one or both of $T_{1,r}$ or $T_{2,r}$ is infinite, provided we allow $\BB d^{\op{CMP}}\left(\beta_1|_{[0,T_{1,r}]} , \beta_2|_{[0,T_{2,r}]} \right) = \infty$ (this doesn't pose a problem due to the $1\wedge$ in~\eqref{eqn-cmp-metric-loc}).

\subsection{Space-filling SLE}
\label{sec-wpsf-prelim}

Space-filling SLE$_{\kappa}$ for $\kappa  > 4$ is a variant of SLE$_{\kappa}$~\cite{schramm0} which was introduced in~\cite[Section~1.2.3]{ig4} (see also~\cite[Section 1.4.1]{wedges} for the whole-plane case). Space-filling SLE$_\kappa$ for $\kappa \geq 8$ is the same as ordinary SLE$_\kappa$ (which is already space-filling~\cite{schramm-sle}), whereas space-filling SLE$_\kappa$ for $\kappa \in (4,8)$ can be obtained from ordinary SLE$_\kappa$ be iteratively filling in the ``bubbles" which it disconnects from its target point by space-filling SLE$_\kappa$ type curves to obtain a space-filling curve (which is not a Loewner evolution). 

We will now review the construction of whole-plane space-filling SLE$_{\kappa}$ from~$\infty$ to~$\infty$, which is the version we will use most frequently. The basic idea of the construction is to first construct the outer boundary of the curve stopped when it hits each fixed point $z\in\BB Q^2$ using a pair of flow lines of a Gaussian free field, then use these outer boundary curves to define a partial order on $\BB Q^2$ which can be extended to a continuous curve. 

Let $\ul\kappa  = 16/\kappa \in (0,4)$ and let $\chi^{\op{IG}} := 2/\sqrt{\ul\kappa} -\sqrt{\ul\kappa}/2$.  Let $h^{\op{IG}}$ be a whole-plane GFF viewed modulo a global additive multiple of $2\pi\chi^{\op{IG}}$, as in~\cite{ig4} (here IG stands for ``Imaginary Geometry" and is used to distinguish the field $h^{\op{IG}}$ from the field $h$ corresponding to an LQG surface).  

For $z\in \BB C$ and $\theta \in (0,2\pi)$, we can define the \emph{flow line} of $h^{\op{IG}}$ started from $z$ with angle $\theta$ as in~\cite[Theorem~1.1]{ig4}. This flow line is a whole-plane SLE$_{\ul\kappa}(2-\ul\kappa)$ curve from $z$ to $\infty$ which is a.s.\ determined by $h^{\op{IG}}$ (whole-plane SLE$_{\ul\kappa}(2-\ul\kappa)$ is a variant of SLE$_{\ul\kappa}$ which is defined rigorously in~\cite[Section~2.1]{ig4}). For $z\in\BB Q^2$, let $\eta_z^L$ and $\eta_z^R$ be the flow lines started from $z$ with angles $\pi/2$ and $-\pi/2$, respectively. These curves will be the left and right boundaries of $\eta$ stopped upon hitting $z$. 

For distinct $z,w \in \BB Q^2$, the flow lines $\eta_z^L$ and $\eta_w^L$ a.s.\ merge upon intersecting, and similarly with $R$ in place of $L$. The two flow lines $\eta_z^L$ and $\eta_z^R$ started at the same point a.s.\ do not cross, but these flow lines bounce off each other without crossing if and only if $\kappa \in (4,8)$~\cite[Theorem~1.7]{ig4}. 

We define a total order on $\BB Q^2$ by declaring that $z$ comes before $w$ if and only if $w$ is in a connected component of $\BB C\setminus (\eta_z^L\cup \eta_z^R)$ which lies to the right of $\eta_z^L$ (equivalently, to the left of $\eta_z^R$).  The whole-plane analog of~\cite[Theorem~4.12]{ig4} (which can be deduced from the chordal case; see~\cite[Footnote 4]{wedges}) shows that there is a well-defined continuous curve $\eta : \BB R\rta \BB C$ which traces the points of $\BB Q^2$ in the above order, is such that $\eta^{-1}(\BB Q^2)$ is a dense set of times, and is continuous when parameterized by Lebesgue measure, i.e., in such a way that $\op{area}(\eta([a,b])) =b-a$ whenever $a < b$. The curve $\eta$ is defined to be the \emph{whole-plane space-filling SLE$_{\kappa}$ from $\infty$ to $\infty$} associated with $h^{\op{IG}}$.

The topology of $\eta$ is rather simple when $\kappa \geq 8$. In this case, the left/right boundary curves $\eta_z^L$ and $\eta_z^R$ do not bounce off each other, so for $a < b$ the set $\eta([a,b])$ has the topology of a closed disk. In the case $\kappa \in (4,8)$, matters are more complicated. The curves $\eta_z^L$ and $\eta_z^R$ intersect in an uncountable fractal set and for $a<b$ the interior of the set $\eta([a,b])$ a.s.\ has countably many connected components, each of which has the topology of a disk. See Figure~\ref{fig-weird-cell}. 

A similar construction, which is explained in~\cite[Section~1.2.3]{ig4}, shows that if $h^{\op{IG}}$ is a GFF on a simply connected domain $\mcl D \subset\BB C$, $\mcl D\not=\BB C$, with appropriate Dirichlet boundary data and $x,y\in\bdy \mcl D$ then one can also define a chordal space-filling SLE$_{\kappa}$ from $x$ to $y$ in $\mcl D$. Taking a limit as $y\rta x$ from the clockwise (resp.\ counterclockwise) direction gives a counterclockwise (resp.\ clockwise) space-filling SLE$_{\kappa}$ loop in $\mcl D$ based at $x$ (see~\cite[Appendix A.3]{bg-lbm} for details). In the case when $\kappa \in (4,8)$, this latter curve can also be defined as the restriction of a whole-plane or chordal space-filling SLE$_{\kappa}$ to one of the time intervals during which it is filling in a ``bubble" which it disconnects from its target point \cite{ig4}.

\subsection{Liouville quantum gravity} 
\label{sec-lqg-prelim}

For $\gamma \in (0,2)$ and $k\in \BB N_0$, a \emph{$\gamma$-Liouville quantum gravity surface} with $k$ marked points is an equivalence class of $k+2$-tuples $(\mcl D,h,z_1,\dots,z_k)$ where $\mcl D\subset\BB C$ is an open domain, $h$ is a distribution on $\mcl D$ (which we will always take to be a realization of some variant of the Gaussian free field on $\mcl D$), and $z_1,\dots,z_k \in \mcl D\cup \bdy \mcl D$. Two such $k+2$-tuples $(\mcl D,h,z_1,\dots,z_k)$ and $(\wt{\mcl D} , \wt h , \wt z_1,\dots , \wt z_k)$ are declared to be equivalent if there is a conformal map $f : \wt{\mcl D} \rta \mcl D$ such that
\eqb \label{eqn-lqg-coord}
\wt h = h\circ f + Q\log |f'| \quad \op{and} \quad f(\wt z_j) = z_j  ,\quad \forall j\in [1,k]_{\BB Z}  \quad \text{where} \quad Q = \frac{2}{\gamma}  + \frac{\gamma}{2} .
\eqe
We call different choices of the distribution $h$ corresponding to the same LQG surface different \emph{embeddings} of the surface. The above definition first appeared in~\cite{shef-kpz}, and also plays an important role, e.g., in~\cite{shef-zipper,wedges}.

If the field $h$ above is locally absolutely continuous in law with respect to the Gaussian free field on $\mcl D$ (see~\cite{shef-gff} and the introductory sections of~\cite{ss-contour,shef-zipper,ig1,ig4} for more on the GFF), then we can define the \emph{$\gamma$-LQG area measure} $\mu_h$ on $\mcl D$, which is the a.s.\ limit of regularized versions of $e^{\gamma h(z)} \,dz$~\cite{shef-kpz} as well as the \emph{$\gamma$-LQG boundary length measure} $\nu_h$, which is a measure on certain curves in $\mcl D$, including $\bdy \mcl D$~\cite{shef-kpz} and SLE$_\kappa$ type curves for $\kappa =\gamma^2$ which are independent from $h$~\cite{shef-zipper}. If $h$ and $\wt h$ are related by a conformal map as in~\eqref{eqn-lqg-coord}, then $f_* \mu_{\wt h} = \mu_h$ and $f_* \nu_{\wt h} = \nu_{ h}$, so $\mu_h$ and $\nu_h$ can be viewed as measures on the LQG surface. We note that there is an alternative, more general approach for defining regularized measures of the above form building on the work of Kahane~\cite{kahane}; see~\cite{rhodes-vargas-review} for an overview of this theory.

In this paper, we will be interested in several different types of $\gamma$-LQG surfaces which are introduced in~\cite{wedges}.

\subsubsection{Quantum cones}

The most important LQG surface for our purposes is the \emph{$\alpha$-quantum cone} for $\alpha \in (-\infty,Q)$, which was first defined in~\cite[Definition~4.10]{wedges}. 
The $\alpha$-quantum cone is a doubly marked surface $(\BB C ,h , 0, \infty)$ whose $\gamma$-LQG measure $\mu_h$ has infinite total mass, but assigns finite mass to every bounded subset of $\BB C$. 
Roughly speaking, this quantum surface is obtained by starting with a whole-plane GFF plus $\alpha\log (1/|\cdot|)$ then ``zooming in" near the origin and re-scaling (i.e., adding a constant to the field) so that the quantum mass of the unit disk remains of constant order~\cite[Proposition~4.13(i)]{wedges}. 

We will primarily be interested in quantum cones with $\alpha \in \{0,\gamma\}$. 
The case $\alpha =\gamma$ is special since a GFF a.s.\ has a $-\gamma$-log singularity at a point sampled from its $\gamma$-LQG measure~\cite[Section 3.3]{shef-kpz}, so this surface can be thought of as describing the behavior of a general $\gamma$-LQG surface near a quantum typical point. Moreover, the $\gamma$-quantum cone arises in the mating-of-trees construction of SLE-decorated LQG~\cite[Theorems 1.9 and 1.1]{wedges}, and hence is involved in the a priori embedding of the $\gamma$-mated-CRT map. 
The 0-quantum cone describes the local behavior of a $\gamma$-LQG surface near a \emph{Lebesgue} typical point. 
This type of quantum cone will be important for us since it can be used to construct cell configurations which satisfy the hypotheses of Theorem~\ref{thm-general-clt-uniform} (Section~\ref{sec-0cone}).
 
We will need some properties of quantum cones which follow from the definition in~\cite[Definition 4.10]{wedges}, so we now recall this definition. Let $\alpha < Q$ and let $A : \BB R \rta \BB R$ be the process such that $A_t =B_t  + \alpha t$ for $t\geq 0$, where $B$ is a standard linear Brownian motion; and for $t < 0$, let $A_t = \wh B_{-t} + \alpha t$, where $\wh B$ is a standard linear Brownian motion conditioned so that $\wh B_t  + (Q-\alpha) t > 0$ for all $t> 0$, taken to be independent from $B$. 
We define $h$ to be the random distribution such that if $h_r(0)$ denotes the circle average of $h$ on $\partial B_r(0)$ (see~\cite[Section 3.1]{shef-kpz} for the definition and basic properties of the circle average), then $t\mapsto h_{e^{-t}}(0)$ has the same law as the process $A$; and $h - h_{|\cdot|}(0)$ is independent from $h_{|\cdot|}(0)$ and has the same law as the analogous process for a whole-plane GFF.

By the definition of an LQG surface, the distribution $h$ is only defined up to re-scaling (as we have fixed only two marked points), but we will almost always consider the particular choice of distribution $h$ defined just above. This choice of $h$ is called the \emph{circle average embedding}, and is characterized by the fact that $1 = \sup\{r > 0 : h_r(0) + Q\log r = 0\}$. The circle average embedding is particularly convenient since it has the property that $h|_{\BB D}$ agrees in law with the restriction to $\BB D$ of a whole-plane GFF plus $-\alpha\log |\cdot|$, normalized so that its circle average over $\bdy\BB D$ is 0. 
 
We will also use a certain special scale invariance property of the $\alpha$-quantum cone which we now describe.
Let $h$ be the circle-average embedding of an $\alpha$-quantum cone, $\alpha < Q$, and for $r > 0$ and $z\in \BB C$, let $h_r(z)$ be the circle average of $h$ over $\bdy B_r(z)$. For $b > 0$, let
\eqb \label{eqn-mass-hit-time}
R_b := \sup\left\{ r > 0 : h_r(0) + Q \log r = \frac{1}{\gamma} \log b \right\} 
\eqe 
where here $Q$ is as in~\eqref{eqn-lqg-coord}. That is, $R_b$ gives the largest radius $r > 0$ so that if we scale spatially by the factor $r$ and apply the change of coordinates formula~\eqref{eqn-lqg-coord}, then the average of the resulting field on $\partial \BB D$ is equal to $\gamma^{-1} \log b$. Note that $R_0 = 0$. It is easy to see from the above definition of $h$ (and is shown in~\cite[Proposition 4.13(i)]{wedges}) that for each fixed $b>0$, 
\eqb \label{eqn-cone-scale}
h \eqD h(R_b \cdot) + Q \log R_b -  \frac{1}{\gamma} \log b .
\eqe
By~\eqref{eqn-lqg-coord}, if we let $h^b$ be the field on the left side of~\eqref{eqn-cone-scale}, then a.s.\ $\mu_{h^b}(A) = b \mu_h(R_b^{-1} A)$ for each Borel set $A\subset \BB C$. In particular, typically $\mu_h(B_{R_b}) \asymp b$. 

We now record a basic estimate for the radii $R_b$ in~\eqref{eqn-mass-hit-time}.

\begin{lem} \label{lem-cone-hit-tail}
There is a constant $a = a(\alpha,\gamma)  > 0$ such that for each $b_2>b_1>0$ and each $C > 1$,
\begin{align} \label{eqn-cone-hit-tail}
&\BB P\left[ C^{-1} (b_2/b_1)^{ \tfrac{1 }{\gamma (Q-\alpha) }} \leq R_{b_2}/R_{b_1} \leq C (b_2/b_1)^{ \tfrac{1 }{\gamma (Q-\alpha)}} \right] \notag\\
&\qquad\geq 1  -   3 \exp\left( - \frac{  a  (\log C)^2 }{   \log (b_2/b_1) +  \log C   } \right)  .
\end{align}
\end{lem}
\begin{proof}   
By the definition of an $\alpha$-quantum cone given just above and the strong Markov property of Brownian motion, $ \log(R_{b_2}/R_{b_1})$ has the same law as the first time a standard linear Brownian motion with negative linear drift $-(Q-\alpha) t$ hits $ -  \frac{1}{\gamma} \log(b_2/b_1) $. Applying the Gaussian tail bound in the final inequality below, if we write $B(\cdot)$ for a standard Brownian motion then
\alb
&\BB P\left[  \log(R_{b_2}/R_{b_1}) >   \frac{ \log (b_2/b_1) }{\gamma (Q-\alpha)} + \log C \right]  \\
 \leq& \BB P\left[ B\left(  \frac{ \log (b_2/b_1) }{\gamma (Q-\alpha)}+ \log C \right) - \frac{1}{\gamma} \log(b_2/b_1) - (Q-\alpha) \log C >  - \frac{1}{\gamma} \log(b_2/b_1) \right]\\  
\leq& \exp\left( - \frac{ (Q-\alpha)^2  (\log C)^2 }{  2  \left( (\gamma(Q-\alpha))^{-1} \log (b_2/b_1) +  \log C  \right)        } \right).
\ale 
This gives the upper bound for $R_{b_2}/R_{b_1}$ in~\eqref{eqn-cone-hit-tail} for a suitable choice of $a$. For the lower bound, we use a similar argument as above together with the reflection principle to get that if $(\gamma(Q-\alpha))^{-1} \log (b_2/b_1) -  \log C  > 0$, then 
\alb
&\BB P\left[  \log(R_{b_2}/R_{b_1}) <  \frac{ \log (b_2/b_1) }{\gamma (Q-\alpha)} - \log C \right] \notag \\
&\qquad \leq 2 \exp\left( - \frac{ (Q-\alpha)^2 (\log C)^2 }{  2  \left( (\gamma(Q-\alpha))^{-1} \log (b_2/b_1) -  \log C  \right)       } \right) .
\ale
On the other hand, the left side equals $0$ if $(\gamma(Q-\alpha))^{-1} \log (b_2/b_1) -  \log C  \leq 0$. 
\end{proof}

\subsubsection{Quantum disks, spheres, and wedges.}

A \emph{quantum disk} is a finite-mass quantum surface $(\BB D , h)$ defined in~\cite[Definition~4.21]{wedges} (we will not need the precise definition here). One can consider quantum disks with specified boundary length $\nu_h(\BB D)$ or with specified boundary length and area $\mu_h(\BB D)$. It is natural to consider quantum disks with any number of marked boundary points and interior points sampled uniformly from $\nu_h$ and $\mu_h$, respectively, but we will typically work with quantum disks with a single marked boundary point.  

A \emph{(doubly marked) quantum sphere} is a quantum surface $(\BB C ,h , 0, \infty)$ introduced in~\cite[Definition~4.21]{wedges} with $\mu_h(\BB C ) < \infty$ (again, we will not need the precise definition here). The marked points 0 and $\infty$ correspond to $\pm\infty$ in the infinite cylinder in the parameterization considered in~\cite[Definition~4.21]{wedges}. Typically one considers a unit-area quantum sphere, which means we fix $\mu_h(\BB C) = 1$. Quantum spheres with other areas are obtained by re-scaling (equivalently, adding a constant to $h$). A singly marked quantum sphere is obtained by forgetting one of the marked points.

For $\alpha \leq Q$, an \emph{$\alpha$-quantum wedge} is a quantum surface $(\BB H , h , 0,\infty)$ defined in~\cite[Definition~4.5]{wedges} which has finite mass in every neighborhood of 0 but infinite total mass. It is the half-plane analog of the $\alpha$-quantum cone considered above. 

One can extend the definition of the $\alpha$-quantum wedge to the case when $\alpha \in (Q,Q+\gamma/2)$~\cite[Definition~4.15]{wedges}. In this case, an $\alpha$-quantum cone does not have the topology of the upper half-plane but instead is a Poissonian string of \emph{beads}, each of which is itself a finite-mass quantum surface homeomorphic to the disk with two marked boundary points. One can consider a single bead of an $\alpha$-quantum wedge with specified area $\frk a$ and left/right boundary lengths $\frk l_L$ and $\frk l_R$, which means we sample from the regular conditional law of the intensity measure on beads conditioned so that the total quantum mass is $\frk a$ and the arcs between the two marked points have quantum lengths $\frk l_L$ and $\frk l_R$, respectively.

\subsection{Mated-CRT maps with the plane and sphere topology}
\label{sec-mated-crt-map}

The special case of a mated-CRT map with the disk topology is defined in Section~\ref{sec-overview} and illustrated in Figure~\ref{fig::map}. For $n\in\BB N$, the mated-CRT map with the sphere topology with $n$ total vertices is defined in exactly the same way, except we replace $(L,R)$ with a pair of Brownian motions with variances and covariances as in~\eqref{eqn-bm-cov} conditioned to stay in the first quadrant for one unit of time and end up at $(0,0)$ (instead of at $(0,1)$), so that the boundary is empty. 

To define the mated-CRT map with the plane topology, we let $ (L_t ,R_t)_{t\in \BB R}$ be a correlated two-sided two-dimensional Brownian motion with variances and covariances as in~\eqref{eqn-bm-cov} (and no conditioning). For $n\in\BB N$, we then let $\mcl G^{1/n}$ be the map with vertex set $\frac{1}{n} \BB Z$, with two such vertices joined by an edge if and only if~\eqref{eqn-inf-adjacency0} holds. This is the variant of the mated-CRT map considered in~\cite{ghs-dist-exponent,ghs-map-dist,gm-spec-dim,gh-displacement}. Note that, by Brownian scaling, the law of $\mcl G^{1/n}$ (as a graph) does not depend on $n$, but it is convenient to view the maps $\{\mcl G^{1/n}\}_{n\in\BB N}$ as all being coupled together with the same Brownian motion.

\subsection{Structure graph and a priori embedding}
\label{sec-sg-def} 
  
The results of~\cite{wedges,sphere-constructions,ag-disk} imply that mated-CRT maps can be realized as cell configurations called \emph{structure graphs} constructed from space-filling SLE$_{\kappa }$ curves parameterized by quantum mass with respect to a certain independent LQG surface. In this subsection, we define these structure graphs and explain their relationship to mated-CRT maps.
 
Suppose we are in one of the following three cases, which correspond to the mated-CRT maps with the topology of the plane, sphere, and disk. 
\begin{enumerate}
\item \textbf{Plane:} $\mcl D = \BB C$; $(\BB C ,h , 0, \infty)$ is a $\gamma$-quantum cone, and $\eta : \BB R\rta\BB C$ is an independent whole-plane space-filling SLE$_{\kappa}$ from $\infty$ to $\infty$ parameterized by $\gamma$-quantum mass with respect to $h$ in such a way that $\eta(0) = 0$. \label{item-rw-conv-cone}
\item \textbf{Sphere:} $\mcl D = \BB C$; $(\BB C ,h ,  \infty)$ is a unit area quantum sphere, and $\eta : [0,1]\rta\BB C$ is an independent whole-plane space-filling SLE$_{\kappa}$ from $\infty$ to $\infty$ parameterized by $\gamma$-quantum mass with respect to~$h$.   \label{item-rw-conv-sphere}
\item \textbf{Disk:} $\mcl D = \BB D$; $(\BB D , h , 1)$ is a singly marked quantum disk with unit area and unit boundary length and $\eta : [0,1] \rta \ol{\BB D}$ is an independent chordal space-filling SLE$_{\kappa}$ loop in $\BB D$ based at 1 parameterized by $\gamma$-quantum mass with respect to $h$. \label{item-rw-conv-disk}  
\end{enumerate}

For $n\in\BB N$, we define the \emph{$1/n$-structure graph} $\mcl G^{1/n}$ associated with $h$ as follows. 
The vertex set of $\mcl G^{1/n}$ is
\eqb \label{eqn-sg-vertex-set}
\mcl V\mcl G^{1/n} := \left\{ x\in (\frac{1}{n} \BB Z) \cap \BB I   \right\} ,
\eqe 
where $\BB I = \BB R$ in case~\ref{item-rw-conv-cone} and $\BB I = (0,1]$ in cases~\ref{item-rw-conv-sphere} and~\ref{item-rw-conv-disk}. 
Two vertices $x_1,x_2\in \frac{1}{n} \BB Z$ are connected by an edge if and only if the corresponding cells $\eta([x_1-1/n ,x_1]  )$ and $\eta([x_2-1/n ,x_2] )$ intersect along a connected boundary arc. They are connected by two edges if these cells share both a left boundary arc and a right boundary arc. 
The following is proven in~\cite{wedges,sphere-constructions}.

\begin{figure}[ht!]
\begin{center}
\includegraphics[scale=.8]{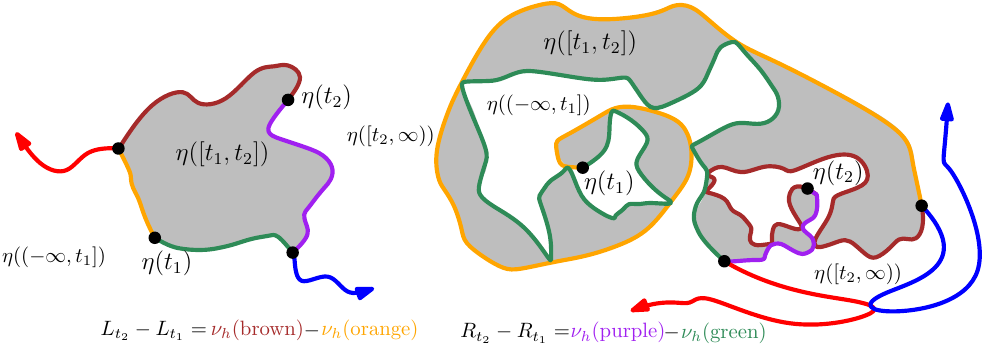} 
\caption{\label{fig-LR-def} Illustration of the definition of the left/right boundary length process $(L,R)$ in the whole-plane case. The left figure corresponds to the case when $\kappa \geq 8$, so that $\eta([t_1,t_2])$ is simply connected. The right figure corresponds to the case when $\kappa  \in (4,8)$, in which case the topology is more complicated, but the definition of $(L,R)$ is the same. In both cases, the intersection of the left (resp.\ right) outer boundaries of $\eta((-\infty,t_1])$ and $\eta([t_2,\infty))$ is shown in red (resp.\ blue).
}
\end{center}
\end{figure}

\begin{prop} \label{prop-a-priori-embedding}
In case~\ref{item-rw-conv-cone},~\ref{item-rw-conv-sphere}, and~\ref{item-rw-conv-disk}, respectively, the family of structure graphs $\{\mcl G^{1/n}\}_{n\in\BB N}$ agrees in law with the family of $1/n$-mated-CRT maps with the plane, sphere, or disk topology, as defined in Section~\ref{sec-mated-crt-map}. In case~\ref{item-rw-conv-disk}, the boundary $\bdy \mcl G^{1/n}$ as defined in~\eqref{eqn-mated-crt-map-bdy} is precisely the set of $y\in (\frac{1}{n}\BB Z)\cap (0,1]$ for which $\eta([y-1/n,y]) \cap \bdy\BB D\not=\emptyset$. 
\end{prop}
\begin{proof}
Let us first consider the whole-plane case~\ref{item-rw-conv-cone}. 
For $t\in\BB R$, let $L_t$ (resp.\ $R_t$) be the net change in the $\gamma$-quantum length of the left (resp.\ right) outer boundary of $\eta$ between times 0 and $t$, so that, e.g., for $t\geq 0$, 
\alb
L_t &=\nu_h\left( \text{left boundary of $ \eta([0,t]) \cap \eta([t,\infty))$} \right) \notag \\
&\qquad - \nu_h\left( \text{left boundary of $ \eta([0,t]) \cap \eta((-\infty,0])$} \right)  
\ale
and similar formulas hold for $t\leq 0$ and/or with $R$ in place of $L$. See Figure~\ref{fig-LR-def} for an illustration.
It is shown in~\cite[Theorem~1.9]{wedges} (and~\cite[Theorem 1.1]{kappa8-cov} in the case $\gamma  < \sqrt 2$) that $(L_t,R_t)_{t\in\BB R}$ has the law of a two-sided correlated two-dimensional Brownian motion with $\op{Corr}(L_t,R_t)  = -\cos(\pi\gamma^2/4)$ for each $t\in\BB R\setminus \{0\}$, i.e., $(L,R)$ is exactly the Brownian motion used to define the whole-plane mated-CRT map. 

It can be checked using the above definition of $(L,R)$ that two cells $\eta([x_1-1/n , x_1])$ and $\eta([x_2-1/n,x_2])$ for $x_1,x_2 \in \frac{1}{n} \BB Z $ intersect along a non-trivial connected boundary arc if and only if the mated-CRT map adjacency condition~\eqref{eqn-inf-adjacency0} holds; see~\cite[Section 8.2]{wedges}. The conditions in terms of $L$ and $R$ correspond to intersections along left or right boundary arcs, respectively.
Thus the proposition statement holds in the whole-plane case. 

The sphere case follows similarly using~\cite[Theorem~1.1]{sphere-constructions}. The disk case follows similarly follows from~\cite[Theorem~1.1]{ag-disk} (see~\cite[Theorem~2.1]{sphere-constructions} for an earlier proof in the special case when $\gamma\in(\sqrt 2 , 2)$). 
\end{proof}

Proposition~\ref{prop-a-priori-embedding} gives us an a priori embedding of the mated-CRT map into $\mcl D$ by sending each $x\in \mcl V\mcl G^{1/n}$ to the point $\eta(x) \in \BB C$. To prove Theorem~\ref{thm-tutte-conv0}, we will show that in the disk case this a priori embedding is close to the Tutte embedding of the mated-CRT map when $n$ is large. 
 
\begin{remark} \label{remark-other-disk}
Proposition~\ref{prop-a-priori-embedding} gives rise to a priori embeddings of a much more general class of mated-CRT map variants, as we now explain.
It is shown in~\cite[Theorem 1.11]{wedges} that the Brownian motion $(L,R)$ a.s.\ determines the pair $(h,\eta)$ modulo rotation and scaling. 
In fact, the Brownian motion $(L,R)$ determines $(h,\eta)$ locally in the sense that for $[a,b]\subset \BB R$, the quantum surface obtained by restricting $h$ to $\eta([a,b])$ and the curve $\eta|_{[a,b]}$, viewed as a curve on this quantum surface, are a.s.\ determined by $((L,R) - (L_a,R_a))|_{[a,b]}$ (however, $\eta([a,b])$ itself is not a.s.\ determined by $((L,R) - (L_a,R_a))|_{[a,b]}$). This is implicit in the proof of~\cite[Theorem~1.11]{wedges} and is proven explicitly in, e.g.,~\cite[Lemma~3.13]{ghs-bipolar}. This implies that if $(\wt L , \wt R)  : [ a_0 ,b_0] \rta \BB R$ is a process such that the law of $((\wt L , \wt R)  - (\wt L_a , \wt R_a))|_{[a,b]}$ is absolutely continuous with respect to the law of $((L,R) - (L_a,R_a))|_{[a,b]}$ for each interval $[a,b] \subset (a_0 ,b_0)$, then one can construct from $(\wt L ,\wt R)$ a random measure space with a conformal structure decorated by a space-filling curve parameterized by the mass of the region it fills in which locally looks like a $\gamma$-LQG surface decorated by an independent space-filling SLE$_{\kappa}$ curve. Moreover, the $1/n$-structure graphs associated with this curve will be the same as the $1/n$-mated-CRT maps defined as in~\eqref{eqn-inf-adjacency0} with $(\wt L , \wt R)$ in place of $(L,R)$. 
\end{remark}

\section{Proofs of main results}
\label{sec-embedding}
 
In this section we will prove the main results stated in Section~\ref{sec-intro}, modulo an SLE/LQG estimate which is proven in Section~\ref{sec-moment}. We start in Section~\ref{sec-0cone} by showing that random walk on the structure graph associated with space-filling SLE$_\kappa$ on a 0-quantum cone converges to Brownian motion. This is done by checking the hypotheses of Theorem~\ref{thm-general-clt-uniform}. We then want to transfer this from the 0-quantum cone to other $\gamma$-LQG surfaces---in particular, the ones involved in the a priori embeddings of mated-CRT maps discussed in Section~\ref{sec-sg-def}---using local absolute continuity. To this end, we need a technical lemma to the effect that locally absolutely continuous GFF-type distributions give rise to locally absolutely continuous structure graphs, which we prove in Section~\ref{sec-cell-abs-cont}. Section~\ref{sec-rw-conv} contains the proof of Theorem~\ref{thm-rw-conv}, which states that random walk on the a priori embedding of the mated-CRT map with the plane, sphere, or disk topology converges to $\gamma$-LQG. We then use this to prove Theorem~\ref{thm-tutte-conv0} in Section~\ref{sec-embedding-proof}.

\subsection{Cell configuration arising from a 0-quantum cone}
\label{sec-0cone}

The purpose of this subsection is to transfer from the general setting of Theorem~\ref{thm-general-clt-uniform} to the particular setting of SLE-decorated Liouville quantum gravity. 
For this, we need to work with an LQG surface wherein the origin is a Lebesgue typical point (due to hypothesis~\ref{item-hyp-resampling} of Theorem~\ref{thm-general-clt-uniform}) instead of working directly with the LQG surfaces described in Section~\ref{sec-sg-def}. 

Let $(\BB C ,h , 0, \infty)$ be a 0-quantum cone (as discussed in Section~\ref{sec-lqg-prelim}). For concreteness we assume that $h$ is given the circle average embedding, as in~\cite[Definition 4.10]{wedges}. Let $\eta$ be a whole-plane space-filling SLE$_{\kappa}$ sampled independently from $h$ and then parameterized by $\gamma$-LQG mass with respect to $h$ in such a way that $\eta(0)=0$. Let $\indshift$ be sampled uniformly from $[0,1]$, independently from everything else, and let
\eqb \label{eqn-0cone-cells}
\mcl H = \left\{ \eta([x-1,x]) : x\in \BB Z + \indshift\right\} .
\eqe 
For $H,H' \in \mcl H$ with $H\not=H'$, we declare that $H\sim H'$ if and only if $H$ and $H'$ intersect along a non-trivial boundary arc. 
We set $\frk c(H,H') = 2$ if $H$ and $H'$ intersect along both their left and right boundaries (this corresponds to a double edge of the mated-CRT map) and otherwise we set $\frk c(H,H')  =1$. 
Then $\mcl H$ is a cell configuration (Definition~\ref{def-cell-config}). (The reason for including the index shift $\indshift$ in~\eqref{eqn-0cone-cells} is to avoid making $0 = \eta(0)$ a special point; this is necessary in order to check that hypothesis~\ref{item-hyp-resampling} is satisfied).

\begin{prop} \label{prop-0cone}
The conclusion of Theorem~\ref{thm-general-clt-uniform} holds for the cell configuration $\mcl H$ above. Moreover, the covariance matrix $\Sigma$ of the limiting Brownian motion is a positive scalar multiple of the identity matrix.
\end{prop}
\begin{proof}
We will check that the hypotheses of Theorem~\ref{thm-general-clt-uniform} are satisfied. Our proofs of two of the hypotheses (ergodicity modulo scaling and coordinate change) use results which we prove later, in Section~\ref{sec-moment} and Appendix~\ref{sec-ergodicity}.
\medskip

\noindent\textbf{Translation invariance modulo scaling.} For $j\in\BB N$, let $R_j$ be the largest $r> 0$ for which $h_r(0) + Q\log r =  \gamma^{-1} \log j$, where $h_{r}(0)$ denotes the circle average, as in~\eqref{eqn-mass-hit-time}. We will check the needed resampling property for $U_j = B_{R_j}(0)$. By~\eqref{eqn-cone-scale}, the field $h^j := h(R_j\cdot) + Q\log R_j - \gamma^{-1} \log j$ agrees in law with $h$. In particular, by the discussion just after~\cite[Definition 4.10]{wedges}, $h^j|_{\BB D}$ agrees in law with the corresponding restriction of a whole-plane GFF, normalized so that its circle average over $\bdy\BB D$ is 0. Consequently, if we sample $z_j$ uniformly from Lebesgue measure on $B_{R_j}(0)$, then the proof of~\cite[Proposition 4.13(ii)]{wedges} along with the translation invariance of the law of the whole-plane GFF, modulo additive constant, shows that $\mu_h(\cdot-z_j)$ converges in law to $\mu_h$ modulo scaling as $j\rta\infty$ in the local total variation sense. That is, there is a sequence of random constants $C_j\rta\infty$ such that for each $r>0$, the restriction to $B_r(0)$ of $\mu_h(C_j(\cdot-z_j)) = \mu_{h(C_j(\cdot-z_j)) + Q\log C_j}$ converges in the total variation sense to $\mu_h|_{B_r(0)}$. Since the law of $\eta$, viewed modulo time parameterization, is invariant under translation and scaling, we find that also $C_j(\mcl H-z_j)$ converges in law to $\mcl H$ as $j\rta\infty$ (here we also use that if $\tau_z$ is the first time that $\eta$ hits $z\in\BB C$, then the fractional part of $\indshift-\tau_z$ is uniform on $[0,1]$). 
\medskip

\noindent\textbf{Ergodicity modulo scaling.} 
Let $F = F(\mcl H)$ be a bounded measurable function of $\mcl H$ satisfying $F(C(\mcl H - z)) = F(\mcl H)$ for every $C > 0$ and $z\in\BB C$. 
We need to show that $F$ is a.s.\ equal to a deterministic constant.
Write $\wh\eta$ for the space-filling SLE curve $\eta$, viewed modulo time parametrization.  
By definition, $\mcl H$ is a measurable function of $h , \wh\eta$, and $\theta$. 
We can therefore write $F = G(h ,\wh\eta , \theta) $ for some measurable function $G$. 

We now express the scale and translation invariance property of $F$ in terms of $G$. 
For $z\in\BB C$, let $\tau_z$ be the first time that $\eta$ hits $z$. Then, by the definition of $\mcl H$ and the LQG coordinate change formula, for every $C > 0$ and $z\in\BB C$,  
\eqb \label{eqn-G-invariant}
G\left( h(C^{-1} \cdot  +z) + Q\log C^{-1} , C \wh\eta  +z ,  [\theta - \tau_z] \right) 
= G(h ,  \wh\eta  , \theta) ,
\eqe 
where $[x] := \lceil x \rceil - x$. 

In Lemma~\ref{lem-cone-trivial} below, we show that any measurable functional of $h$ which is invariant under $h \mapsto h(C^{-1} \cdot  +z) + Q\log C^{-1}$ for $C> 0$ and $z\in\BB C$ is a.s.\ equal to a deterministic constant. 
We want to use this to argue that $G$ is a.s.\ equal to a deterministic constant. 
To do this we need to reduce to looking at a function which depends only on $h$. 

We first eliminate the dependence on $\theta$ using an averaging trick. 
Let $\phi : \BB R\rta \BB R$ be a bounded measurable function and define
\eqbn
G'  =G'(h,\wh\eta) 
:= \BB E\left[ \phi(G) \,|\, h ,  \wh\eta \right] 
= \int_0^1 (\phi\circ G)(h,\wh\eta , t ) \, dt .
\eqen
For each fixed $z$, the mapping $t \mapsto  [t - \tau_z]$ is measure-preserving from $[0,1]$ to $[0,1]$. Consequently,
\allb \label{eqn-Gbar-invariant}
&G' \left( h(C^{-1} \cdot  +z) + Q\log C^{-1} , C \wh\eta + z \right) \notag\\
&\qquad = \int_0^1  (\phi\circ G) \left( h(C^{-1} \cdot  +z) + Q\log C^{-1} , C \wh\eta+ z , t \right) \,dt \notag\\ 
&\qquad= \int_0^1  (\phi\circ G)\left( h(C^{-1} \cdot  +z) + Q\log C^{-1} , C \wh\eta + z , [t-\tau_z] \right) \,dt \notag\\
&\qquad= \int_0^1 (\phi\circ G)\left( h  ,  \wh\eta  , t \right) \,dt \quad \text{(by~\eqref{eqn-G-invariant})} \notag\\
&\qquad= G'(h , \wh\eta) .
\alle
That is, $G'$ is a translation and scale invariant functional of the pair $(h, \wh\eta)$. 

We next eliminate the dependence on $\wh\eta$ using a similar trick. Let $\phi' : \BB R\rta \BB R$ be another bounded measurable function and let 
\eqbn
G'' = G''(h) = \BB E\left[ \phi'(G') \,|\, h \right] .
\eqen 
For a given realization of $h$ and $C>0$ and $z\in\BB C$, we can construct $G''(h(C^{-1} \cdot + z) + Q\log C^{-1} )$ by first sampling a whole-plane space-filling SLE$_\kappa$ $\wh\eta'\eqD \wh \eta$, then averaging $(\phi'\circ G')(h(C^{-1} \cdot + z) + Q\log C^{-1} , \wh\eta')$ w.r.t.\ the law of $\wh\eta'$. 
Since the law of $\wh\eta$ is scale and translation invariant, we have $\wh\eta' \eqD C \wh \eta + z$. 
Therefore, $G''$ is also given by averaging $(\phi'\circ G')(h(C^{-1} \cdot + z) + Q\log C^{-1} , C\wh\eta +z )$ w.r.t.\ the law of $\wh\eta$. 
Hence by~\eqref{eqn-Gbar-invariant}, we obtain $G''(h(C^{-1} \cdot + z) + Q\log C^{-1}) = G''(h) $. 

By Lemma~\ref{lem-cone-trivial}, $G''$ is a.s.\ equal to a deterministic constant. This holds for any possible choice of $\phi'$, so $G'$ is independent from $h $. Therefore, $G'$ is a.s.\ equal to a measurable function of only $\wh\eta$.
Recall from Section~\ref{sec-wpsf-prelim} that $\wh\eta$ is constructed from a whole-plane GFF $h^{\op{IG}}$. From this construction one gets that replacing $h^{\op{IG}}$ by $h^{\op{IG}}(C^{-1}(\cdot + z))$ corresponds to replacing $\wh\eta$ by $C\wh\eta+z$. 
It is easy to see using the orthonormal basis expansion of the GFF and the zero-one law for translation invariant events that the tail $\sigma$-algebra generated by the functions of $h^{\op{IG}}$ which are invariant under replacing $h^{\op{IG}}$ by $h^{\op{IG}}(C^{-1}\cdot + z)$ is trivial (see Remark~\ref{remark-gff-trivial}). 
Therefore, $G'$ is a.s.\ equal to a deterministic constant. 
Since the bounded measurable function $\phi$ used to define $G'$ is arbitrary, we infer that $G$ is independent from $h$ and $h^{\op{IG}}$, so $G$ is a.s.\ equal to a function of only $\theta$. 
By~\eqref{eqn-G-invariant}, we have $G([\theta-\tau_z]) = G(\theta)$ for each $z\in\BB C$. Since $[\theta-\tau_z]$ ranges over all of $[0,1]$ as $z$ ranges over $\BB C$, we get that $G$ is a.s.\ equal to a deterministic constant.  
\medskip

\noindent\textbf{Finite expectation.} Since our conductances are all equal to 1 or 2,~\eqref{eqn-hyp-moment} is equivalent to\\
$\BB E\left[ \op{deg}(H_0) \op{diam}(H_0)^2/\op{area}(H_0)\right] < \infty$, where $\op{deg}$ denotes the degree (number of incident edges). The proof of this estimate is given in Section~\ref{sec-moment} (in fact, we get finite moments up to order $4/\gamma^2$). 
\medskip

\noindent\textbf{Connectivity along lines.} Obviously, any two points on a line in $\BB C$ can be joined by a path of cells which intersect at at least one point (not necessarily along a non-trivial boundary arc). If $H$ and $H'$ are cells which intersect at a point $z$, but do not share a non-trivial boundary arc, then we can find a finite number of cells $H = H_0 , H_1 , \dots , H_N = H'$ such that each $H_j$ contains $z$ and $H_j$ and $H_{j-1}$ share a non-trivial boundary arc for each $j=1,\dots,N$ (see Figure~\ref{fig-weird-cell}). Indeed, this follows from the fact that space-filling SLE can hit a fixed point at most a deterministic constant number of times~\cite[Section 8.2]{wedges}. Combining these facts gives the desired connectivity property. 
\medskip 

\noindent The covariance matrix $\Sigma$ is a scalar multiple of the identity since the law of $(h,\eta)$ is invariant under rotations about the origin. 
\end{proof}

\subsection{Local absolute continuity for space-filling SLE cells}
\label{sec-cell-abs-cont}

The purpose of this brief subsection is to address the following technical difficulty. 
The cells of the structure graphs $\mcl G^{1/n}$ constructed from a GFF-type distribution $h$ and an independent space-filling SLE$_{\kappa}$ curve $\eta$ parameterized by $\gamma$-LQG mass as in Section~\ref{sec-sg-def} are \emph{not} locally determined by $h$ and $\eta$. Indeed, if $z\in\BB C$ and $U\subset \BB C$ is an open set containing $z$, then the cell of $\mcl G^{1/n}$ containing $z$ is only determined by $h|_U$ and $\eta$ (viewed modulo time parameterization) by an index shift: it could be any of the sets $\eta([t-{1/n},t])$ for $t$ in some ${1/n}$-length interval of time. The index shifts could be different for different points $z,w$ in $U$ if $\eta$ exits $U$ between hitting $z$ and $w$. 

Nevertheless, we will have occasion to compare the local structure graphs associated with two different GFF's whose laws are locally mutually absolutely continuous. In this subsection we will prove a lemma which allows us to do so.

Let $\mcl D \subset \BB C$ be a simply connected open set (possibly all of $\BB C$) and let $U\subset U'\subset \mcl D$ be open sets with $\ol U \subset U'$. 
Let $\mu_1,\mu_2$ be random non-atomic locally finite Borel measures on $\mcl D$ which assign positive mass to every open subset of $\mcl D$. Assume that the laws of the restrictions $\mu_1|_{U'}$ and $\mu_2|_{U'}$ are each mutually absolutely continuous with respect to the law of the corresponding restriction of the $\gamma$-LQG measure induced by a (free or zero boundary) GFF on $\mcl D$. 

Let $\eta$ be a space-filling SLE$_{\kappa}$ in $\mcl D$ whose starting and target points are in $\ol{\mcl D} \setminus U'$ (if $\mcl D\not=\BB C$) or are each equal to $\infty$ (if $\mcl D=\BB C$), sampled independently of $\mu_1,\mu_2$ and parameterized by Lebesgue measure. Let $\eta_1 : [0,T_1]\rta \ol{\mcl D}$ (resp.\ $\eta_2 : [0,T_2]\rta\ol{\mcl D}$) be the curve obtained by parameterizing $\eta$ by $\mu_1$- (resp.\ $\mu_2$-mass). 

For $i\in \{1,2\}$, let $\indshift_i \in [0,1]$ be a random variable determined by $\mu_i|_{\mcl D\setminus U'}$ and for $i\in \{1,2\}$, ${1/n} \in (0,1]$, and $z\in \ol{\mcl D}$, let $H_{z,i}^{1/n}$ be the cell $\eta([x-{1/n},x])$ where $x = x_{z,i}^{1/n} \in \frac{1}{n} \BB Z+ \frac{1}{n} \indshift_i$ is chosen so that $z \in \eta([x-{1/n},x] \cap [0,T_i])$ (we make an arbitrary choice if there is more than one such $x$).

\begin{defn}[Uniform absolute continuity in the limit] \label{def-umac-limit}
Let $\{P_1^n\}_{n \in \BB N}$ and $\{P_2^n\}_{n \in \BB N}$ be two sequences of probability measures such that for each $n\in\BB N$, $P_1^n$ and $P_2^n$ are defined on a common measure space $(\Omega^n , \mcl M^n)$. We say that $\{P_1^n\}_{n\in\BB N}$ is \emph{uniformly absolutely continuous as $n\rta\infty$} with respect to $\{P_2^n\}_{n\in\BB N}$ if for each $p \in (0,1)$, there exists $q = q(p) \in (0,1)$ and $n_* = n_*(p) > 0$ such that for each $n\geq n_*$ and each measurable event $E \subset \Omega^n $ with $P_2^n[E] \leq p$, we have $P_1^n[E] \leq q$. We say that $\{P_1^n\}_{n \in \BB N}$ and $\{P_2^n\}_{n \in \BB N}$ are \emph{uniformly mutually absolutely continuous} as $n\rta\infty$ if each of $\{P_1^n\}_{n \in \BB N}$ and $\{P_2^n\}_{n \in \BB N}$ is uniformly absolutely continuous with respect to each other as $n\rta\infty$. 
\end{defn}

\begin{lem} \label{lem-cell-abs-cont}
The laws of the cell configurations
\eqb \label{eqn-sle-cell-config}
\left\{ H_{z,1}^{1/n} \right\}_{z\in U} \quad \op{and} \quad \left\{H_{z,2}^{1/n} \right\}_{z\in U}
\eqe 
are uniformly mutually absolutely continuous as $n\rta\infty$. 
\end{lem}
\begin{proof}
Let $\sigma_0$ be the first time $\eta$ enters $\ol U'$. Inductively, if $k \in\BB N$ and $\sigma_{k-1}$ has been defined, let $\tau_k$ be the first time after $\sigma_{k-1}$ at which $\eta$ enters $\ol U$ (or $\tau_k=\infty$ if no such time exists) and let $\sigma_k$ be the first time after $\tau_k$ at which $\eta$ exits $U'$ (or $\sigma_k=\infty$ if no such time exists). 
Since $\eta$ is continuous, the integer $K+1:= \min\{k\in\BB N:  \tau_k = \infty\}$ is a.s.\ finite. 
 
For $k\in [1,K]_{\BB Z}$, let $A_k := \eta([ \sigma_{k-1} ,\tau_k]) \subset \mcl D \setminus U$ and let $\mcl A := \bigcup_{k=1}^K A_k$.
For $n\in\BB N$, $i\in \{1,2\}$, and $k\in[1,K]_{\BB Z}$ let $s_{k,i}^n \in [0,1/n)$ be chosen so that $\mu_i(A_k) - s_{k,i}^{1/n} \in \frac{1}{n} \BB Z + \frac{1}{n} \indshift_i$. 
On the event that each of the cells $H_{z,i}^{1/n}$ for $z\in U$ is contained in $U'$ (which holds with probability tending to 1 as $n\rta\infty$ since $\eta$ is continuous and $\mu_i$ assigns positive mass to every open set), the cell configuration $\left\{ H_{z,1}^{1/n}  \right\}_{z\in U}$ is a.s.\ determined by $\eta$, $\mu_i|_{U' \setminus \mcl A}$, and $\{s_{k,i}^n \}_{k\in [1,K]_{\BB Z}}$, in a deterministic manner which is the same for each $i\in \{1,2\}$. By assumption, the conditional laws of $\mu_1|_{U'\setminus \mcl A}$ and $\mu_2|_{U'\setminus \mcl A}$ given $\eta$ are mutually absolutely continuous, so we just need to check that the conditional laws of $\{s_{k,i}^n\}_{k\in [1,K]_{\BB Z}}$ given $\eta$ and $\mu_i|_{U' \setminus \mcl A}$ are mutually absolutely continuous as $n \rta\infty$. 

We will show that each of these conditional laws is mutually absolutely continuous with respect to Lebesgue measure on $[0,1/n]^K$. 
Indeed, the law of the restriction of $\mu_i$ to $U'$ is mutually absolutely continuous with respect to a $\gamma$-LQG measure by hypothesis and each $A_k$ intersects $U'$ in an open set (this is because $\eta$ covers an open subset of $\mcl D$ in any non-trivial time interval). Using this, one can check via the spatial Markov property of the GFF that for each $i\in \{1,2\}$, under the conditional law given $\eta$ and $\mu_i|_{\mcl D\setminus ( \mcl A \cap U')}$, the law of the $K$-tuple $(\mu_i(A_1),\dots,\mu_i(A_K))$ is mutually absolutely continuous with respect to Lebesgue measure on some open subset of $(0,\infty)^K$ (possibly depending on $\eta$ and $\mu_i|_{\mcl D\setminus \mcl A}$). From this and the fact that $\indshift_i$ is determined by $\mu_i|_{\mcl D\setminus U'}$, it follows that the conditional law of $\{s_{k,i}^n \}_{k\in [1,K]_{\BB Z}}$ is mutually absolutely continuous with respect to Lebesgue measure on $[0,1/n]^K$ as $n\rta\infty$. 
\end{proof}

\subsection{Random walk on the embedded mated-CRT map converges to Brownian motion}
\label{sec-rw-conv}

In this subsection we use Proposition~\ref{prop-0cone} and absolute continuity arguments to show that the random walk on the a priori embedding of the mated-CRT map with the plane, sphere, or disk topology, as described in Section~\ref{sec-sg-def}, converges to Brownian motion modulo time parameterization in the quenched sense. Recall that $\eta $ is a space-filling SLE$_{\kappa}$ curve on an appropriate type of $\gamma$-LQG surface, parameterized by $\gamma$-LQG mass. Also recall that $\mcl G^{1/n}$ is the corresponding mated-CRT map, or equivalently the adjacency graph of cells $\eta([x-1/n,x])$, with two cells considered adjacent if they intersect along a non-trivial connected boundary arc. 

\begin{thm} \label{thm-rw-conv}
Suppose we are in one of the three settings listed at the beginning of Section~\ref{sec-sg-def}. For $z\in\BB C$ and $n\in\BB N$, write $x_z^n$ for the smallest $x \in \frac{1}{n} \BB Z$ for which $z\in \eta([x-1/n,x])$. 
For $z\in \ol{\mcl D}$ and $n\in\BB N$, let $Y^{z,n}  $ be a random walk on $\mcl G^{1/n}$ started from $Y_0^{z,n} = x_z^n$, and stopped when it hits $\bdy\mcl G^{1/n}$ in the disk case and run for all time in the sphere and plane cases.
Let $\wh Y^{z,n} := \eta(Y^{z,n})$ and extend $\wh Y^{z,n}$ to a function from $[0,\infty)$ to $\ol{\mcl D}$ by piecewise linear interpolation at constant speed.  

For each deterministic compact set $K\subset \ol{\mcl D}$, the supremum over all $z\in K$ of the Prokhorov distance between the conditional law of $\wh Y^{z,n} $ given $(h,\eta  )$ and the law of a standard two-dimensional Brownian motion started from $z$ (and stopped when it hits $\bdy \BB D$ in the disk case), with respect to the metric on curves viewed modulo time parameterization (i.e., the metric~\eqref{eqn-cmp-metric} in the disk case or the metric~\eqref{eqn-cmp-metric-loc} in plane or sphere cases) converges to 0 in probability as $n \rta\infty$.  
\end{thm} 
 
We note that, as discussed just after the statement of Theorem~\ref{thm-quenched-clt0}, one can easily transfer Theorem~\ref{thm-rw-conv} to other $\gamma$-LQG surfaces decorated by SLE$_{\kappa}$-type curves using local absolute continuity. 

We now commence with the proof of Theorem~\ref{thm-rw-conv}. We first consider the case when $h$ is replaced by a 0-quantum cone, which follows almost immediately from Proposition~\ref{prop-0cone}.
Note that this is the point in the proof where we move from a.s.\ convergence to convergence in probability since we switch from spatial scaling to scaling the lengths of the time intervals corresponding to the cells.

\begin{lem} \label{lem-rw-conv-0cone}
Theorem~\ref{thm-rw-conv} is true in case~\ref{item-rw-conv-cone} with a 0-quantum cone in place of the $\gamma$-quantum cone and with the $\frac{1}{n} \BB Z$ replaced by $\frac{1}{n} \BB Z +\frac{1}{n} \indshift$ in~\eqref{eqn-sg-vertex-set}, where $\indshift \in [0,1]$ is a uniform random variable independent from everything else. 
\end{lem}
\begin{proof}
By Brownian scaling the statement of the lemma is invariant under the operation of changing the embedding $h$ (i.e., replacing $h$ by $h(r\cdot)  + Q\log r$ for some possibly random $r>0$), so we can assume without loss of generality that $h$ has the circle-average embedding, as described in Section~\ref{sec-lqg-prelim} and~\cite[Definition 4.10]{wedges} (we could also, e.g., embed so that $\mu_h(\BB D ) =1$). 

Let $\mcl H$ be the cell configuration as in~\eqref{eqn-0cone-cells}. 
Proposition~\ref{prop-0cone} together with Theorem~\ref{thm-general-clt-uniform} tells us that a.s.\ the conditional law given $\mcl H$ of the random walk on $\ep\mcl H$ converges in law as $\ep\rta0$ to standard two-dimensional Brownian motion modulo time parameterization, and the convergence is uniform over all starting points in any fixed compact subset of $\BB C$. 

The graph of cells $\mcl H$ is isomorphic to the graph $\mcl G^1$ via the map $x \mapsto \eta([x-1,x])$. However, $\mcl G^{1/n}$ is typically \emph{not} isometric to $\ep  \mcl H$ for any $\ep > 0$. Nevertheless, we can use scale invariance \emph{in law} to compare $\mcl G^{1/n}$ with $\ep \mcl H$ for a certain random choice of $\ep $, as we now explain. 

For $b > 0$, let $R_b > 0$ be as in~\eqref{eqn-mass-hit-time} and let $h^b := h(R_b\cdot) + Q\log R_b - \gamma^{-1} \log b$, so that by~\eqref{eqn-cone-scale}, $h^b\eqD h$. If we set $\eta^b := R_b^{-1}\eta(b \cdot)$, then the $\gamma$-LQG coordinate change formula~\cite[Proposition 2.1]{shef-kpz} implies that $\eta^b$ is parameterized by $\gamma$-LQG mass with respect to $h^b$. Therefore, $(h^b,\eta^b) \eqD (h,\eta)$. 
 
If we define $\mcl H^b$ as above with $(h^b,\eta^b)$ in place of $(h,\eta)$, then $\mcl H^b \eqD \mcl H$ and $\mcl G^b$ is isometric to $R_b \mcl H^b$ via $x\mapsto R_b \eta^b([x-b,x])$. It is easily seen from the definition of the 0-quantum cone in Section~\ref{sec-lqg-prelim} that a.s.\ $R_b \rta 0$ as $b\rta 0$. Setting $b = 1/n$, we now get the desired convergence in probability from Proposition~\ref{prop-0cone} and Theorem~\ref{thm-general-clt-uniform}.
\end{proof}

We next transfer from the 0-quantum cone to the $\gamma$-quantum cone using local absolute continuity.

\begin{proof}[Proof of Theorem~\ref{thm-rw-conv} in case~\ref{item-rw-conv-cone}]
As in the proof of Lemma~\ref{lem-rw-conv-0cone}, we work with the circle-average embedding of the $\gamma$-quantum cone, which has the property that $h|_{\BB D}$ agrees in law with the corresponding restriction of a whole-plane GFF plus $-\gamma \log |\cdot|$, normalized so that its circle average over $\bdy\BB D$ is 0. We also let $\wh h$ be the circle-average embedding of a 0-quantum cone in $(\BB C, 0,\infty)$, so that $\wh h|_{\BB D} \eqD (h +\gamma\log|\cdot|)|_{\BB D}$. 

The statement of the lemma is essentially a consequence of Lemma~\ref{lem-rw-conv-0cone} and local absolute continuity (in the form of~\cite[Proposition 3.4]{ig1}), but a little care is needed since we only have local absolute continuity between the laws of a $h$ and $\wh h$ on domains at positive distance from 0 (due to the $\gamma$-log singularity of $h$) and from $\bdy \BB D$ (due to our choice of embedding). Throughout the proof, the Prokhorov distance is always taken with respect to the metric on curves viewed modulo time parameterization.
 
For $\rho  > 0$ and $z\in B_\rho(0)$, let $J_\rho^{z,n}$ for $n\in\BB N$ be the exit time of the embedded walk $\wh Y^{z,n}$ from $ B_\rho(0) $. Also let $\mcl B^z$ be a standard two-dimensional Brownian motion started from $z$ and let $\tau_\rho^z$ be its exit time from $B_\rho(0)$. We need to show that for each $\rho > 0$, the supremum over all $z\in B_\rho(0)$ of the Prokhorov distance between the conditional laws of $\wh Y^{z,n}|_{[0,J_\rho^{z,n}]}$ and $\mcl B^z|_{[0,\tau_\rho^z]}$ given $(h,\eta  )$ converges to 0 in probability as $n\rta\infty$.

We first consider a radius $\rho  \in (0,1)$ and deal with the log singularity at 0. 
For $\delta \in (0,\rho)$, choose $\alpha = \alpha(\delta) \in (0,\delta)$ such that the probability that a Brownian motion started from any point of $\BB D\setminus B_\delta(0)$ hits $B_\alpha(0)$ before leaving $\BB D$ is at most $\delta$. By Lemma~\ref{lem-rw-conv-0cone} and local absolute continuity (using Lemma~\ref{lem-cell-abs-cont}) it holds with probability tending to 1 as $n \rta\infty$ that for each $z\in B_\rho(0) \setminus B_\delta(0)$, the Prokhorov distance between the conditional laws of $\wh Y^{z,n}|_{[0,J_\rho^{z,n}]}$ and $\mcl B^z|_{[0,\tau_\rho^z]}$ given $(h,\eta  )$ is at most $\delta$. 
Since the law of $\mcl B^z|_{[0,\tau_\rho^z]}$ depends continuously on $z$, the Prokhorov-distance diameter of the set of laws of the curves $\mcl B^z|_{[0,\tau_\rho^z]}$ for $z\in B_\delta(0)$ tends to 0 as $\delta\rta 0$.  
 
By the last two sentences of the preceding paragraph and the strong Markov property of $Y^{z,n}$ and of $\mcl B^z$, it holds with probability tending to 1 as $n\rta\infty$ that for each $z\in B_\delta(0)$, the Prokhorov distance between the conditional laws of $Y^{z,n}|_{[J_\delta^{z,n} ,J_\rho^{z,n}]}$ and $\mcl B^z|_{[0,\tau_\rho^z]}$ given $(h,\eta )$ is $o_\delta(1)$, at a deterministic rate depending only on $\rho$. The distance between the curves $Y^{z,n}|_{[J_\delta^{z,n} ,J_\rho^{z,n}]}$ and $Y^{z,n}|_{[0,J_\rho^{z,n}]}$, viewed modulo time parameterization, is at most $2\delta$. Sending $\delta \rta 0$ now gives the theorem statement in the case $\rho < 1$. 

The case when $\rho\geq 1$ follows from the case when $\rho \in (0,1)$ and the scale invariance property of the $\gamma$-quantum cone~\cite[Proposition 4.13(i)]{wedges}, applied similarly as in Lemma~\ref{lem-rw-conv-0cone}. 
\end{proof}

We next transfer from the case of a $\gamma$-quantum cone to the case of a special type of quantum wedge which can be naturally found inside of a $\gamma$-quantum cone. This is what allows us to extend our result from surfaces without boundary to surfaces with boundary.

\begin{lem} \label{lem-rw-conv-wedge}
Theorem~\ref{thm-rw-conv} is true with $(\mcl D, h)$ replaced by a $\frac{3\gamma}{2}$-quantum wedge $(\BB H , h , 0, \infty)$ (if $\gamma \leq \sqrt 2$) or a single bead $(\BB H , h , 0, \infty)$ of such a quantum wedge with random area $\frk a$ and left/right boundary lengths $(\frk l_L , \frk l_R)$ sampled from some probability measure which is absolutely continuous with respect to Lebesgue measure on $(0,\infty)^3$ (if $\gamma \in (\sqrt 2 , 2)$); and $\eta$ replaced by an independent chordal space-filling SLE$_{\kappa}$ from $0$ to $\infty$ in $\BB H$, parameterized by $\gamma$-quantum mass with respect to $h$ (here we stop at the exit time from $\BB H$ and use the metric~\eqref{eqn-cmp-metric}, as in case~\ref{item-rw-conv-disk}).
\end{lem} 
\begin{proof}
Let $(\BB C , \wh h , 0, \infty)$ be a $\gamma$-quantum cone and let $\wh\eta$ be an independent space-filling SLE$_{\kappa}$ from $\infty$ to $\infty$, parameterized by $\gamma$-quantum mass, as in case~\ref{item-rw-conv-cone} of Theorem~\ref{thm-rw-conv}. Let $\wh{\mcl G}^{1/n}$ for $n\in\BB N$ be its associated $1/n$-structure graph. 
  
By~\cite[Theorem~1.9]{wedges}, the law of the quantum surface $\wh{\mcl S} := (\wh\eta([0 , \infty)), \wh h|_{\wh \eta([0 , \infty))} , 0, \infty)$ is that of a $\frac{3\gamma}{2}$-quantum wedge. Furthermore, the curve $\wh \eta|_{[0 ,\infty)}$ is a chordal space-filling SLE$_{\kappa}$ going between the two marked points of $\wh{\mcl S}$ (if $\gamma \leq \sqrt 2$) or a concatenation of independent chordal space-filling SLE$_{\kappa}$'s in the beads of $\wh{\mcl S}$ (if $\gamma \in (\sqrt 2 , 2)$). 
  
By the $\gamma$-quantum cone case, the conditional law given $(\wh h,\wh\eta)$ of the linearly interpolated image under $\wh\eta$ of the simple random walk on $\wh{\mcl G}^{1/n}$ stopped upon exiting $\wh\eta([0,\infty))$ converges in probability to Brownian motion modulo time parameterization as $n\rta\infty$, uniformly over all choices of starting points in any fixed compact subset of $\wh\eta([0,\infty))$. 
Combining this with the conformal invariance of Brownian motion yields the statement of the lemma in the case $\gamma \leq \sqrt 2$. 
 
In the case when $\gamma \in (\sqrt 2 , 2)$, we sample $U$ uniformly from the uniform measure on $[0,1]$ and let $\wh{\mcl B} $ be the first bead of $\wh{\mcl S}$ with quantum mass at least $U$ (so that $\wh{\mcl B}$ is a quantum surface with the topology of the disk). The conditional law of $\wh{\mcl B}$ given $U$ is that of a sample from the intensity measure on beads of a $\frac{3\gamma}{2}$-quantum wedge conditioned to have area at least $U$. Using the peanosphere construction~\cite[Theorem~1.9]{wedges}, it is easy to see that the conditional law of the area and left/right boundary lengths of $\wh{\mcl B}$ given $U$ is mutually absolutely continuous with respect to Lebesgue measure on $(U,\infty) \times (0,\infty)^2$, so the marginal law of the area and left/right boundary lengths of $\wh{\mcl B}$ is mutually absolutely continuous with respect to Lebesgue measure on $(0,\infty)^3$. 
We then conclude exactly as in the case $\gamma \leq \sqrt 2$ except with $U$ used in place of $\wh\eta([0,\infty))$. 
\end{proof}

\begin{proof}[Proof of Theorem~\ref{thm-rw-conv}]
The case~\ref{item-rw-conv-cone} of a $\gamma$-quantum cone was treated above.
The case~\ref{item-rw-conv-sphere} of a unit-area quantum sphere decorated by a whole-plane space-filling SLE$_{\kappa}$ from $\infty$ to $\infty$ is an immediate consequence of the $\gamma$-quantum cone case and the local absolute continuity between the laws of the $\gamma$-quantum cone and the single marked quantum sphere near their respective marked points. 

Case~\ref{item-rw-conv-disk}, when $((\BB D , h , 1) , \eta)$ is a quantum disk decorated by an independent space-filling SLE$_{\kappa}$ loop, follows from Lemma~\ref{lem-rw-conv-wedge} and the local absolute continuity of (a) a quantum disk with respect to a quantum wedge (or a bead of such a wedge if $\gamma \in (\sqrt 2 , 2)$) away from the marked points and (b) a space-filling SLE$_{\kappa}$ loop with respect to a chordal space-filling SLE$_{\kappa}$ curve away from the start and end points. The former absolute continuity statement is immediate from the definitions in~\cite{wedges} and the latter absolute continuity statement can be obtained by using standard local absolute continuity results for the GFF (e.g.,~\cite[Proposition 3.4]{ig1}) applied to the imaginary geometry field $h^{\op{IG}}$ used to construct $\eta$ as in Section~\ref{sec-wpsf-prelim}. 
\end{proof}

\subsection{Convergence of the Tutte embedding}
\label{sec-embedding-proof}

\begin{proof}[Proof of Theorem~\ref{thm-tutte-conv0}]
For $n\in\BB N$, let $\mcl G^{1/n}$ be the mated-CRT map with the disk topology and let $\phi^{1/n} : \mcl V\mcl G^{1/n} \rta \ol{\BB D}$ be its Tutte embedding, as defined in Section~\ref{sec-overview}. Also let $\eta : [0,1]\rta \ol{\BB D}$ be the associated space-filling SLE$_{\kappa}$ curve parameterized by $\gamma$-LQG mass as in Proposition~\ref{prop-a-priori-embedding}, so that $x \mapsto \eta(x)$ is the a priori embedding of $\mcl G^{1/n}$. Theorem~\ref{thm-rw-conv} 
implies that the maximum over all vertices $x \in \mcl V\mcl G^{1/n}$ of the Prokhorov distance between the $ \mcl G^{1/n}$-harmonic measure on $\eta(\bdy\mcl G^{1/n})$ as viewed from $x$ and the Euclidean harmonic measure on $\bdy\BB D$ as viewed from $\eta(x)$ tends to 0 in probability as $n\rta\infty$. By this and the definition of $\phi^{1/n}$, 
\eqb \label{eqn-tutte-conv}
\max_{x\in\mcl V\mcl G^{1/n}} |\phi^{1/n}(x) - \eta(x)| \rta 0 \quad \text{in probability}.  
\eqe
The uniform convergence statement for curves is an immediate consequence of~\eqref{eqn-tutte-conv}, the claimed measure convergence statement follows from~\eqref{eqn-tutte-conv} and the fact that $\eta$ is parameterized by $\mu_h$-mass, and the random walk convergence follows from~\eqref{eqn-tutte-conv} and Theorem~\ref{thm-rw-conv}. 
\end{proof}

\begin{remark}[Tutte embeddings in the whole-plane and sphere cases] \label{remark-other-embedding}
One can also define Tutte embeddings of the mated-CRT map with the topology of the plane or sphere for which an analog of Theorem~\ref{thm-tutte-conv0} holds, but the definition is more complicated since these maps do not have a boundary. We explain this in the plane case (the sphere case is similar). Let $\{\mcl G^{1/n} \}_{n\in\BB N}$ be the mated-CRT maps with the whole-plane topology, as in Section~\ref{sec-mated-crt-map}, and let $\eta$ be the corresponding whole-plane space-filling SLE$_{\kappa}$ parameterized by $\gamma$-LQG mass, as in Section~\ref{sec-sg-def}. Assume that we have chosen the embedding of our curve-decorated quantum surface (i.e., rotated and scaled) in such a way that $\eta(1)=1$. 

For $T>1$, let $\mcl G^{1/n}_T$ be the sub-graph of $\mcl G^{1/n}$ induced by $[-T,T]\cap (\frac{1}{n} \BB Z)$. Then $\mcl G^{1/n}_T$ is a planar map with boundary (the boundary corresponds to vertices which are joined by edges to vertices not in $[-T,T]\cap (\frac{1}{n} \BB Z)$), so we can define its Tutte embedding into $\BB D$ exactly as in the disk case (Section~\ref{sec-overview}). Let $\phi_T^n$ be obtained from this Tutte embedding by multiplying by a complex number so that $1 \in \mcl V\mcl G^{1/n}$ is mapped to $1\in\BB C$, and define $\phi_{T }^{1/n}$ arbitrarily outside of $[-T  ,T ]\cap (\frac{1}{n} \BB Z)$. One can check using Theorem~\ref{thm-rw-conv} that if $\{T^n\}_{n\in\BB T}$ is a deterministic sequence which tends to 0 sufficiently slowly, then the uniform distance between $\phi_{T^n }^n$ and $\eta$ on any fixed interval of the form $[a,b]\cap (\frac{1}{n} \BB Z)$ tends to 0 in probability as $n\rta\infty$, and hence the conclusions of Theorem~\ref{thm-tutte-conv0} hold for the image of $\mcl G^{1/n}$ under $\phi_{T^n}^{1/n}$.\footnote{To check this, let $\wt\psi_T$ be the conformal map from the connected component of the interior of $\eta([-T,T])$ containing 0 to $\BB D$ which sends 0 to 0 and let $\psi_T$ be obtained by multiplying $\wt\psi_T$ by a complex number chosen so that $\psi_T(\eta(1) ) = 1$. 
Theorem~\ref{thm-rw-conv} shows that for each fixed $T>1$, the uniform distance between $\phi_T^n$ and $x\mapsto (\psi_T\circ \eta)(x)$ tends to 0 in probability as $n\rta\infty$. If $T>>T'$, then $\psi_T$ is nearly constant on $\eta([-T',T'])$, so the desired statement is indeed true if we make $T^n$ tend to $\infty$ sufficiently slowly.}
 
We also note that Proposition~\ref{prop-0cone} together with~\cite[Theorem~1.16]{gms-random-walk} shows that a variant of Theorem~\ref{thm-tutte-conv0} holds for the adjacency graph of cells arising from space-filling SLE$_{\kappa}$ on a 0-quantum cone, but the discrete harmonic function $\phi_\infty$ of~\cite[Theorem~1.16]{gms-random-walk} is not an explicit functional of the graph structure. 
\end{remark}

\section{Moment bound for diameter$^2$/area and degree of cells}
\label{sec-moment}
 
Throughout this subsection we assume we are in the setting of Section~\ref{sec-0cone}, so that $(\BB C ,h , 0, \infty)$ is a 0-quantum cone with the circle average embedding and $\eta$ is a whole-plane space-filling SLE$_{\kappa }$ sampled independently from $h$ and then parameterized by $\gamma$-LQG mass with respect to $h$ in such a way that $\eta(0)= 0$. We define the cell configuration $\mcl H$ as in~\eqref{eqn-0cone-cells}. The goal of this section is to check hypothesis~\ref{item-hyp-moment} of Theorem~\ref{thm-general-clt-uniform} for the cell configuration $\mcl H$ (which is needed in the proof of Proposition~\ref{prop-0cone}). In fact, we will prove the following stronger statement.

\begin{thm} \label{thm-moment}
Let $\mcl H$ be as above and let $H_0$ be the cell of $\mcl H$ containing the origin. Then
\eqb \label{eqn-area-diam-moment}
\BB E\left[ \left( \frac{\op{diam}\left(H_0 \right)^2}{\op{area}\left( H_0 \right) } \right)^p \right] < \infty ,\quad \forall p \geq 1 \quad \text{and} \quad
\eqe  
\eqb \label{eqn-deg-moment}
\BB E\left[ \op{deg}\left( H_0 \right)^p \right]  < \infty ,\quad \forall p \in [1,4/\gamma^2) .
\eqe 
\end{thm}

\begin{remark}  
It is easily seen from Brownian motion estimates that the degree of the root vertex in the whole-plane mated-CRT map has an exponential tail~\cite[Lemma 2.2]{gms-harmonic}. However, this does not help us prove~\eqref{eqn-deg-moment} since we are working with a 0-quantum cone so the origin is a Lebesgue typical point, not a quantum typical point.
\end{remark}

The proof of Theorem~\ref{thm-moment} builds on estimates for SLE and LQG from~\cite{shef-kpz,rhodes-vargas-review,ghm-kpz}. 
Roughly speaking,~\eqref{eqn-area-diam-moment} comes from that fact that $\eta$ is exponentially unlikely to travel Euclidean distance $r$ without first swallowing a Euclidean ball of radius a little bit less than $r$~\cite[Lemma 3.6]{ghm-kpz}; and the fact that such a ball is very unlikely to have smaller quantum mass than we would expect~\cite[Lemma 4.5]{shef-kpz}.  

Our estimate for $\op{deg}(H_0)$ is split into two parts: we will separately bound the ``inner degree", defined as the maximal number of disjoint $\ep$-length segments of $\eta$ which intersect $\eta([-1,1]) \supset H_0$ and which are contained in the ball $\ol B_0$ of radius $4 \op{diam} \eta([-1,1])$ centered at 0; and the ``outer degree", defined as the number of times that $\eta$ crosses between $\eta([-1,1])$ and $\bdy \ol B_0$ (see Definition~\ref{def-deg}). The inner degree is bounded by $\mu_h(\ol B_0)$, which we will show has finite moments up to order $4/\gamma^2$ using moment bounds for the $\gamma$-LQG measure~\cite[Theorem 2.11]{rhodes-vargas-review}. The outer degree can be bounded using the fact that with high probability, each crossing of $\eta$ of an annulus has to contain a Euclidean ball with radius slightly smaller than the distance across the annulus~\cite[Proposition 3.4]{ghm-kpz}. 
 
The aforementioned bounds for space-filling SLE and the LQG measure require us to work in a Euclidean ball of a fixed size, but the diameter of $\eta([-1,1])$, hence the radius of $\ol B_0$, is random. In order to get around this difficulty, we will use the radii $R_b$ defined as in~\eqref{eqn-mass-hit-time} and bound from above the probability that $\eta([-1,1])$ is contained in $B_{R_{1/C}}(0)$ and the probability that it is not contained $B_{R_{C}}(0)$ when $C$ is large (Lemma~\ref{lem-radius-segment}). We will then re-scale by $R_b^{-1}$ for a certain value of $b>0$ and work with a re-scaled version of $(h,\eta)$ which has the same law as $(h,\eta)$ due to~\eqref{eqn-cone-scale}. 
 
Our proof in fact yields a stronger version of Theorem~\ref{thm-moment} which gives analogous moment bounds for quantities which are deterministically larger than $\op{diam}(H_0)^2/\op{area}(H_0)$ and $\op{deg}(H_0)$, respectively, but are locally determined by $h$ and $\eta$ (we recall from Section~\ref{sec-cell-abs-cont} that cells of $\mcl H$ are not locally determined by $h$ and $\eta$). We define these quantities at the beginning of Sections~\ref{sec-area-diam} and~\ref{sec-deg-bound}. The moment bound for quantities which are locally determined by $h$ and $\eta$ will be important in~\cite{gms-harmonic}.

\subsection{Estimating the diameter of $\eta([-1,1])$ via circle averages}
\label{sec-radius-segment}

Throughout the rest of this section, we will use the following notation. Let $h$ the circle-average embedding of a 0-quantum cone, as above, and let $\{h_r(z)\}_{r > 0, \, z\in \BB C}$ be its circle-average process. 
For $b > 0$, let $R_b$ be as in~\eqref{eqn-mass-hit-time} and let   
\eqb \label{eqn-shift-fields}
 h^b  :=     h(R_b \cdot  )    + Q\log R_b  \quad \op{and} \quad \eta^b(t) := R_b^{-1} \eta(t) ,\: \forall t\in \BB R .  
\eqe 
By~\eqref{eqn-cone-scale}, $h^b \eqD h + \frac{1}{\gamma} \log b $, so in particular $h^b|_{\BB C}$ agrees in law with the corresponding restriction of a whole-plane GFF normalized so that its circle average over $\bdy\BB D$ is $\frac{1}{\gamma} \log b$. 
By the $\gamma$-LQG coordinate change formula~\eqref{eqn-lqg-coord}, $\mu_h(X) = \mu_{h^b}(R_b^{-1} X)$ for each Borel set $X\subset \BB C$, hence $\eta^b$ is parameterized by $\gamma$-quantum mass with respect to $h^b$.
We will also make frequent use of the times
\eqb \label{eqn-sle-hit-time}
\tau(r) := \inf\left\{ t  > 0 : |\eta(t)| \geq r \right\} \quad\op{and} \quad 
\tau^-(r) := \sup\left\{t < 0 : |\eta(t)| \geq r \right\} ,\quad \forall r > 0. 
\eqe 
The following lemma allows us to estimate the size of $\eta([-1,1])$ in terms of the radii $R_b$. 
   
\begin{lem} \label{lem-radius-segment}  
For $C >1$,  
\eqb \label{eqn-radius-segment-upper}
\BB P\left[ \eta([ 0 ,1]) \not\subset B_{R_{1/C} }(0) \right] \geq 1 - C^{-\tfrac{4}{\gamma^2} + o_C(1)}
\eqe
and 
\eqb\label{eqn-radius-segment-lower}
\BB P\left[  \eta([  - 1 , 1])  \subset B_{  R_C / 8  }(0) \right] = 1 - o_C^\infty(C) 
\eqe
as $C\rta\infty$, at a rate which depends only on $\gamma$ (see Section~\ref{sec-basic-notation} for an explanation of our usage of $o(\cdot)$ notation).  
\end{lem}

\begin{proof}
Let $h^b$ and $\eta^b$ for $b > 0$ be as in~\eqref{eqn-shift-fields}. 
We first prove~\eqref{eqn-radius-segment-upper}.
Since $\eta^{1/C}$ is parameterized by $h^{1/C}$-quantum mass, for each $C>1$, 
\allb
\BB P\left[ \eta \left( \left[ 0 , 1 \right] \right) \subset    B_{R_{1/C} }(0) \right]
 = \BB P\left[ \eta^{1/C} \left( \left[0,  1\right] \right) \subset \BB D   \right] 
 \leq \BB P\left[   \mu_{  h^{1/C} } (\BB D) >     1  \right] . \label{eqn-radius-segment-coord}
\alle
Since $h^{1/C} \eqD h + \frac{1}{\gamma} \log(1/C)$, we can apply a standard moment bound from Gaussian multiplicative chaos theory~\cite[Theorem~2.11]{rhodes-vargas-review} to get that $C \mu_{ h^{1/C}}(\BB D)$ has finite $C$-independent moments of all orders $\beta \in (0, 4/\gamma^2)$.  
By the Chebyshev inequality, we obtain that for any such $\beta$, $\BB P\left[  \mu_{  h^{1/C} } (\BB D) >   1  \right]  \preceq  C^{-\beta}$.
Combining this with~\eqref{eqn-radius-segment-coord} and sending $\beta\rta 4/\gamma^2$ yields~\eqref{eqn-radius-segment-upper}. 

To prove~\eqref{eqn-radius-segment-lower}, let $\zeta \in (0,1/100)$ be a small parameter which we will eventually send to 0 and fix $C>1$. 
We define the time $\tau(R_C/8)$ as in~\eqref{eqn-sle-hit-time}, so that $\tau(R_C/8)$ is the exit time of $\eta^{C } $ from $B_{1/8}(0)$. Since $\eta^{C }$ is parameterized by $\mu_{h^{C }}$-mass, $\tau(   R_C / 8 )  = \mu_{h^{C }}( \eta^{C }([0, \tau(  R_C / 8)]) )$. 
We will prove a lower bound for this quantum mass.  

To this end, let $M$ be the radius of the largest Euclidean ball contained in $ \eta^{C }([0,\tau(R_C/8)])$. Let $w$ be the center of such a ball with maximal radius, chosen in a manner which depends only on  $ \eta^{C} ([0, \tau(R_C/8)])$. By~\cite[Lemma~3.6]{ghm-kpz},
\eqb \label{eqn-max-rad-tail}
\BB P[M \geq C^{-\zeta} ] = 1 - o_C^\infty(C)
\eqe 
as $C\rta\infty$, at a rate depending only on $\gamma$. Since $B_M(w)$ depends only on the curve $ \eta^{C }$, viewed modulo monotone re-parameterization, this ball is independent from the field $ h^{C }$. Since $h^C \eqD h + \frac{1}{\gamma} \log C$, we can apply a standard estimate for the $\gamma$-LQG measure (see, e.g.,~\cite[Lemma~3.9]{ghm-kpz}) to get
\eqb \label{eqn-measure-tail}
\BB P\left[ \mu_{ h^{C } }(B_M(w)) \geq C^{-\zeta} M^{2+\frac{\gamma^2}{2}} e^{\gamma h^{C }_M(w)} \,|\, B_M(w) \right] = 1 -  o_C^\infty(C)
\eqe 
at a deterministic rate depending only on $\gamma$ (here $h^{C }_M(w)$ is the circle average of $h^{C }$ over $\bdy B_M(w)$). Since $h^C|_{\BB D}$ agrees in law with the corresponding restriction of a whole-plane GFF normalized so that its circle average over $\bdy\BB D$ is $\frac{1}{\gamma} \log C$, the conditional law of the circle average $ h^{C }_M(w)$ given $B_M(w)$ is Gaussian with mean $\frac{1}{\gamma} \log(C ) $ and variance at most $\log M^{-1} + O(1)$, where $O(1)$ denotes a quantity bounded above in absolute value by a universal constant. By the Gaussian tail bound,
\eqb \label{eqn-circle-avg-tail}
\BB P\left[     h^{C }_M(w) \geq   \frac{1}{\gamma} \log(C^{100\zeta}  )   \,|\, B_M(w)\right] \BB 1_{(M \geq C^{-\zeta})}
\geq \left( 1 -  C^{-\frac{(1-100\zeta)^2}{2\gamma^2 \zeta}   }  \right) \BB 1_{(M \geq C^{-\zeta})} 
\eqe 
at a rate depending only on $\gamma$. 
By using~\eqref{eqn-max-rad-tail} and~\eqref{eqn-circle-avg-tail} to bound the quantity $ C^{-\zeta} M^{2+\frac{\gamma^2}{2}} e^{\gamma h^{C }_M(w)}$ appearing in~\eqref{eqn-measure-tail} from below and then sending $\zeta \rta 0$, we get
\eqb \label{eqn-radius-exit-mass}
\BB P\left[  \mu_{ h^{C } }(B_M(w)) \geq  1   \right] = 1 -  o_C^\infty(C)  .
\eqe  
Therefore,
\alb
\BB P\left[   \tau (R_C/8 )     \geq 1  \right]  
&= \BB P\left[    \mu_{ h^{C } }(\eta^C([ 0 ,  \tau (R_C/8 )  ]))    \geq  1  \right] \\
&\geq \BB P\left[   \mu_{ h^{C } }(B_M(w))  \geq 1   \right] \quad \text{(since $B_M(w) \subset \eta^{C }([0,\tau(R_C/8)])$)} \notag \\
&= 1 -  o_C^\infty(C)  \quad \text{(by~\eqref{eqn-radius-exit-mass})} .
\ale
By the definition~\eqref{eqn-sle-hit-time} of $\tau(R_C/8)$, if $\tau(R_C/8) \geq 1$ then $\eta([0,1]) \subset B_{R_C/8}(0)$. By the forward/reverse symmetry of $\eta$, this implies~\eqref{eqn-radius-segment-lower}.
\end{proof}

\subsection{Moment bound for diameter$^2$/area}
\label{sec-area-diam}

We will now prove~\eqref{eqn-area-diam-moment} of Theorem~\ref{thm-moment}. In fact, we will prove a moment bound for the larger quantity
\eqb \label{eqn-area-diam-localize}
\op{DA}(H_0) := \sup\left\{ \frac{\op{diam}\left(\eta([s-1 , s   ]) \right)^2}{\op{area}\left(\eta([s-1 , s   ]) \right) } \,:\,  s \in [0,1] \right\} .
\eqe 
In words, $\op{DA}(H_0)$ is equal to the maximum ratio of the squared diameter to the area over all of the segments of $\eta$ with unit quantum mass which contain $0$.
By definition, the cell $H_0$ is one of these segments of $\eta$ (namely, the one with $s =\indshift$), so 
\eqb \label{eqn-area-diam-localize-bound}
 \frac{\op{diam}(H_0)^2}{\op{area}(H_0)} \leq   \op{DA}(H_0) .
\eqe  
 
\begin{prop} \label{prop-area-diam}
Define $\op{DA}(H_0)$ as in~\eqref{eqn-area-diam-localize}. Then
\eqb \label{eqn-area-diam}
 \BB E\left[     \op{DA}(H_0)^p \right]  < \infty ,\quad \forall p \geq 1 .
\eqe  
\end{prop}  
\begin{proof}
Let $\zeta\in (0,1)$ be a small parameter which we will eventually send to $0$. We will apply~\cite[Lemma 3.6]{ghm-kpz} (which says that $\eta$ is unlikely to cross an annulus without absorbing a Euclidean ball of radius comparable to the distance across the annulus) at several scales and use Lemma~\ref{lem-radius-segment} to argue that with high probability, if $s \in [0,1]$, then $\eta([s-1 , s])$ has to cross one of these scales. We will then integrate to convert a probability estimate to a moment estimate. 
 
The random variables $R_{C^\zeta}$ and $R_{1/C^{1/\zeta}}$ for $C > 1$ depend only on $ h$, so are independent from the curve $ \eta$, viewed modulo monotone re-parameterization. By~\cite[Lemma~3.6]{ghm-kpz} and a union bound over $k$, we have (in the notation~\eqref{eqn-sle-hit-time})
\alb
&\BB P\left[ \op{area}\left( \eta([ 0 ,\tau(  e^k ) ]  ) \right) \geq C^{-1} e^{2k} ,\: \forall k \in \left[  \log R_{1/C^{1/\zeta}} -1  ,   \log R_{C^\zeta}     +1 \right]_{\BB Z} \,|\, h \right]  \\
&\qquad \geq 1 - \left( \log \left( R_{C^\zeta} / R_{1/C^{1/\zeta}} \right)+2\right) e^{-c_0 C^{1/2}}            .
\ale
for $c_0 > 0$ a constant depending only on $\gamma$. By combining this with Lemma~\ref{lem-cone-hit-tail} (applied with $\alpha = 0$ and, e.g., $e^{C^{1/4}}$ in place of $C$), we get that except on an event of probability $o_C^\infty(C)$ (at a rate depending only on $\zeta$ and $\gamma$), 
\eqb \label{eqn-exp-scale-area}
 \op{area}\left( \eta([ 0 ,\tau(  e^k ) ]) \right) \geq C^{-1} e^{2k} ,\quad \forall k \in \left[  \log R_{1/C^{1/\zeta}} -1   ,   \log R_{C^\zeta}   +1  \right]_{\BB Z}.
\eqe 
By the invariance of the law of $\eta$ under time reversal, it also holds except on an event of probability $o_C^\infty(C)$ that
\eqb \label{eqn-exp-scale-area'}
 \op{area}\left( \eta([\tau^-(e^k), 0  ]) \right) \geq C^{-1} e^{2k} ,\quad \forall k \in \left[  \log R_{1/C^{1/\zeta}} -1   ,   \log R_{C^\zeta}  +1  \right]_{\BB Z} 
\eqe
where here $\tau^-(\cdot)$ is as in~\eqref{eqn-sle-hit-time}.
By Lemma~\ref{lem-radius-segment} and the invariance of the law of $\eta$ under time reversal, it holds with probability $1-C^{-\frac{4 }{\gamma^2 \zeta} + o_C(1)}$ (at a rate depending only on $\zeta$ and $\gamma$) that 
\eqb \label{eqn-use-free-gff-mass}
\eta([0, 1/2]) ,\eta([-1/2 , 0]) \not\subset B_{R_{1/C^{1/\zeta}}}(0)  
\quad \op{and} \quad   
\eta([ -1 ,1]) \subset B_{R_{C^\zeta }}(0)   .
\eqe  

Suppose now that~\eqref{eqn-exp-scale-area},~\eqref{eqn-exp-scale-area'}, and~\eqref{eqn-use-free-gff-mass} all hold (which happens with probability at least $1-C^{-\frac{4 }{\gamma^2 \zeta} + o_C(1)}$). We seek to bound the quantity $\op{DA}(H_0)$ from~\eqref{eqn-area-diam-localize}. 
To this end, suppose $s\in [0,1]$. Then
\alb
  \eta([s - 1 , s]) \subset \eta([-1 , 1] )  
  \quad \op{and}\quad
 \eta([ s - 1 , s]) \not\subset  \eta([-1/2, 1/2] ) .
\ale
By~\eqref{eqn-use-free-gff-mass}, if we let $k$ be the smallest integer for which $\eta([ s -1 ,  s]) \subset B_{e^{k+1}}(0)$, then 
\eqbn
k\in \left[  \log R_{1/C^{1/\zeta}} -1   ,    \log R_{C^\zeta}    +1  \right]_{\BB Z} .
\eqen
Furthermore, by the definition of $k$ either $ \eta([  0 , \tau(  e^k ) ]) $ or $ \eta([ \tau^-(e^k) , 0  ]) $ is contained in $ \eta([ s - 1 , s])$. By~\eqref{eqn-exp-scale-area} and~\eqref{eqn-exp-scale-area'},
\eqb \label{eqn-area-diam-implication}
\op{area}\left(  \eta([ s - 1 ,  s]) \right) \geq e^{-2} C^{-1} \op{diam}\left(  \eta([ s -1 ,  s])  \right)^2 .
\eqe 
Therefore, for $C>1$,
\eqbn
\BB P\left[ \op{DA}(H_0)  >   e^2 C \right]  \leq C^{-\tfrac{4 }{\gamma^2 \zeta} + o_C(1)}  .
\eqen
Sending $\zeta \rta 0$, then multiplying this estimate by $p C^{p-1}$ for $p \geq 1$ and integrating from 1 to $\infty$ with respect to $C$ shows that $\BB E[ \op{DA}(H_0)^p ] < \infty$.
\end{proof}

\subsection{Moment bound for degree}
\label{sec-deg-bound}

As noted just after Theorem~\ref{thm-moment}, we will prove our moment bound for $\op{deg}(H_0)$ in two parts.
Define the closed ball
\eqb \label{eqn-localize-set}
\ol B_0 := B_{4 \op{diam}\left( \eta([-1,1]) \right) }(0) .
\eqe  
We will define two quantities which count the number of unit-$\mu_h$ mass segments of $\eta$ which intersect $\eta([-1,1])$ and are contained in and not contained in $\ol B_0$, respectively, whose sum provides an upper bound for $\op{deg}(H_0)$. 

\begin{defn} \label{def-deg}
\begin{itemize}
\item \textbf{Inner degree.} Let $\op{deg}_{\op{in}}(H_0)$ be the largest number $N\in\BB N$ with the following property: there is a collection of $N$ intervals $\{[a_j , b_j]\}_{j\in [1,N]_{\BB Z}}$ which intersect only at their endpoints, each of which has length $b_j  -a_j =1$, satisfies $\eta([a_j ,b_j]) \subset \ol B_0$, and is such that $\eta([a_j,b_j])\cap \eta([-1,1]) \not=\emptyset$.
We note that
\eqb \label{eqn-deg-in-upper}
\op{deg}_{\op{in}}(H_0) \leq  \mu_h(\ol B_0) .
\eqe 
\item \textbf{Outer degree.} Let $\op{deg}_{\op{out}}(H_0)$ be the largest number $N' \in\BB N$ with the following property: there is a collection of $N'$ intervals $\{[a_j , b_j]\}_{j\in [1,N']_{\BB Z}}$ which intersect only at their endpoints such that for each $j\in [1,N']_{\BB Z}$, $\eta((a_j,b_j)) $ is contained in the interior of $\ol B_0$, one of the endpoints $\eta(a_j)$ or $\eta(b_j)$ is contained in $\bdy \ol B_0$, and the other endpoint is contained in $\eta([-1,1])$.  
\end{itemize}
\end{defn} 

Since $H_0 \subset \eta([-1,1])$, the set of intervals $ [x-1 , x] $ for $x\in \BB Z +  \indshift$ whose corresponding cell $\eta([x-1,x]) \in \mcl H$ is adjacent to $H_0$ and contained in $\ol B_0$ is a collection as in the definition of $\op{deg}_{\op{in}}(H_0)$, so the number of such $x$ is at most $\op{deg}_{\op{in}}(H_0)$. Similarly, the number of $x\in \BB Z +  \indshift$ such that $\eta([x-1,x])$ is adjacent to $H_0$ in $\mcl H$ and $\eta([x-1,x])\not\subset \ol B_0$ is at most $\op{deg}_{\op{out}}(H_0)$, since for any such $x$ the cell $\eta([x-1,x])$ contains a different interval $[a_j,b_j]$ as in the definition of $\op{deg}_{\op{out}}(H_0)$. Therefore,
\eqb \label{eqn-deg-localize-bound}
\op{deg}\left(H_0\right) \leq   \op{deg}_{\op{in}}(H_0) +   \op{deg}_{\op{out}}(H_0)    .
\eqe  
 
\begin{prop} \label{prop-fixed-pt-deg} 
Define $\op{deg}_{\op{in}}(H_0)$ and $\op{deg}_{\op{out}}(H_0)$ as in Definition~\ref{def-deg}. Then
\eqb \label{eqn-fixed-pt-deg-in}
\BB E\left[   \op{deg}_{\op{in}}(H_0)^p \right] < \infty ,\quad \forall p \in [1,4/\gamma)^2 \quad \text{and} \quad
\eqe   
\eqb \label{eqn-fixed-pt-deg-out}
\BB E\left[    \op{deg}_{\op{out}}(H_0)^p \right] < \infty ,\quad \forall p \geq 1 .
\eqe    
\end{prop}
 
\begin{proof}[Proof of~\eqref{eqn-fixed-pt-deg-in} of Proposition~\ref{prop-fixed-pt-deg}]
In light of~\eqref{eqn-deg-in-upper}, we need an upper bound for the $\mu_h$-mass of the ball $ \ol B_0 $ defined in~\eqref{eqn-localize-set}.
Fix $\zeta \in (0,1)$ (which we will eventually send to 0) and let $C  > 1$. By~\eqref{eqn-radius-segment-lower} of Lemma~\ref{lem-radius-segment} and the forward/reverse symmetry of the law of $\eta$, it holds with probability at least $1-o_C^\infty(C)$ (at a rate depending only on $  \zeta$ and $\gamma$) that $\eta([  -1  ,   1]) \subset B_{  R_{C^\zeta} /8}(0)$. If this is the case, then the ball $\ol B_0 $ of~\eqref{eqn-localize-set} is contained in $ B_{R_{C^\zeta} }(0)$, so by~\eqref{eqn-deg-in-upper} it holds with probability at least $1-o_C^\infty(C)$ that
\eqb \label{eqn-in-deg-contain}
\op{deg}_{\op{in}}(H_0) \leq  \mu_h( B_{R_{C^\zeta } }(0) )  =  \mu_{h^{C^\zeta}}(\BB D) .
\eqe  
By~\cite[Theorem~2.11]{rhodes-vargas-review}, $C^{-\zeta} \mu_{ h^{C^\zeta } }( \BB D )$ has finite, $C $-independent moments of all orders $\beta \in (0, 4/\gamma^2)$  
so by the Chebyshev inequality, for any such $\beta$ and any $C > 1$,
\eqb  \label{eqn-in-deg-chebyshev}
\BB P\left[ \op{deg}_{\op{in}}(H_0) > C \right]  
\leq \BB P\left[  \mu_{h^{C^\zeta}}( \BB D )  >C   \right]  + o_C^\infty(C)
\preceq   C^{\beta(1-\zeta)}  
\eqe  
with the implicit constant depending only on $\beta$ and $\gamma$. Sending $\beta \rta 4/\gamma^2$ and $\zeta\rta 0$, then multiplying the resulting estimate by $C^p$ for $p\in [1,4/\gamma^2)$ and integrating from 1 to $\infty$ with respect to $C$ shows that~\eqref{eqn-fixed-pt-deg-in} holds.
\end{proof}

Next we turn our attention to the outer localized degree $\op{deg}_{\op{out}}(H_0)$. The idea of the proof is to apply~\cite[Proposition 3.4]{ghm-kpz}---which says that every segment of $\eta$ contained in $\BB D$ much contain a Euclidean ball of radius slightly smaller than the diameter of the segment---at several scales; and use Lemma~\ref{lem-radius-segment} to argue that with high probability the diameter of $\eta([-1,1])$ is comparable to the size of one of these scales. For the proof, we will work on the regularity event defined in the following lemma (the parameter $\zeta$ will eventually be sent to 0).

\begin{lem} \label{lem-deg-segment-event}
Let $\zeta \in (0,1)$ and for $C> 1$, let $G_C = G_C(\zeta)$ be the event that the following is true.
\begin{enumerate}
\item $ \eta^{C^\zeta }\left(\left[- 1 ,   1     \right]\right)     \subset  B_{1/8}(0)$. \label{item-deg-segment-upper}
\item $\eta^{C^\zeta}\left(\left[-  1 ,    1        \right]\right) \not\subset B_{C^{-1/\zeta}}(0)$. \label{item-deg-segment-lower}
\item For each $k \in [0 , \log_8 C^{1/\zeta}+1]_{\BB Z}$ and each $a,b\in \BB R$ with $\eta^{C^\zeta }([a,b]) \subset B_{8^{-k}}(0)$ and $\op{diam}(\eta^{C^\zeta }([a,b])) \geq C^{- (1-\zeta)/2 } 8^{-k}$, the set $\eta^{C^\zeta }([a,b])$ contains a Euclidean ball of radius at least $C^{-1/2} 8^{-k}$.  \label{item-deg-segment-ball}
\end{enumerate}
Then
\eqb \label{eqn-deg-segment-event}
\BB P\left[G_C \right] \geq 1 - C^{\frac{4}{\gamma^2}  \zeta-1/\zeta + o_C(1)} ,
\eqe 
at a rate depending only on $\zeta$ and $\gamma$. 
\end{lem}
\begin{proof}
It follows from~\eqref{eqn-radius-segment-lower} of Lemma~\ref{lem-radius-segment} that condition~\ref{item-deg-segment-upper} in the definition of $G_C$ holds except on an event of probability $o_C^\infty(C)$.

To show that condition~\ref{item-deg-segment-lower} also holds with high probability, we first use that $h^{C^\zeta} \eqD h  +\frac{1}{\gamma} \log C^\zeta$ and basic moment estimates for the $\gamma$-LQG measure (see, e.g.,~\cite[Lemma~5.2]{ghm-kpz}) to get that for $\beta \in (0,4/\gamma^2)$ and $r\in (0,1)$, 
\eqb \label{eqn-deg-ball-moment}
\BB E\left[ \left( C^{-\zeta} \mu_{h^{C^\zeta }}\left( B_r(0) \right) \right)^\beta \right] \preceq r^{  f(\beta)/2}
\eqe 
where $f(\beta) = (2+\tfrac{\gamma^2}{2}) \beta - \tfrac{\gamma^2}{2} \beta^2 $ and the implicit constant depends only on $\gamma$ and $\beta$.  
Using~\eqref{eqn-deg-ball-moment}, the Chebyshev inequality, and the fact that $\mu_{h^{C^\zeta }}(\eta^{C^\zeta }\left(\left[-   1 , 1      \right]\right) ) =  2$, we get that for each $\beta \in (0,4/\gamma^2)$, 
\allb \label{eqn-deg-segment-small}
\BB P\left[\eta^{C^\zeta }\left(\left[-  1 ,  1      \right]\right) \subset B_{C^{-1/\zeta}}(0) \right]
 \leq \BB P\left[ \mu_{h^{C^\zeta }}\left( B_{C^{-1/\zeta}}(0) \right) \geq 2 \right]   
 \preceq C^{\beta \zeta  - f(\beta) / (2\zeta)  }   ,
\alle  
with the implicit constant depending only on $  \zeta$ and $\gamma$. Sending $\beta \rta 4/\gamma^2$, we see that the right side can be bounded by $C^{\frac{4}{\gamma^2} \zeta - 1/\zeta + o_C(1)}$. 

By~\cite[Proposition~3.4 and Remark~3.9]{ghm-kpz}, since $\eta^{C^\zeta} \eqD \eta$ viewed as curves modulo monotone re-parameterization, and by scale invariance, for each $k \in [0 , \log_8 C^{1/\zeta}+1]_{\BB Z}$ the event condition~\ref{item-deg-segment-ball} in the definition of $ G_C$ at scale $8^{-k}$ has probability at least $1-o_C^\infty(C)$. By taking a union bound over a logarithmic number of scales and combining the resulting estimate with our above bounds for the probabilities of the first two conditions, we obtain~\eqref{eqn-deg-segment-event}.
\end{proof}

\begin{proof}[Proof of~\eqref{eqn-fixed-pt-deg-out} of Proposition~\ref{prop-fixed-pt-deg}]
Fix $\zeta \in (0,1)$ which we will eventually send to 0. 
Suppose $C > 1$ and the event $G_C$ of Lemma~\ref{lem-deg-segment-event} occurs. 

Let $N' = \op{deg}_{\op{out}}(H_0)$ and let $\{[a_j ,b_j]\}_{j\in [1,N']_{\BB Z}}$ be a collection of intervals as in the definition of $\op{deg}_{\op{out}}(H_0)$, so that the interiors $(a_j , b_j)$ are disjoint, each $\eta([a_j,b_j]) $ is contained in the interior of $\ol B_0 $, and for each $j\in [1,N']_{\BB Z}$ one of $\eta(a_j)$ or $\eta(b_j)$ is contained in $\bdy \ol B_0 $ and the other is contained in $\eta([ -1,1])$. 

Let $k$ be the largest integer for which 
$\eta^{C^\zeta }\left(\left[- 1  ,  1       \right]\right) \subset B_{8^{-k}}(0)$. 
By the definition~\eqref{eqn-localize-set} of $\ol B_0 $, this implies that $R_{C^\zeta}^{-1}   \ol B_0  \subset B_{8^{-k+1}}(0)$. 
By conditions~\ref{item-deg-segment-upper} and~\ref{item-deg-segment-lower} in the definition of $ G_C$, we have $k \in [1 , \log_8 C^{1/\zeta}+1]_{\BB Z}$. 
Hence condition~\ref{item-deg-segment-ball} in the definition of $  G_C$ implies that for $j\in [1,N']_{\BB Z}$ the curve segment $ \eta^{C^\zeta }([a_j ,b_j])  $ contains a ball of radius at least $C^{-1/2} 8^{-k } $ which is itself contained in $B_{8^{-k+1}}(0)$.   
These balls for different choices of interval $[a,b]$ intersect only along their boundaries, so
\eqbn
\op{deg}_{\op{out}}(H_0) \leq 64 C \quad \op{on} \:   G_C .
\eqen 
Using the estimate for $\BB P[ G_C]$ from Lemma~\ref{lem-deg-segment-event}, we obtain that for $C > 1$,
\eqb \label{eqn-free-gff-deg-out}
\BB P\left[ \op{deg}_{\op{out}}(H_0) >  64 C \right] \leq C^{\frac{4}{\gamma^2} \zeta-1/\zeta + o_C(1)} .
\eqe 
Sending $\zeta \rta 0$, then multiplying this estimate $pC^{p-1}$ and integrating from 1 to $\infty$ with respect to $C$ shows that~\eqref{eqn-fixed-pt-deg-out} holds.
\end{proof}

\begin{proof}[Proof of Theorem~\ref{thm-moment}]
This follows from Propositions~\ref{prop-area-diam} and~\ref{prop-fixed-pt-deg} along with~\eqref{eqn-area-diam-localize-bound} and~\eqref{eqn-deg-localize-bound}.
\end{proof}

\section{Possible extensions}
\label{sec-extensions}

\noindent\textbf{Other discretizations of LQG.}
There is a variety of other natural random environments on $\BB C$ arising in the theory of LQG which are similar in spirit to the a priori embedding of the mated-CRT map but which do not admit analogous descriptions in terms of Brownian motion. 
For example, following~\cite[Section 1.4]{shef-kpz}, we could consider the adjacency graph of maximal dyadic squares with LQG mass at most 1. 
Alternatively, we could consider the adjacency graph of cells of space-filling SLE$_{\wt\kappa}$ parameterized by $\gamma$-LQG mass for $\wt\kappa \in (4,\infty)\setminus \{16/\gamma^2\}$.  

Theorem~\ref{thm-general-clt-uniform} together with arguments of the sort used in this paper can be used to show that the random walk in the above random environments converges to Brownian motion (i.e., Theorem~\ref{thm-rw-conv} is also true for such random environments).
To accomplish this, we would first consider the case of a 0-quantum cone (which makes hypothesis~\ref{item-hyp-resampling} true), then transfer to other GFF-type distributions using local absolute continuity. In each case the only work necessary beyond what is done in this paper is to verify the moment bound of hypothesis~\ref{item-hyp-moment}, which we expect can be done using similar arguments to the ones found in Section~\ref{sec-moment}. 
\medskip

\noindent\textbf{Other embeddings of the mated-CRT map.}
It may be possible to extend our techniques to show that one still has convergence to LQG for other embeddings of the mated-CRT map---such as Riemann uniformization, circle packing, or Smith embedding. The reason for this is that these other embeddings can be defined analogously to the Tutte embedding but using a different notion of ``discrete harmonic functions" on the mated-CRT map. Indeed, the statement that the Riemann uniformization embedding is close to the a priori embedding at large scales is equivalent to the statement that the Brownian motion on the mated-CRT map---defined by endowing each face with the conformal structure of a polygon---converges to ordinary Euclidean Brownian motion in the quenched sense. The circle packing can be seen as the Tutte embedding of a weighted version of the map with edge weights depending on the sizes of adjacent circles (see, e.g.,~\cite{dubejko-circle-packing}). See~\cite{gjn-macroscopic-circles} for some recent progress on the circle-packing of the mated-CRT map. 
The Smith embedding can be defined in terms of discrete harmonic functions on the map and its dual. 

In fact, the techniques of this paper would already allow us to treat the Smith embedding of the mated-CRT map if we knew the convergence of boundary-reflected random walk on the mated-CRT map and its dual under the a priori SLE/LQG embedding to boundary-reflected Brownian motion, i.e., if we had an analog of Theorem~\ref{thm-quenched-clt0} for reflected random walk. 
We now explain this in more detail. 

Let $\mcl G^{1/n}$ be the mated-CRT map with the disk topology as above, and let $\wh{\mcl G}^{ 1/n}$ be its planar dual. We do not include a vertex corresponding to the unbounded face, so that $\wh{\mcl G}^{ 1/n}$ is a planar map with boundary $\bdy \wh{\mcl G}^{1/n}$ consisting of the vertices of $\wh{\mcl G}^{1/n}$ whose corresponding faces have a vertex of $\bdy \mcl G^{1/n}$ on their boundaries. 

\begin{figure}[t!]
 \begin{center}
\includegraphics[scale=.65]{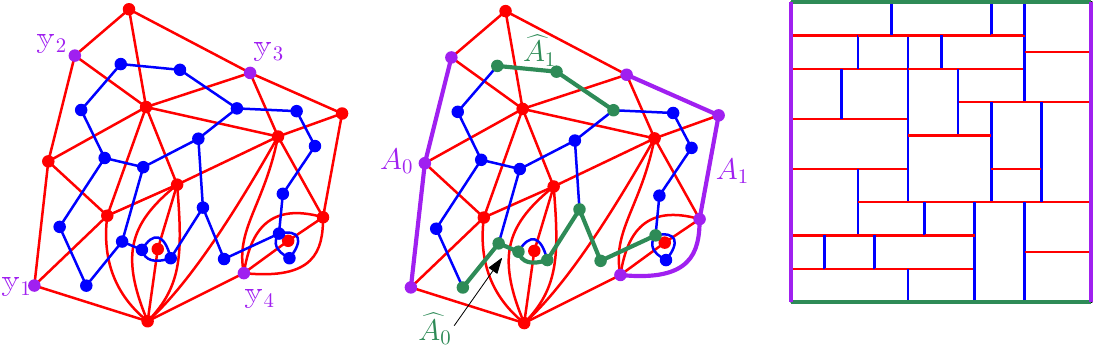}
\vspace{-0.01\textheight}
\caption{ \textbf{Left:} The mated-CRT map $\mcl G^{1/n}$ (red) and its dual $\wh{\mcl G}^{1/n}$ (blue).
\textbf{Middle:} The boundary arcs $A_0,A_1,\wh A_0,\wh A_1$ used in the definition of the Smith embedding.
\textbf{Right:} The Smith embedding of $\mcl G^{1/n}$. Each edge corresponds to a rectangle with horizontal (resp.\ vertical) coordinates determined by the hitting probabilities for simple random walk on $\mcl G^{1/n}$ (resp.\ $\wh{\mcl G}^{1/n}$) started from its endpoints. (This does not correspond to the actual Smith embedding of the map on the left). 
}\label{fig-smith}
\end{center}
\vspace{-1em}
\end{figure}

The \emph{Smith embedding} of $\mcl G^{1/n}$ associates a rectangle with each edge of $\mcl G^{1/n}$. See Figure~\ref{fig-smith} for an illustration.
To define this embedding, choose points $\BB y_1,\BB y_2,\BB y_3,\BB y_4\in\mcl V(\bdy\mcl G^{1/n})$, enumerated in counterclockwise order. Let $A_0$ (resp.\ $A_1$) be the counterclockwise boundary arc of $\bdy\mcl G^{1/n}$ from $\BB y_1$ to $\BB y_2$ (resp.\ $\BB y_3$ to $\BB y_4$). Also let $\wh A_0$ (resp.\ $\wh A_1$) be the largest arc of $\bdy \wh{\mcl G}^{1/n}$ with the property that none of its vertices correspond to faces of $\mcl G^{1/n}$ which share an edge with the counterclockwise arc of $\bdy\wh{\mcl G}^{1/n}$ from $\BB y_4$ to $\BB y_1$ (resp.\ $\BB y_2$ to $\BB y_3$). 

For an edge $e = \{x_1,x_2\} \in \mcl E\mcl G^{1/n}$, let $\wh e = \{\wh x_1 ,\wh x_2\} \in \wh{\mcl G}^{1/n}$ be the dual edge which crosses it.
For $i\in \{1,2\}$, let $\wh p_{x_i}$ (resp.\ $\wh p_{\wh x_i}$) be the probability that simple random walk on $\mcl G^{1/n}$ (resp.\ $\wh{\mcl G}^{1/n}$) started from $x_i$ (resp.\ $\wh x_i$) and reflected off $\bdy \mcl G^{1/n} \setminus (A_0 \cup A_1)$ (resp.\ $\bdy \wh{\mcl G}^{1/n} \setminus (\wh A_0 \cup \wh A_1)$) hits $A_1$ before $A_0$ (resp.\ $\wh A_1$ before $\wh A_0$). The rectangle associated with $e$ is given by
\eqbn
\left[ p_{x_1} \wedge p_{x_2}, p_{x_1} \wedge p_{x_2} \right] \times \left[ \wh p_{\wh x_1} \wedge \wh p_{\wh x_2}, \wh p_{\wh x_1} \wedge \wh p_{\wh x_2} \right] .
\eqen
One can check that under this embedding, two rectangles intersect along their horizontal (resp.\ vertical) boundaries if and only if the corresponding primal (resp.\ dual) edges share an endpoint.  

It is not hard to show using~\cite[Theorem 1.11]{gms-random-walk} and the estimates of Section~\ref{sec-moment} that under the a priori SLE/LQG embedding, the random walk on the \emph{dual} of the mated-CRT map converges to Brownian motion modulo time parametrization, in the quenched sense. 
Hence all that is needed to prove a version of Theorem~\ref{thm-tutte-conv0} for the Smith embedding is to show that the boundary-reflected random walks on the mated-CRT map and its dual converge to Brownian motion modulo time parametrization.

\medskip

\noindent\textbf{Other random planar map models.}  The results of this paper suggest a possible strategy for showing that the Tutte embeddings of other random planar maps---such as uniform triangulations, spanning tree-weighted planar maps, bipolar-orientation-decorated planar maps, etc.---converge to Liouville quantum gravity, although we emphasize that this is more speculative than the other potential extensions discussed here.

The above random planar maps $M$ are special since they can be bijectively encoded by means of a random walk on $\BB Z^2$ with certain i.i.d.\ increments~\cite{mullin-maps,bernardi-maps,shef-burger,bernardi-dfs-bijection,kmsw-bipolar}. We can use the two-dimensional variant of the KMT strong coupling theorem~\cite{kmt,zaitsev-kmt} to couple this encoding walk with the Brownian motion used to generate the mated-CRT map in such a way that the processes differ by at most $O(\log n)$ on the time interval $[-n,n]$. 
This gives us an a priori embedding of $M$ into $\BB C$ by sending the vertex corresponding to each step of the encoding walk to the space-filling SLE cell associated with the corresponding Brownian motion increment. 

One could then attempt to apply Theorem~\ref{thm-general-clt-uniform} or some variant thereof to the above embedding of $M$. 
One of the most significant (although not the only) difficulties with this approach is checking the finite expectation condition of hypothesis~\ref{item-hyp-moment}. This could potentially come from a fine analysis of the construction of the KMT coupling (to get up-to-constants, rather than $O(\log n)$, bounds at typical vertices) as well as arguments to show that the coupling error is not too strongly correlated with the Euclidean size of the SLE cells.

The above strong coupling approach is used in~\cite{ghs-map-dist,dg-lqg-dim,gm-spec-dim,gh-displacement} to deduce estimates for $M$ from estimates for the mated-CRT map, but the estimates in these works are less precise than the ones which would be needed to apply Theorem~\ref{thm-general-clt-uniform}. 
\medskip

\appendix

\section{Ergodicity for the 0-quantum cone}
\label{sec-ergodicity}

In this appendix we will prove the following statement, which is needed for the proof of ergodicity modulo scaling in Proposition~\ref{prop-0cone}. 

\begin{lem} \label{lem-cone-trivial}
Let $h$ be an embedding of the 0-quantum cone. Let $  F = F(h)$ be a bounded,  non-negative real-valued, measurable functional of $h$ which is scale and translation invariant, in the sense that, a.s., 
\eqb \label{eqn-cone-trivial-scale}
F(h(C \cdot + z)  + Q\log C ) = F(h) ,\quad \forall C > 0, \quad \forall z \in \BB C .
\eqe
Then $F$ is a.s.\ equal to a deterministic constant.  
\end{lem}

\begin{remark} \label{remark-gff-trivial}
The analog of Lemma~\ref{lem-cone-trivial} with $h$ replaced by a whole-plane GFF is essentially trivial. Indeed, if $h$ is the whole-plane GFF then we can write $h = \sum_{j=1}^\infty \xi_j \phi_j$ where the $\xi_j$'s are i.i.d.\ standard Gaussians and $\{\phi_j\}_{j\in\BB N}$ is an orthonormal basis for the Hilbert space used to define the whole-plane GFF, with each $\phi_j$ compactly supported. The whole-plane GFF analog of Lemma~\ref{lem-cone-trivial} then follows from the usual zero-one law for translation invariant events. However, there is not an obvious way to write an embedding of the 0-quantum cone as the sum of i.i.d.\ functions with compact supports, so a more complicated argument is necessary in this case. 
\end{remark} 

Before we discuss the proof of Lemma~\ref{lem-cone-trivial}, let us first record another tail triviality statement which is much easier to prove. 

\begin{lem} \label{lem-cone-tail}
The tail $\sigma$-algebra $\bigcap_{r>0} \sigma(h|_{\BB C\setminus B_r(0)})$ is trivial.
\end{lem}
\begin{proof} 
Let $R_1' := \inf\{r > 0 : h_r(0) + Q\log r = 0\}$. Note that $R_1'$ is \emph{not} the same as $R_1$ from~\eqref{eqn-mass-hit-time} since $R_1'$ is defined using an inf instead of a sup. The definition of the 0-quantum cone shows that if we condition on $R_1'$, the conditional law of $h|_{\BB C \setminus B_{R_1'}(0)}$ agrees in law with the corresponding restriction of a whole-plane GFF, normalized so that its circle average over $\bdy B_{R_1'}(0)$ is 0. 
The desired tail triviality for the 0-quantum cone therefore follows from the analogous statement for the whole-plane GFF; see, e.g.,~\cite[Lemma 2.2]{hs-euclidean}. 
\end{proof}

We now want to deduce Lemma~\ref{lem-cone-trivial} from Lemma~\ref{lem-cone-tail}. The idea of the proof is as follows.
For $z\in\BB C$, let $T_b(z) > 0$ be chosen so that the $\mu_h$-mass of the Euclidean ball $ B_{T_b(z)}(z) $ is $b$. 
By the martingale convergence theorem, for any fixed $z\in\BB C$, the conditional expectations $\BB E[F \,|\, h|_{B_{T_b(z)}(z) }]$ converge to $F$ as $b \rta \infty$. 
We want to use this to approximate $F$ by the average of $\BB E[F \,|\, h|_{B_{T_b(z)}(z) }]$ over all points $z$ in a large region. 
If $r$ is fixed, then when the region is large, the contribution to the average coming from points $z$ such that $B_{T_b(z)}(z) \cap B_r(0) \not=\emptyset$ is negligible. 
Hence if we can establish such an approximation, then we will have shown that $F \in \bigcap_{r>0} \sigma(h|_{\BB C\setminus B_r(0)})$, so we can conclude via Lemma~\ref{lem-cone-tail}. 

The main technical difficulty with this argument is that, a priori, the rate of convergence of $\BB E[F \,|\, h|_{B_{T_b(z)}(z) }]$ to $F$ as $b \rta \infty$ depends on $z$. So, it could be that when $|z|$ is large, we need to make $b = b(z)$ huge in order to say that $\BB E[F \,|\, h|_{B_{T_b(z)}(z) }]$ is close to $F$. In fact, it could be that we need to make $b$ so large that typically $T_b(z) > |z|$, which would defeat the above argument. 
To circumvent this difficulty, we will produce arbitrarily large random squares $S$ such that if $z$ is sampled uniformly from Lebesgue measure on $S$, then $(\BB C  ,h , z ,\infty)$ is again a 0-quantum cone. This allows us to choose the \emph{same} value of $b$ for every point $z$ when we take our averages. 

To produce these large random squares, we will use the concept of a \emph{dyadic system}, which leads to a convenient scale/translation invariant way of decomposing space into ``blocks". This concept was originally introduced in~\cite[Section 2.1]{gms-random-walk}. We will now review the definition.
Let $S \subset \BB C$ be a square (not necessarily dyadic). We write $|S|$ for the side length of $S$. A \emph{dyadic child} of $S$ is one of the four squares $S' \subset S$ (with $|S'| = |S|/2$) whose corners include one corner of $S$ and the center of $S$. A \emph{dyadic parent} of $S$ is one of the four squares with $S$ as a dyadic child. Note that each square has 4 dyadic parents and 4 dyadic children. A \emph{dyadic descendant} (resp.\ \emph{dyadic ancestor}) of $S$ is a square which can be obtained from $S$ by iteratively choosing dyadic children (resp.\ parents) finitely many times. 

A \emph{dyadic system} is a collection $\mcl D$ of closed squares (not necessarily dyadic) with the following properties.
\begin{enumerate}
\item If $S\in\mcl D$, then each of the four dyadic children of $S$ is in $\mcl D$.
\item If $S\in\mcl D$, then exactly one of the dyadic parents of $S$ is in $\mcl D$. 
\item Any two squares in $\mcl D$ have a common dyadic ancestor. 
\end{enumerate} 
The set of all side lengths of the squares in a dyadic system is precisely $\{2^{s+k}\}_{k\in\BB Z}$ for some $s\in [0,1]$ (determined by the system). 
If $\mcl D$ is a dyadic system, then for any $z\in \BB C$ there is a bi-infinite sequence of squares in $\mcl D$ which contain $z$. If $z$ does not lie on the boundary of a square in $\mcl D$, then this sequence is unique up to translation of the indices. If we are given any single square $S\in\mcl D$ and all of its dyadic ancestors, then $\mcl D$ is uniquely determined: $\mcl D$ is the set of all dyadic descendants of $S$ and its dyadic ancestors. This in particular allows us to define a topology on the space of dyadic systems, e.g., by looking at the local Hausdorff distance on the union of the origin-containing squares.

A \emph{uniform dyadic system} is the random dyadic system defined as follows. Let $s$ be sampled uniformly from $[0,1]$ and, conditional on $s$, let $w$ be sampled uniformly from $[0,2^s] \times [0,2^s]$. Set $S_0 := [0,2^s]\times [0,2^s] - w$. For $k \in \BB N$, inductively let $S_k$ be sampled uniformly from the four dyadic parents of $S_{k-1}$. Then let $\mcl D$ be the set of all dyadic descendants of $S_k$ for each $k\in\BB N$.

Throughout the rest of this subsection, we let $\mcl D$ be a uniform dyadic system sampled independently from $h$.
For $m > 0$ we define
\eqb \label{eqn-lqg-square}
\wh S_m := \text{largest dyadic square $S$ in $\mcl D$ containing 0 for which $\mu_h(S) \leq m$.} 
\eqe
The following lemma is the reason for introducing the uniform dyadic system in our proof. 
 
\begin{lem} \label{lem-cone-dyadic}
For $m > 0$, let $\wh S_m$ be as in~\eqref{eqn-lqg-square} and let $z_m$ be sampled uniformly from Lebesgue measure on $\wh S_m$, normalized to be a probability measure. 
Then $(\BB C , h , z_m , \infty)$ agrees in law with $(\BB C ,h , 0, \infty)$ as a quantum surface, i.e., there is a random constant $C_m > 0$ such that $h(C_m \cdot + z_m)  +Q\log C_m \eqD h$. 
\end{lem}
\begin{proof}
The proof is similar to the argument that condition 1 implies condition 5 in~\cite[Appendix A]{gms-random-walk}. 
Let $\rng h$ be a zero-boundary GFF on $\BB D$ sampled independently from $\mcl D$ and let $w$ be sampled uniformly from Lebesgue measure on $\BB D$, independently from $h$ and $\mcl D$. 
By the proof of~\cite[Proposition 4.13(ii)]{wedges}, the quantum surfaces $(\BB D , \rng h   +A , w)$ converge in the local total variation sense to $(\BB C , h , 0, \infty)$ as $A \rta \infty$. That is, for each $A >1$, there exists a random $c_A > 0$ such that for each fixed $r > 0$, the total variation distance between the laws of $(\rng h(c_A\cdot + w)   + Q\log c_A  +A )|_{B_r(0)}$ and $h|_{B_r(0)}$ tends to zero as $A\rta\infty$. 

For $A >1$ and $m  >0$, also let $\wh S_m^A(w)$ be the largest dyadic square $S$ in $\mcl D$ containing $w$ such that $\mu_{\rng h + A}(S) \leq m$. 
By~\cite[Lemma 2.1]{gms-random-walk}, $c_A^{-1}(\mcl D - w)$ is a uniform dyadic system independent from $\rng h$. 
Due to the LQG coordinate change formula, the square $\wh S_m^A(w)$ is determined by the pair $(\rng h(c_A\cdot + w)   + Q\log c_A  +A , c_A^{-1}(\mcl D - w) )$ in the same manner that $\wh S_m$ is determined by $(h , \mcl D)$. From the above local total variation convergence, it therefore follows that for each fixed $r >0$, 
\eqb \label{eqn-cone-dyadic-conv}
\left( (\rng h(c_A\cdot + w)   + Q\log c_A  +A)|_{B_r(0)} ,  \wh S_m^A(w) \right) 
\rta \left( h|_{B_r(0)} ,  \wh S_m\right) 
\eqe
in the total variation sense as $A\rta\infty$.

By~\eqref{eqn-cone-dyadic-conv}, it holds with probability tending to 1 as $A\rta\infty $ ($m$ fixed) that $\wh S_m^A(w)$ is contained in $\BB D$. 
Since $w$ is sampled uniformly from Lebesgue measure on $\BB D$, independently from $\mcl D$, it follows that the total variation distance between the conditional law of $w$ given $\left( (\rng h(c_A\cdot + w)   + Q\log c_A  +A)  ,  \wh S_m^A(w) \right) $ and the uniform measure on $\wh S_m^A(w)$ tends to zero as $A\rta\infty$. 
By combining this with~\eqref{eqn-cone-dyadic-conv}, we obtain the lemma statement.
\end{proof}

Using Lemma~\ref{lem-cone-dyadic}, we obtain the following ergodic averaging property for functionals of $h$.

\begin{lem} \label{lem-cone-ergodic}
Let $G = G(h)$ be a bounded measurable functional which is scale invariant in the sense that a.s.\ $G(h(C\cdot)  +Q\log C)) = G(h)$ for all $C>0$, equivalently $G$ is a.s.\ determined by the quantum surface $(\BB C ,h , 0,\infty)$.  
Almost surely, the limit
\eqb \label{eqn-cone-ergodic}
\lim_{m\rta\infty} \frac{1}{|\wh S_m|^2} \int_{\wh S_m} G(h(\cdot-z)) \,dz 
\eqe
exists.
\end{lem}
\begin{proof}
The proof is essentially identical to that of~\cite[Lemma 2.6]{gms-random-walk}, but we will give the proof for completeness. 
Let $\mcl F_m$ be the $\sigma$-algebra generated by the measurable functions $H = H(h,\mcl D)$ of $(h,\mcl D)$ which satisfy 
\eqb \label{eqn-dyadic-sigma-algebra}
H\left(h(C \cdot - z) + Q\log C  , C^{-1}(\mcl D-z) \right) = H( h ,   \mcl D) ,\quad \forall z\in \wh S_m ,\quad \forall C >0 .
\eqe
Then $\mcl F_m$ encodes $(h,\mcl D)$, viewed modulo translation within $\wh S_m$ and scaling, but not the location of the origin within the square $\wh S_m$.  
By Lemma~\ref{lem-cone-dyadic}, 
\eqb
\BB E\left[ G \,|\, \mcl F_m \right] = \frac{1}{|\wh S_m|^2} \int_{\wh S_m} G(h(\cdot-z)) \,dz .
\eqe
Since $\{\mcl F_m\}_{m > 0}$ is a decreasing family of $\sigma$-algebras, we can apply the backward martingale convergence theorem to get that the limit in~\eqref{eqn-cone-ergodic} exists a.s.\ and equals $\BB E\left[ G \,|\, \bigcap_{m>0}\mcl F_m \right]$. 
\end{proof}

\begin{proof}[Proof of Lemma~\ref{lem-cone-trivial}]
Throughout the proof we assume without loss of generality that the embedding $h$ is chosen so that $\mu_h(B_1(0)) =1$. 
For $b > 0$ and $z\in\BB C$, let
\eqb \label{eqn-inf-radius-def}
T_b(z) := \inf\left\{ r > 0  : \mu_h(B_r(0)) = b \right\} .
\eqe
We abbreviate $T_b = T_b(0)$, and note that our choice of embedding is equivalent to requiring $T_1=1$.   
We also let 
\eqb
\mcl S_b(z) := \left(B_{T_b}(z) , h|_{B_{T_b(z)}(z)}, z \right)  
\eqe
be the quantum surface obtained by restricting $h$ to $B_{T_b(z)}(z)$, and we abbreviate $\mcl S_b = \mcl S_b(0)$. 
\medskip

\noindent\textit{Step 1: approximating $F$ by $\BB E[F \,|\,\mcl S_b]$.}
Since $T_b \rta\infty$ a.s.\ as $b\rta\infty$, we have $\bigvee_{b > 1} \sigma(\mcl S_b) = \sigma(\BB C ,h , 0, \infty)$ up to events of probability zero.
By the scaling property of $F$ (i.e.,~\eqref{eqn-cone-trivial-scale} applied with $z =0$), $F$ is a functional of the quantum surface $(\BB C ,h , 0,\infty)$. 
By the martingale convergence theorem, a.s.\ $\BB E[F \,|\, \mcl S_b ] \rta F$ as $b\rta\infty$.
So, for each $\ep > 0$ we can find $b = b(\ep) > 1$ such that 
\eqb \label{eqn-cond-approx}
\BB E\left[ \left| F - \BB E[F \,|\,\mcl S_b]   \right|   \right] \leq \ep .
\eqe
\medskip

\noindent\textit{Step 2: approximating $F$ by $\BB E[F \,|\,\mcl S_b(z)]$ for a random $z$.} 
As in Lemma~\ref{lem-cone-dyadic}, for $m>0$ let $z_m$ be sampled from Lebesgue measure on the square $\wh S_m \in \mcl D$.
By Lemma~\ref{lem-cone-dyadic}, $(\BB C , h , z_m , \infty)$ is a 0-quantum cone. 
Hence, $(h(T_1(z_m) \cdot + z_m) + Q\log T_1(z_m) , \mcl S_b(z_m)  ) \eqD (h , \mcl S_b)$. 
By~\eqref{eqn-cone-trivial-scale}, $F(h(T_1(z_m) \cdot + z_m) + Q\log T_1(z_m) ) = F$. By~\eqref{eqn-cond-approx}, 
\eqb \label{eqn-cond-approx-k}
\BB E\left[ \left| F - \BB E[F \,|\, \mcl S_b(z_m) ]   \right| \right]  \leq \ep .
\eqe
\medskip

\noindent\textit{Step 3: approximating $F$ by a limit of averages over squares.} 
By averaging~\eqref{eqn-cond-approx-k} w.r.t.\ the law of $z_m$ (i.e., taking the conditional expectation given $(h,\mcl D)$), we get
\allb \label{eqn-cond-approx-avg}
\ep 
&\geq \BB E\left[ \BB E\left( \left| F - \BB E[F \,|\, \mcl S_b(z_m) ]   \right| \,|\, (h , \mcl D)  \right) \right] \notag\\ 
&= \BB E\left[ \frac{1}{|\wh S_m|^2} \int_{\wh S_m} \left| F - \BB E[F \,|\, \mcl S_b(z) ]   \right| \,dz  \right] \notag\\ 
&\geq \BB E\left[   \left| F -  \frac{1}{|\wh S_m|^2}\int_{\wh S_m}  \BB E[F \,|\, \mcl S_b(z) ] \,dz  \right|   \right]  ,
\alle
where $|\wh S_m|$ denotes the side length of $\wh S_m$. By Lemma~\ref{lem-cone-ergodic}, the limit 
\eqbn
\wt F := \lim_{m\rta\infty}  \frac{1}{|\wh S_m|^2}\int_{\wh S_m}  \BB E[F \,|\, \mcl S_b(z) ] \,dz
\eqen
exists a.s. By~\eqref{eqn-cond-approx-avg} and the dominated convergence theorem (recall that $F$ is bounded), 
\eqb \label{eqn-cond-approx-lim}
\BB E[|F - \wt F|] \leq \ep . 
\eqe
\medskip

\noindent\textit{Step 4: $\wt F $ is $\sigma(\mcl D)$-measurable.} 
We will now argue that $\wt F \in  \sigma(\mcl D) \vee \bigcap_{r>0} \sigma(h|_{\BB C\setminus B_r(0)}) $, so that by Lemma~\ref{lem-cone-tail} $\wt F$ is a.s.\ determined by $\mcl D$. To see this, let $\{S_k\}_{k\in\BB Z}$ be an enumeration of the origin-containing squares of $\mcl D$, in increasing order, which is chosen in a manner depending only on $\mcl D$ (e.g., we could fix the ordering by requiring that $S_0$ is the smallest square with Euclidean side length $\geq 1$). 
Then each $\wh S_m$ is one of the $S_k$'s and each $S_k$ is one of the $\wh S_m$'s, so $\wt F = \lim_{k \rta\infty}  \frac{1}{|S_k|^2}\int_{S_k}  \BB E[F \,|\, \mcl S_b(z) ] \,dz$. Using basic estimates for $\mu_h$, one easily checks that for each fixed $r>0$, the set of $z\in \BB C$ for which $B_{T_b(z)}(z) \cap B_r(0) \not=\emptyset$ has finite Lebesgue measure. This yields the formula
\eqb
 \wt F = \lim_{k \rta\infty}  \frac{1}{|S_k|^2}\int_{S_k}  \BB E[F \,|\, \mcl S_b(z) ] \BB 1_{(  B_{T_b(z)}(z) \cap B_r(0)  =\emptyset   )} \,dz .
\eqe
Therefore, $\wt F \in \sigma(\mcl D) \vee \sigma(h|_{\BB C\setminus B_r(0)})$. This holds for every $r$ and $\mcl D$ is independent from $h$, so by Lemma~\ref{lem-cone-tail}, $\wt F $ agrees a.s.\ with a $\sigma(\mcl D)$-measurable random variable. 

We know that $F$ is a functional of $h$, which is independent from $\mcl D$. By~\eqref{eqn-cond-approx-lim}, $F$ can be approximated arbitrarily closely by a random variable in $\sigma(\mcl D)$. Therefore, $F$ is a.s.\ equal to a deterministic constant.  
\end{proof}

\bibliography{cibiblong,cibib}

\bibliographystyle{hmralphaabbrv}

\end{document}